\documentclass{amsart}
\usepackage{fontenc}
\usepackage[english]{babel}
\usepackage[margin=1in]{geometry}

\usepackage{amssymb, amsmath,mathtools,mathabx}
\usepackage{mathrsfs,esint}
\usepackage{amscd}
\usepackage{verbatim}
\usepackage{stmaryrd}
\usepackage[dvipsnames]{xcolor}
\mathtoolsset{showonlyrefs}

\usepackage{enumerate}

\usepackage[colorlinks,linkcolor={blue},citecolor={blue},urlcolor={purple},]{hyperref}

\usepackage[colorinlistoftodos,prependcaption,textsize=tiny]{todonotes}

\numberwithin{equation}{section}

\usepackage{forest}
\usepackage{tikz}
\tikzset{>=latex}

\usepackage{tabularx}
\newcolumntype{L}{>{\arraybackslash}X}
\usepackage{multirow}

\theoremstyle{plain}
\newtheorem{theorem}{Theorem}[section]
\theoremstyle{remark}
\newtheorem{remark}[theorem]{Remark}

\theoremstyle{plain}
\newtheorem{corollary}[theorem]{Corollary}
\newtheorem{lemma}[theorem]{Lemma}
\newtheorem{proposition}[theorem]{Proposition}
\newtheorem{definition}[theorem]{Definition}

\newtheorem{assumption}[theorem]{Assumption}

\numberwithin{equation}{section}

\def\N{{\mathbb N}}
\def\Z{{\mathbb Z}}

\def\R{{\mathbb R}}

\newcommand{\E}{{\mathbb E}}
\renewcommand{\P}{{\mathbb P}}
\newcommand{\F}{{\mathscr F}}

\renewcommand{\H}{{\mathscr H}}

\newcommand{\g}{\gamma}
\newcommand{\om}{\omega}
\renewcommand{\O}{\Omega}

\renewcommand{\a}{\kappa}

\newcommand{\vp}{\varphi}

\newcommand{\Dom}{\mathcal{O}}

\newcommand{\I}{\mathcal{I}}
\newcommand{\J}{\mathcal{J}}
\newcommand{\Tor}{\mathbb{T}}
\newcommand{\T}{\mathbb{T}}

\newcommand{\A}{{\mathcal A}}

\newcommand{\loc}{{\rm loc}}

\newcommand{\calL}{{\mathscr L}}

\newcommand{\D}{\mathscr{D}}

\newcommand{\EE}{\mathscr{E}}

\newcommand{\LLL}{\mathscr{L}}

\DeclareMathOperator*{\esssup}{\textup{ess\,sup}}
\DeclareMathOperator*{\essinf}{\textup{ess\,inf}}

\newcommand{\tT}{T^{*}}

\newcommand{\wt}{\widetilde}
\newcommand{\rr}{\rrbracket}
\newcommand{\rParen}{\rrparenthesis}

\newcommand{\rro}{\rParen}

\usepackage{stmaryrd}

\newcommand{\MRtash}{\mathcal{SMR}_{p,\hat{\a}}^{\bullet}(t_0,T)}

\newcommand{\MRtzeroa}{\mathcal{SMR}_{p,\a}^{\bullet}(0,T)}

\newcommand{\MRtaszh}{\mathcal{SMR}_{p,\hat{\a}}(t_0,T)}

\newcommand{\MRttwoszero}{\mathcal{SMR}_{2,0}^{\bullet}(t_0,T)}

\newcommand{\MRttwosz}{\mathcal{SMR}_{2,0}(t_0,T)}

\newcommand{\Xap}{X^{\mathrm{Tr}}_{\a,p}}

\newcommand{\Xp}{X^{\mathrm{Tr}}_{p}}

\newcommand{\Xzp}{X^{\mathrm{Tr}}_{0,p}}

\newcommand{\one}{{{\bf 1}}}
\newcommand{\embed}{\hookrightarrow}

\renewcommand{\div}{{\rm div}}

\newcommand{\stoc}{\mathrm{sto}}
\newcommand{\Fou}{\mathcal{F}}

\newcommand{\norm}[1]{{\left\vert\kern-0.25ex\left\vert\kern-0.25ex\left\vert #1
	\right\vert\kern-0.25ex\right\vert\kern-0.25ex\right\vert}}

\renewcommand{\emptyset}{\varnothing}

\newcommand{\Progress}{\mathscr{P}}

\renewcommand{\S}{\mathscr{S}}

\def\XXint#1#2#3{{\setbox0=\hbox{$#1{#2#3}{\int}$ }
	\vcenter{\hbox{$#2#3$ }}\kern-.6\wd0}}

\newcommand{\vt}{\vartheta}
\newcommand{\avg}{\int_{\T} u_0\, \dd x}
\newcommand{\dd}{\mathrm{d}}

\renewcommand{\epsilon}{\varepsilon}

\renewcommand{\hat}{\widehat}
\renewcommand{\tilde}{\widetilde}

\newcommand{\pstar}{q_*}
\newcommand{\ppstar}{p_*}
\newcommand{\Qstar}{Q_*}
\newcommand{\Bstar}{B_*}
\newcommand{\Cstar}{C_*}
\newcommand{\Istar}{I_*}
\newcommand{\Iso}{\mathcal{J}}

\allowdisplaybreaks

\begin{document}

\author{Antonio Agresti}
\address{Delft Institute of Applied Mathematics\\
	Delft University of Technology \\ P.O. Box 5031\\ 2600 GA Delft\\The
	Netherlands.}
\curraddr{Department of Mathematics Guido Castelnuovo\\
	Sapienza University of Rome \\ P.le Aldo Moro 5\\ 00185 Rome\\ Italy.}
\email{antonio.agresti92@gmail.com}

\author{Max Sauerbrey}
\address{Delft Institute of Applied Mathematics\\
	Delft University of Technology \\ P.O. Box 5031\\ 2600 GA Delft\\The
	Netherlands.}
\curraddr{Max Planck Institute for Mathematics in the Sciences\\
	Inselstr. 22 \\ 04103 Leipzig \\ Germany.}
\email{maxsauerbrey97@gmail.com}

\subjclass[2020]{35R60, 76A20, 35K59} 

\thanks{The first author has received funding from the VICI subsidy VI.C.212.027 of the Netherlands Organisation for Scientific Research (NWO)}

\date\today

\title[Stochastic thin-film equation with an interface potential]{Well-posedness of the stochastic thin-film\\ equation with an interface potential}

\keywords{Thin-Film Equation, Conservative Noise, Regularity, Degenerate Equations, Stochastic Evolution Equations,  A-Priori Estimates, Stochastic Maximal Regularity.}

\begin{abstract}
		We consider strictly positive solutions to a class of fourth-order conservative quasilinear SPDEs on the $d$-dimensional torus modeled after the stochastic thin-film equation. We prove local Lipschitz estimates in Bessel potential spaces under minimal assumptions on the parameters and corresponding stochastic maximal $L^p$-regularity estimates for thin-film type operators with measurable in time coefficients. As a result, we deduce local well-posedness of the stochastic thin-film equation as well as blow-up criteria and instantaneous regularization for the solution. In dimension one, we additionally close $\alpha$-entropy estimates and subsequently an energy estimate for the stochastic thin-film equation with an interface potential so that global well-posedness follows. We allow for a wide range of mobility functions, including the power laws $u^n$ for $n\in [0,6)$ as long as the interface potential is sufficiently repulsive. 
\end{abstract}

\maketitle
\tableofcontents

\section{Introduction and statement of the main results}
We consider fourth-order quasilinear stochastic PDEs of the form 
\begin{align}\label{Eq101}
	\dd u\,+\,\div( m(u) \nabla \Delta u)\, \dd t \,=\, \div(\Phi(u)\nabla u )\, \dd t \,+\, \sum_{k\in \N}\div( g(u) \psi_k )\, \dd \beta^{(k)}, \qquad u(0)=u_0,
\end{align}
on the $d$-dimensional torus $\Tor^d$.
Throughout this article, we assume that 
\begin{equation}\label{eq:assumption_coeff}
m\colon (0,\infty)\to (0,\infty) \quad \text{ and }\quad g,\Phi\colon (0,\infty)\to \R \ \ \text{ are \emph{smooth functions}},
\end{equation}
and we only specify them for positive values of $u$, as we are interested in the situation where \eqref{Eq101} preserves positivity, and the initial value $u_0$ is a strictly positive function. Moreover, 
\begin{equation}\label{eq:assumption_psi}
	(\psi_k)_{k\in \N}\quad \text{ are \emph{ sufficiently regular vector fields} }\quad   \psi_k\colon \Tor^d\to\R^d \quad\text{ with \emph{sufficient  decay as $k\to\infty$},}\end{equation}
and $(\beta^{(k)})_{k\in \N}$ a corresponding family of independent standard Brownian motions. The precise formulation of the assumptions on the noise coefficients $(\psi_k)_{k\in \N}$ is given below.
We remark that the stochastic PDE \eqref{Eq101} is understood in the It\^o sense. 

The class of equations of the form \eqref{Eq101} contains, in particular, the \emph{stochastic thin-film equation} 
\begin{align}\label{Eq100}
	 \dd u\,=\, -\div( m(u) \nabla (\Delta u- \phi'(u)))\,\dd t\,+\, \div( m^{1/2}(u) \,\dd B ), 
\end{align} 
which results from setting $g(u) =  m^{1/2}(u)$, $\Phi(u)=m(u)\phi''(u)$ in \eqref{Eq101} and 
 $B= \sum_{k\in \N} \psi_k \beta^{(k)}$  
is a function-valued Brownian motion, whenever 
\begin{equation}\label{eq:assumption_STFE_coeff}
m\colon (0,\infty)\to (0,\infty) \quad \text{ and }\quad \phi\colon (0,\infty)\to \R \ \ \text{ are smooth and $ B_t $  is regular enough for each $t>0$.}
\end{equation}
Then \eqref{Eq101} becomes indeed the \emph{It\^o} equation \eqref{Eq100}, but we stress that under reasonable symmetry conditions on the driving process $B$ also the \emph{Stratonovich} version of \eqref{Eq100} can be cast into the form \eqref{Eq101} by adjusting the coefficient $\Phi$, see \cite[Remark 2.1]{metzger2022existence}, \cite[Appendix A]{KleinGruen22} and the comments below \eqref{Eq100_stra}.

The solution $u$ to \eqref{Eq100} models the height of a thin liquid film driven by surface tension, interaction forces between the liquid and the substrate, and thermal fluctuations of the fluid molecules. The coefficient $m(u)$ is then called the mobility function and reflects the boundary condition of the fluid velocity at the substrate. The choice $m(u) = u^3$ corresponds to the no-slip condition and $m(u) = u^{n}$ for $n\in [1,3)$ to various slip conditions, with $n=2$ corresponding to the Navier-slip condition. The resulting thin-film operator $-\div( m(u) \nabla \Delta u)$ captures the effects of surface tension on the film height. As derived independently in \cite{DMES2005,GruenMeckeRauscher2006}, the effects of thermal noise on the fluid velocity are modeled by the conservative noise term in \eqref{Eq100} for an $L^2(\T^d)$-cylindrical Brownian motion $B$. We remark that the smoothness conditions, which we impose later on in the manuscript on the sequence $(\psi_k)_{k\in \N}$, restrict our results to the case of spatially colored $B$. The second-order operator $\div(m(u)\nabla[ \phi'(u)])$ in \eqref{Eq100} accounts for the effect of interaction forces between molecules of the fluid and the substrate, where $\phi$ is the effective interface potential. A typical example is
\begin{equation}
\label{eq:prototype_example_potential}
\phi(u) = u^{-8}- u^{-2} +1, 
\end{equation}
corresponding to conjoining and disjoining van der Waals forces modeled by the \emph{$6$-$12$ Lennard-Jones potential}, see \cite[Eq. (2.48c)]{ODB97}.  

Our main results can be summarized as follows:
\begin{itemize}
\item Local well-posedness and blow-up criteria for \eqref{Eq101} in any space dimensions -- Theorem \ref{Thm_local}  and Proposition \ref{prop:blow_up_criteria}.
\item Global well-posedness for \eqref{Eq100} with a repulsive potential in one dimension -- Theorems \ref{Thm_global} and \ref{cor_global_stratonovich} (It\^o $\&$ Stratonovich noise).
\end{itemize}

Let us stress that, since the derivations of the stochastic thin-film equation in \cite{DMES2005} and \cite{GruenMeckeRauscher2006}, our results are the first ones on the \emph{well-posedness} of \eqref{Eq100}. 
Indeed, only existence and \emph{no} uniqueness results for solutions to \eqref{Eq100} are available at the moment. 
In particular, at least if the interface potential is sufficiently singular near $0$ as in \cite{fischer_gruen_2018}, we can prove the uniqueness of global solutions, as conjectured in \cite[Section 6]{fischer_gruen_2018} for \eqref{Eq100} with $m(u)=u^2$.
Interestingly, our results cover the potential \eqref{eq:prototype_example_potential} even in the presence of a non-quadratic mobility, c.f., Assumption \ref{Assumptions_m}.
Of course, the main obstacle in obtaining well-posedness for the stochastic thin-film equation is the degeneracy of the leading order operator. 
However, in the presence of a repulsive potential, the solutions are strictly positive at all times, and therefore the thin-film operator remains parabolic a posteriori. As a consequence, pathwise uniqueness is amenable to be proven.

Further novelties of our approach are:
\begin{itemize}
\item Global well-posedness for various mobility functions -- see Assumption \ref{Assumptions_m}.
\item Reduced regularity of the initial data -- Theorem \ref{Thm_local} and Remark \ref{Rem2}.
\item Instantaneous high-order regularization -- Proposition \ref{Prop_regularization}.
\end{itemize}

The global well-posedness results of Theorems \ref{Thm_global} and \ref{cor_global_stratonovich} hold for a wide range of mobility functions, including power laws of the form $m(u)=u^n$ with $n\in [0,6)$, but also $m(u)=u^3+\lambda^{3-n} u^n$ for some $\lambda\geq 0$. While the former is often assumed in the mathematical literature on \eqref{Eq100}, it usually serves as a simplification of the latter mobility function obtained from the lubrication approximation, see \cite[Section II B]{ODB97}. Moreover, the literature on martingale solutions to \eqref{Eq100}, which we review in Subsection \ref{ss:review_literature},  imposes more restrictive assumptions on the exponent $n$. Since our global well-posedness result relies on the positivity-preserving mechanism of the effective interface potential $\phi$, it comes at the expense of excluding the interesting case in which a contact line, i.e., a triple junction of liquid-solid and gas, is present. Nevertheless, if one is interested in the situation of a non-fully supported fluid film, Theorems \ref{Thm_global} and \ref{cor_global_stratonovich} can be useful to construct solutions by taking $\phi\searrow 0$, as performed successfully in \cite{KleinGruen22} for $m(u)=u^2$. This limit is of particular interest for future research since positive approximations of the thin-film equation are often compatible with formal a-priori estimates of the equation. Additionally, the uniqueness part of Theorems \ref{Thm_global} and \ref{cor_global_stratonovich} has also numerical implications -- for example in the case $m(u)=u^2$ in which a subsequence of a finite difference discretization of \eqref{Eq100} was shown to converge in law to a solution in \cite[Theorem 3.2]{fischer_gruen_2018}. Indeed, the pathwise uniqueness implied by the Gy\"ongy-Krylov lemma \cite[Theorem 2.10.3]{Breit_Feireisl_Hofmanova} ensures that the finite difference scheme converges in probability to the unique solution to \eqref{Eq100} on the original probability space, at least for a subsequence. Since this solution is unique, the convergence also holds for the full sequence of approximations.

Concerning the regularity of the initial data $u_0$, in all dimensions, we can allow $u_0\in H^{1/2+\varepsilon,q}(\T^d)$ with $\varepsilon>0$ arbitrary and $q\ge 2$ large, and thus below the energy level $H^1(\T^d)$.
Moreover, in the case $d=1$, we can choose $u_0\in H^{1/2+\varepsilon}(\T)$ for $\varepsilon>0$.
Proposition \ref{Prop_regularization} shows that the regularity of the initial data only affects the regularity of $u$ at times $t\sim 0$, while for $t>0$ the solutions become smooth. More precisely, if $(\psi)_{k\in\N}$ are regular enough, then $u$ becomes smooth in space regardless the regularity of $u_0$: 
\begin{equation}
\label{eq:smoothness_classical_intro}
u\in C^{\theta,\infty}_{{\rm loc}}((0,\sigma)\times\T^d) \,\text{ a.s.\ 
for all $\theta \in [0,\tfrac{1}{2})$, where $\sigma$ is the explosion time of $u$.}
\end{equation}
Let us remark that $\sigma=\infty$ a.s.\ if $d=1$ and $\phi$ is sufficiently singular, cf., Theorems \ref{Thm_global} and \ref{cor_global_stratonovich}.

\smallskip
The proof of these results relies on the following three advances. Firstly, we show \emph{stochastic maximal regularity} estimates for thin-film type operators with strictly positive coefficients which depend only measurably on time. The latter is central in the derivation of suitable blow-up criteria for the quasilinear stochastic PDE \eqref{Eq101}. Secondly, we adapt the theory \cite{AV19_QSEE_1,AV19_QSEE_2} on quasilinear stochastic evolution equations to stochastic PDEs, which are \emph{a-priori} only degenerate parabolic. Thirdly, we estimate the \emph{energy production} of \eqref{Eq100} by the \emph{$\alpha$-entropy dissipation} for different $\alpha$ leading to new a-priori estimates for the stochastic thin-film equation with an interface potential. In particular, this allows us to deduce that the equation remains a.s.\ parabolic, \emph{a-posteriori}.
To keep this work as self-contained as possible, we rely only on local well-posedness results from \cite{AV19_QSEE_1, AV19_QSEE_2}, along with a relatively simple consequence in terms of blow-up criteria. 
For completeness, relevant background material is included in Appendix \ref{app:Sobolev_spaces}, whereas Appendix \ref{app:MR_Lp_time_measurable} provides a new and self-contained proof of maximal estimates for parabolic fourth-order operators, which play a key role in establishing stochastic maximal regularity of thin-film type operators.

\smallskip
The rest of this section is organized as follows. In Subsection \ref{ss:local_all_dimensions_intro} we discuss the local well-posedness of \eqref{Eq101} in all dimensions $d\geq 1$, while in Subsection \ref{ss:global_1d_intro} we state the global well-posedness result for \eqref{Eq100} in $d=1$. In Subsection \ref{ss:review_literature}, we provide further comments on the literature, and in Subsection \ref{ss:notation} we collect notation used throughout the manuscript.

\subsection{Local well-posedness, regularity and blow-up criteria in any dimension}
\label{ss:local_all_dimensions_intro}
For the local well-posedness of \eqref{Eq101}, we use the well-posedness theory for quasilinear stochastic evolution equations developed by Veraar and the first author in \cite{AV19_QSEE_1,AV19_QSEE_2}. 
 To consider \eqref{Eq101} as a quasilinear stochastic evolution equation, we introduce the operators
\begin{equation}\begin{aligned}\label{Eq8}
	A[u](f)\,=\, \div( m(u) \nabla \Delta f), \qquad
	F(u)\,=\,  \div(\Phi(u)\nabla u ),\qquad
	G_k(u)\,=\, \div( g(u) \psi_k ),
	\end{aligned}
\end{equation}
where the latter gives rise to the operator on $\ell^2(\N)$ defined by $G[u](e_k)  = G_k(u)$ for the $k$-th unit vector $e_k\in \ell^2(\N)$. If we introduce the cylindrical Brownian motion
\begin{align}
	W(t)\,=\, \sum_{k\in \N} e_k \beta_t^{(k)}
\end{align}
on $\ell^2(\N)$, \eqref{Eq101} takes the form of a quasilinear stochastic evolution equation  \cite[Eq. (1.1)]{AV19_QSEE_1}:
\begin{align}\label{Eq125}
	\dd u\,+\,  A[u](u)\,\dd t\,=\, F(u)\, \dd t \,+\, G[u]\,\dd W, \qquad u(0)=u_0.
\end{align}
We let the solution $u$ lie for almost all times in the \emph{Bessel potential space} $H^{s+2,q}(\Tor^d)$ while the deterministic and stochastic nonlinearities take values in the spaces $H^{s-2,q}(\Tor^d)$ and $H^{s,q}(\T^d;\ell^2(\N))$, respectively. The trajectory of the solution $u$ is continuous in the \emph{trace space} depending on the temporal integrability of $u$, which is described by the parameters $p$ and $\kappa$ and turns out to be the \emph{Besov space}  $B^{s+2-4\frac{1+\kappa}{p}}_{q,p}(\Tor^d)$. 
More precisely, we require $u$ itself and the nonlinearities in \eqref{Eq125} to be $p$-integrable in time with respect to the power weight $w_\kappa(t)=|t|^\kappa$. 
By the trace theory of anisotropic spaces, which we recall in  Appendix \ref{app:Sobolev_spaces}, this determines the above trace space as the optimal space for the initial value $u_0$.
We remark that the use of temporal weights plays a central role in the proof of high-order regularity and of blow-up criteria, see Propositions \ref{Prop_regularization} and \ref{prop:blow_up_criteria} below.

In what follows, we employ the following condition on the parameters $(p,\a,s,q)$. 

\begin{assumption}[Admissible parameters]\label{Assumptions_coefficients} The parameters $s\in (\frac{-1}{2},\infty)$, $p,q\in [2,\infty)$ and $\kappa\in [0,\infty)$ satisfy the following conditions:
	\begin{align}
		&
		\label{Eq1}
		p\in (2,\infty), \, \kappa \in \bigl[0,\tfrac{p}{2}-1\bigr)
		\quad \text{ or } \quad q=p=2, \, \kappa= 0,
		\\& 
		\label{Eq104}
		s\,+\,2\,-\,4\,\tfrac{1+\kappa}{p}\,-\,\tfrac{d}{q}\,>\, 0,
		\\&
		\label{Eq103}
		s\,+\,2\,-\,4\,\tfrac{1+\kappa}{p}\,>\, 1-s.
	\end{align}
\end{assumption}

\noindent
We also say that $(p,\a,s,q)$ are \emph{admissible parameters} if they satisfy Assumption \ref{Assumptions_coefficients}.
\smallskip

The restriction of the smoothness parameter $s>\frac{-1}{2}$ is made explicit in Assumption \ref{Assumptions_coefficients} because it is, anyway, implied by \eqref{Eq103}. We impose the condition $q\in [2,\infty)$ to make sure that the spaces $H^{s\pm2,q}(\T^d)$ are UMD-Banach spaces of type $2$, see \cite[Ex. 3.6.13, Prop. 4.2.15, Prop. 4.2.17 (1)]{Analysis1}. Together with $p\in [2,\infty)$ and \eqref{Eq1} this ensures that \cite[Assumption 3.1]{AV19_QSEE_1} is satisfied.
Condition \eqref{Eq104}, on the other hand, implies that the trace space $B^{s+2-4\frac{1+\kappa}{p}}_{q,p}(\Tor^d)$ embeds into a space of H\"older continuous functions by Sobolev embeddings. Therefore, we can control the oscillations of the coefficient $m(u)$, and the operator $A[u]$ behaves locally like the Bi-Laplacian, which is a key ingredient in proving stochastic maximal regularity of $A[u]$ required to apply \cite{AV19_QSEE_1,AV19_QSEE_2}. On the other hand, condition \eqref{Eq103} is imposed to give sense to the product 
\begin{equation}\label{Eq105}
	m(u)	\nabla \Delta f
\end{equation}
in the definition of $A[u]$ using Sobolev pointwise multipliers. Indeed, for $u$ in the trace space $B^{s+2-4\frac{1+\a}{p}}_{q,p}(\T^d)$ we expect smoothness $	s\,+\,2\,-\,4\,\frac{1+\kappa}{p}$ from $m(u)$ and $s-1$ from $	\nabla \Delta f$ where $f\in H^{s+2,q}(\Tor^d)$. If $s-1\ge 0$, there is no problem in giving sense to the pointwise multiplication \eqref{Eq105} and the condition \eqref{Eq103} is implied by \eqref{Eq104}. However if $s-1<0$, the function $\nabla \Delta f$ becomes a distribution and $m(u)$ must admit smoothness $1-s$ to define the product  $\eqref{Eq105}$, which is expressed in \eqref{Eq103}. 
\begin{remark}\label{Rem2}We investigate which choices of parameters are compatible with Assumption \ref{Assumptions_coefficients}. Firstly, we observe that for any $s\in (\frac{-1}{2},\infty)$ choosing $p,q\in [2,\infty)$ large and $\kappa=0$ guarantees that \eqref{Eq1}, \eqref{Eq104} and \eqref{Eq103} are satisfied. Therefore, for each $s\in (\frac{-1}{2},\infty)$ there is a feasible choice of parameters $(p,\kappa,q)$ subject to Assumption \ref{Assumptions_coefficients}. Secondly, we analyze how low we can choose the smoothness $s\,+\,2\,-\,4\,\frac{1+\kappa}{p}$ of the trace space, determining the roughness of initial values we can allow for \eqref{Eq125}. Condition \eqref{Eq1} implies that
\begin{equation}
	s\,+\,2\,-\,4\,\tfrac{1+\kappa}{p}\,\ge\, s,
\end{equation}
which together with \eqref{Eq103} yields
	\begin{equation}\label{Eq118}
	s\,+\,2\,-\,4\,\tfrac{1+\kappa}{p} \,>\, \tfrac{1}{2}.
\end{equation}
We convince ourselves that it is possible to choose admissible parameters $(p,\a,s,q)$ such that $s\,+\,2\,-\,4\,\tfrac{1+\kappa}{p}$ becomes arbitrarily close to $\frac{1}{2}$. If $d=1$,  we can choose simply $p=q=2$, $\kappa=0$ for any $s>\frac{1}{2}$, resulting in the trace space 
	$
	B_{2,2}^{s}(\Tor)\,=\, H^{s}(\Tor).
	$
	If however $d\ge 2$, we choose  $q>\frac{d}{s}$ for $s>\frac{1}{2}$, $p\in (2,\infty)$ and $\kappa$ close to $\frac{p}{2}-1$. Then the smoothness of the trace space is close to $s$, which can be chosen arbitrarily close to $\frac{1}{2}$.
\end{remark}

We can now define local solutions to \eqref{Eq101}. Recall that we are interested in the situation where the solution remains positive for all times due to the possible loss of parabolicity of $A[u]$ for non-positive $u$.

\begin{definition}[Local solution]\label{Defi_Lpk}
Let $(p,\a,s,q)$ be admissible parameters as in Assumption \ref{Assumptions_coefficients}. 
	Let $\sigma\colon \Omega \to [0,\infty]$ be a stopping time and $u\colon \llbracket 0,\sigma\rrparenthesis \to H^{s+2,q}(\Tor^d)$ be a progressively measurable process. 
	Then the tuple $(u,\sigma)$ is called a \emph{positive local $(p,\a,s,q)$-solution} to \eqref{Eq101}, if there exists a sequence of stopping times $(\sigma_{l})_{l \in \N}$ such that $0\le \sigma_l \nearrow \sigma$ and a.s.\ for all $l\in \N$ we have
	\begin{align}
	\label{Eq4}
		&u\,\in \, 
		L^p(0,\sigma_l, w_\kappa; H^{s+2,q}(\Tor^d)) \,\cap\, C([0,\sigma_l]; B^{s+2-4\frac{1+\kappa}{p}}_{q,p}(\Tor^d)),\\
		\label{Eq7} \noeqref{Eq7}
			&F(u)\in L^p(0,\sigma_l, w_\kappa; H^{s-2,q}(\Tor^d)),\quad G[u]\in L^p(0,\sigma_l, w_\kappa;H^{s,q}(\T^d;\ell^2(\N))),\\
			\label{Eq34}
		&\inf_{ [0,\sigma_l ]\times \Tor^d} u\,>\, 0  \ \ \  (\normalfont{\text{local positivity}}),
	\end{align}
	and a.s.\ for all $t\in [0,\sigma_{l}]$:
	\begin{align}
	\label{eq:stoch_integrated_eq}
		u(t)\,-\, u(0)\,+\,\int_0^t A[u(r)](u(r))\, \dd r \,=\,  \int_0^t F(u(r))\, \dd r
		\,+\, \int_0^t 
		G[u(r)]
		\, \dd W_r\,.
	\end{align}
	\end{definition}
Due to the  conditions \eqref{Eq4}--\eqref{Eq34} and $L^p(w_{\a})\embed L^2$ for $\a<\frac{p}{2}-1$, the deterministic and stochastic integrals in \eqref{eq:stoch_integrated_eq} are well-defined as $H^{s-2,q}$- and $H^{s,q}$-valued Bochner and It\^o integral, respectively (see, e.g., \cite[Theorem 4.7 and Proposition 5.3]{NVW13} for It\^o integration in type 2 spaces).
We remark that the term \emph{local} comes from the fact that the latter requirements \eqref{Eq4}--\eqref{Eq34} are demanded only away from the stopping time $\sigma$ 
and we sometimes refer to a sequence $(\sigma_l)_{l\in\N}$  as in Definition \ref{Defi_Lpk} as \emph{localizing sequence} for $(u,\sigma)$.
Finally, we define positive maximal unique $(p,\a,s,q)$-solutions.

\begin{definition}[Maximal unique positive solution]\label{Defi_max}A positive local $(p,\a,s,q)$-solution $(u,\sigma)$ to \eqref{Eq101} is called \emph{positive maximal unique $(p,\a,s,q)$-solution}, if for every positive local $(p,\a,s,q)$-solution $(v,\tau)$ to \eqref{Eq101}, one has $\tau\leq \sigma$ a.s.\ and $u=v$ a.s.\ on $[0, \tau)$. 
\end{definition} 

It remains to specify the regularity of the noise coefficients needed for the local well-posedness of \eqref{Eq101}.
In what follows, we use the parameters $s_{\psi}>-1/2$ and $q_{\psi}\in [2,\infty)$ to capture the smoothness of the noise. 

\renewcommand{\thetheorem}{\arabic{section}.\arabic{theorem}$(s_{\psi},q_{\psi})$}

\begin{assumption}[Noise regularity -- Local well-posedness]\label{Assumptions_noise_local}
There are parameters $s_{\psi}\in \R$ and $q_{\psi}\in [2,\infty)$, such that
	\begin{equation}\label{Eq140}
\biggl\|
\biggl(\sum_{k\in \N} \bigl| (1-\Delta)^{(1+s_{\psi})/2}\psi_k  \bigr|^2 \biggr)^{1/2}
\biggr\|_{L^{q_{\psi}}(\Tor^d)}\,<\, \infty.
	\end{equation}
\end{assumption}
\renewcommand{\thetheorem}{\arabic{section}.\arabic{theorem}}

The left-hand side of the above is the norm of $(\psi_k)_{k\in \N}$  in $ H^{1+s_\psi,q_\psi}(\T^d;\ell^2(\N;\R^d))$ and \eqref{Eq140} holds in particular if $(\psi_k)_{k\in \N}\in \ell^2(\N;H^{1+s_{\psi},q_{\psi}}(\T^d;\R^d))$, see \cite[Theorem 9.2.10]{Analysis2} and \eqref{eq:identification_gamma_norms} below.
For the local well-posedness of \eqref{Eq101} formulated below, we only need Assumption \ref{Assumptions_noise_local} for some $s_{\psi}>{-1}/{2}$ and $q_{\psi}$ large depending on $d$, see the discussion in Remark \ref{Rem2}. In particular, this includes less regular noise than assumed in the literature on (global in time) martingale solutions to \eqref{Eq100}. Indeed, the most general condition treated so far, see \cite[Assumption 1.2]{dareiotis2023solutions}, implies that Assumption \ref{Assumptions_noise_local} holds for $s_\psi=0$ and any $q_\psi<\infty$. A more restrictive condition on $(\psi_k)_{k\in \N}$  for the global well-posedness of \eqref{Eq100} with $d=1$ is given below in Assumption \ref{Assumptions_noise_global}. 

We are ready to state our first result on local well-posedness of \eqref{Eq101} in all dimensions $d\geq 1$.

\begin{theorem}[{Local well-posedness}]\label{Thm_local}
Let Assumptions \ref{Assumptions_coefficients} and \ref{Assumptions_noise_local} with $(s_{\psi},q_{\psi})=(s,q)$ be satisfied, and 
\begin{equation*}
 u_0 \colon\Omega  \to B^{s+2-4\frac{1+\kappa}{p}}_{q,p}(\Tor^d) \quad \text{ is strongly $\mathscr{F}_0$-measurable and satisfies }\quad 
 \inf_{\Tor^d} u_0>0 \text{ a.s. }
\end{equation*}
Then there exists a positive maximal unique $(p,\a,s,q)$-solution $(u,\sigma)$ to \eqref{Eq101} such that a.s.\ $\sigma>0$ and
\begin{align*}
u\in H^{\theta,p}_{\loc}([0,\sigma), w_\kappa;H^{s+2-4\theta,q}(\T^d)) \cap C((0,\sigma); B^{s+2-\frac{4}{p}}_{q,p}(\Tor^d))
\end{align*}
for all $\theta\in [0,\frac{1}{2})$, if  $p>2$. 
\end{theorem}

Next, we discuss how the regularity of the noise coefficients affects the regularity of solutions. Let us stress that the following result is independent of the regularity of the initial data $u_0$.

\begin{proposition}[Instantaneous regularization]\label{Prop_regularization}
	Let the assumptions of Theorem \ref{Thm_local} be satisfied, and $(u,\sigma)$ be the corresponding positive maximal unique $(p,\a,s,q)$-solution to \eqref{Eq101}. Assume that Assumption \ref{Assumptions_noise_local} holds for some $s_{\psi}\geq s$ and all $q_{\psi}\in [2,\infty)$. Then
		\begin{equation}\label{Eq35}
		u\,\in\, H_{\loc}^{\theta,r}(0,\sigma; H^{2+s_{\psi}-4\theta,\zeta}(\Tor^d))  \text{ a.s.\ for all }\theta\in[0,\tfrac{1}{2}),\ r,\zeta\in (2,\infty).
	\end{equation}
	In particular $u\,\in\,C_{\loc}^{\theta_1,\theta_2}((0,\sigma)\times \T^d)$ a.s.\ for all $\theta_1\in [0,\frac{1}{2})$ and $\theta_2\in (0,2+s_{\psi})$.
\end{proposition}

The above shows that, if Assumption \ref{Assumptions_noise_local} is satisfied for all $s_{\psi}>0$ and $q_{\psi}\in [2,\infty)$, then the solution $(u,\sigma)$ provided by Theorem \ref{Thm_local} is smooth in space as claimed in \eqref{eq:smoothness_classical_intro}. Moreover, the above regularization result for $s_{\psi}=1$ will play an important role in the study of global well-posedness as it allows us to justify the integrations by parts when working on intervals $(t_0,T)$ with $t_0>0$ even if $u_0\not\in H^1(\T^d)$. 
The above result also implies that solutions to \eqref{Eq101} are \emph{compatible} in the sense that if Theorem \ref{Thm_local} can be applied to different sets of parameters $(p_i,\a_i,s_i,q_i)$ for $i\in\{1,2\}$, the corresponding maximal unique $(p_i,\a_i,s_i,q_i)$-solution $(u_i,\sigma_i)$ coincides, i.e., $\sigma_1=\sigma_2$ a.s.\ and $u_1=u_2$ a.e.\ on $[0,\sigma_1)\times \O$.

Now, we turn to the question of how to determine whether $\sigma=\infty$ a.s.\ or not.
Usually, one needs to obtain a-priori estimates for solutions to the corresponding stochastic PDE in a sufficiently regular norm. In practice, the smoothness needed in the blow-up criteria reflects the one used for the initial data in local well-posedness results. In particular, the lower the regularity allowed from the initial data, the better the corresponding blow-up criteria. Since we allow for a leading order operator which degenerates near $0$, we also need to assume a positivity condition in the following blow-up criterion.

\begin{proposition}[Blow-up criteria]
\label{prop:blow_up_criteria}
Let the assumptions of Theorem \ref{Thm_local} be satisfied, and let $(u,\sigma)$ be the corresponding positive maximal unique $(p,\a,s,q)$-solution to \eqref{Eq101}. Assume that Assumption \ref{Assumptions_noise_local} holds for some $s_{\psi}\geq s$ and all $q_{\psi}\in [2,\infty)$. Moreover, let $(p_0,\a_0,s_0,q_0)$ be admissible exponents (cf., the comments below Assumption \ref{Assumptions_coefficients}) satisfying $s_0\leq s_{\psi}$ and set $\g_0 := s_0+2-4\frac{1+\kappa_0}{p_0}$. Then, for all $0<\varepsilon<T<\infty$,
\begin{equation}
\P \Bigl(\varepsilon<
	\sigma<T\,,\, \sup_{t\in [\varepsilon, \sigma)} \|u(t)\|_{ B^{\g_0}_{q_0,p_0}(\T^d)}<\infty\,,\, \inf_{ [\varepsilon,\sigma)\times \T^d}  u\,>\,0
	\Bigr)\,=\, 0 .
\end{equation}
\end{proposition}

The norm in the above blow-up criterion is well-defined even if $s_0\gg s$ by Proposition \ref{Prop_regularization}.
Moreover, due to \eqref{Eq104} in the admissibility condition, $B^{\g_0}_{q_0,p_0}\embed C^{\alpha_0}$ for some $\alpha_0>0$.  Finally, if $d=1$ and Assumption \ref{Assumptions_noise_local} holds with $s_{\psi}=1$ and for all $q_{\psi}<\infty$, then one can choose $p_0=q_0=2$, $\a_0=0$ and $s_0=1$ corresponding to the usual energy space for the thin-film equation.

The key point in the above result is the independence of the blow-up criteria of the original set $(p,\a,s,q)$ of admissible parameters. 
Such independence is essentially a consequence of the instantaneous regularization of solutions as given in Proposition \ref{Prop_regularization}. Indeed, at any time $t>\varepsilon$, the regularity of $u|_{[\varepsilon,\sigma)}$ does not depend on the original admissible parameters, and thus one can restart \eqref{Eq101} with any new set of admissible parameters $(p_0,\a_0,s_0,q_0)$ as long they are compatible with the noise, i.e., $s_0\leq s_{\psi}$ (see the proof of \cite[Theorem 2.10]{AV22_localRD}).

\subsection{Global well-posedness in one dimension}
\label{ss:global_1d_intro}
We turn our attention to the global in time well-posedness of the stochastic thin-film equation \eqref{Eq100}, which is of the form \eqref{Eq101} with $g(u)=m^{1/2}(u)$. This stochastic PDE admits several dissipated quantities which, in one dimension, allow us to extend the local well-posedness and regularization results globally in time using the blow-up criteria of Proposition \ref{prop:blow_up_criteria}. In particular, as in many works on the stochastic thin-film equation, our analysis is centred around the \emph{energy functional} 
\begin{equation}\label{Eq2}
	\EE(u)\,=\, \int_{\Tor}\Big[ \tfrac{1}{2}|u_x|^2 \,+\, \phi(u)\Big]\, \dd x
\end{equation}
of the solution $u$ to \eqref{Eq100}. The key point is that the functional $\EE$ estimates at the same time the norm $\|u\|_{H^{1}{(\T)}}^2$ by mass conservation and the smallness of $u$ due to the singularity of $\phi$ which we demand in Assumption \ref{Assumptions_phi}  to obtain global well-posedness, see \eqref{Eq45}. 

In general, it is challenging to close an a-priori estimate on \eqref{Eq2}, especially in the \emph{nonlinear noise} case, where $m(u) $ is not proportional to $ u^2$ and therefore $g(u)=m^{1/2}(u)$ is nonlinear. In \cite{dareiotis2021nonnegative,sauerbrey2023solutions} this was achieved for the Stratonovich interpretation of \eqref{Eq100} with $\phi=0$ by controlling the smallness of the solution leading to a film height which is positive a.e.\  for all times. In \cite{fischer_gruen_2018}, however, a repulsive interface potential $\phi$ was used instead of the Stratonovich correction to close an a-priori estimate on \eqref{Eq2} together with the \emph{entropy functional}
\begin{equation}
	\H_0 (u)\,=\, \int_{\T} h_{0} (u)\, \dd x ,\qquad  h_{0}(r)\,=\, \int_1^r \int_1^{r'}\frac{1}{m(r'')} \, \dd r''\, \dd r'
\end{equation} in the linear noise case $m(u) =u^2$. We extrapolate this approach using so-called \emph{$\alpha$-entropy functionals}
\begin{equation} \label{Eq46}
\H_\beta (u)\,=\, \int_{\T} h_{\beta} (u)\, \dd x ,\qquad h_{\beta}(r)\,=\, \int_1^r \int_1^{r'}\frac{(r'')^\beta}{m(r'')} \, \dd r''\, \dd r'
\end{equation}
for $\beta \in (-{1}/{2},1)$ originally defined in \cite{Beretta_Bertsch_DalPasso_95} for $\beta=\alpha +n-1$. Notably, while a key ingredient in \cite{fischer_gruen_2018} is the spatial discretization of \eqref{Eq100} developed in \cite{grun2000nonnegativity, Bertozzi_numerics}, which is compatible with the entropy estimate, an $\alpha$-entropy consistent discretization of \eqref{Eq100} is to the authors' knowledge not available and may depend on the specific choice of $\alpha$. Thus, working directly with the maximal local solutions provided by Theorem \ref{Thm_local} enables us to use a wider class of a-priori estimates and to cover many different cases of $m$. Specifically, we impose the following assumptions on the \emph{smooth function} $m\colon (0,\infty) \to (0,\infty)$.  
\begin{assumption}[Mobility coefficient]\label{Assumptions_m}
	There exist $n\in \R$ and $\nu\in [0,6)$ such that, 
	for all $r\in (0,\infty)$,
	\begin{align}& \label{Eq40}
		\limsup_{r\searrow 0} {m(r)}/{r^{n}} \,<\,\infty , \qquad
	\liminf_{r\searrow 0} {m(r)}/{r^{n+2}} \,>\,0,
		\\&
		\limsup_{r\to \infty} {m(r)}/r^{\nu} \,<\, \infty ,\qquad 
		\liminf_{r \to \infty} m(r) >0 , \label{Eq41}
		\\&\label{Eq42}
		|m'(r)|\,\lesssim \, {m(r)}/{r},\qquad |m''(r)|\,\lesssim\,{m(r)}/{r^2}.
	\end{align} 
\end{assumption}
\noindent
In this case, we call $n$ and $\nu$ \emph{exponent of degeneracy} and \emph{growth exponent} of $m$, respectively.
\smallskip

Due to the flexibility of our conditions \eqref{Eq40}--\eqref{Eq41}, the exponent of degeneracy $n$   and growth exponent $\nu$ are not uniquely determined by $m$. To keep the resulting Assumption \ref{Assumptions_phi} on $\phi$ below as weak as possible, one should, however, take the largest value of $n$  that the function $m$ allows for.

While \eqref{Eq40} expresses that $m(r)$ degenerates at least as much as the power $r^n$ near $0$, the condition \eqref{Eq41} bounds the growth of $m$  near $\infty$. The technical condition \eqref{Eq42} assumes that $m$ behaves under differentiation like a power-law. 
One can readily check that the following examples of mobilities satisfy Assumption \ref{Assumptions_m}:
 \begin{enumerate}[{\rm(i)}]
	\item \label{Ex1} {\rm (Power laws)} $m(r) =r^n$ for $n\in [0,6)$.
	\item \label{Ex2} {\rm (Mixed powers)} $m(r) =  c_1 r^{n_1}\,+\, \dots \,+\, c_J r^{n_J}$, $c_j \in (0,\infty)$, $n_j\in (-\infty, 6)$ provided $\displaystyle{\max_{j\in \{1,\dots,J\}}n_{j}\geq 0}$, which guarantees the second condition of \eqref{Eq41}.
	\item \label{Ex3} {\rm (Nonlinear Interpolation)} Let $m$ and $\tilde{m}$ be two mobility functions satisfying Assumption \ref{Assumptions_m} with respective exponents of degeneracy $n,\tilde{n}$ and growth exponents $\nu,\tilde{\nu}$.
	Then, for all $\delta>0$
	\[
	m_\delta(r)\,=\, \frac{m(r)\tilde{m}(r)}{\delta m (r) +\tilde{m}(r)}
	\]
	suffices Assumption \ref{Assumptions_m} with exponent of degeneracy $\max\{ n,\tilde{n}\}$ and growth exponent $\min\{\nu,\tilde{\nu}\}$.
\end{enumerate}
As mentioned before, when deriving the stochastic thin-film equation using a lubrication approximation, one obtains the mobility function $m(u) = u^3 + \lambda^{3-n} u^n$ with $n\in [1,3]$ and $\lambda\geq 0$ depending on the boundary condition of the fluid velocity near the substrate. Since the physically relevant regime is $u\ll 1$, this is usually approximated by $u^n$ in the mathematical literature, which is covered by \eqref{Ex1}, but we can also cover the former mobility function with \eqref{Ex2}.
The class \eqref{Ex3} is of mathematical interest since these nonlinear interpolations of two mobility functions can be used to construct solutions to the (stochastic) thin-film equation as a limit of strictly positive solutions to regularized equations, see \cite[Section 6]{BERNIS1990}.
We impose a corresponding assumption on the effective interface potential given by a \emph{smooth function} $\phi \colon (0,\infty) \to (0,\infty)$.
\begin{assumption}[Interface potential]\label{Assumptions_phi}
Let $n$ be the exponent of degeneracy of the mobility function $m$ as in Assumption \ref{Assumptions_m}. 
	There exists $\vartheta>\max\{2, 6-2n\}$ and $c_0 \in(0,\infty)$ such that, for all $r\in (0,\infty)$,
	\begin{align}
			 \label{Eq44}
		r^{-\vt} \,\lesssim \, \phi(r)\qquad \text{ and }\qquad
		 r^{-\vt-2} \,-\, c_0  \,\lesssim \, \phi''(r)\,\lesssim\, r^{-\vt-2}.
	\end{align}
\end{assumption}
The above assumption includes effective interface potentials of the form $\phi(u) = u^{-\vt} -u^{-2} +c_\vt$ with $ \vt >\max\{2, 6-2n\}$ as long as $c_\vt$ is large enough to ensure that $\phi(r)>0$ for all $r>0$. For $\vt=8$, this becomes \eqref{eq:prototype_example_potential} corresponding to the $6$-$12$ \emph{Lennard-Jones pair potential} for the van der Waals forces between the fluid and solid molecules. For $n\ge 2$, Assumption \ref{Assumptions_phi} reduces precisely to the assumption on the interface potential imposed in \cite[Hypothesis (H2)]{fischer_gruen_2018}. In particular, the continuous version
\begin{align}\label{Eq45}
	\sup_{x\in \T} u^{(2-\vt)/2}(x)\,\lesssim \,\EE(u) \,+\, \biggl(\int_{\T} u\, \dd x\biggr)^{(2-\vt)/2}
\end{align}
of \cite[Lemma 4.1]{fischer_gruen_2018} implies that a profile $u$ is strictly positive, if its energy \eqref{Eq2} is finite.
Moreover, for $n<2$, the coefficient $m^{1/2}(u)$ is not differentiable at $0$ anymore, leading to increased production of energy by the noise for small film heights. Thus, an improved control of the smallness of $u$ is required, which is reflected by the additional condition $\vt>6-2n$.
Finally, we state the main assumption on the noise, which allows us to obtain global well-posedness of the stochastic thin-film equation in dimension one.

\begin{assumption}[Noise regularity -- Global well-posedness]\label{Assumptions_noise_global}
	We assume that $(\psi_k)_{k\in\N}$ satisfies
	\begin{equation}
		\sum_{k\in \N} \| \psi_k \|_{W^{2,\infty}(\T^d;\R^d)}^2\,<\, \infty.
	\end{equation}
\end{assumption}

Note that the above implies that Assumption \ref{Assumptions_noise_local} holds with $s_{\psi}=1$ and for all $q_{\psi}<\infty$. 
We are ready to state our result on global well-posedness in one spatial dimension.
Below, uniqueness is understood as in Definition \ref{Defi_max}.

\begin{theorem}[Global well-posedness in one dimension -- It\^o]\label{Thm_global} Fix $s\in (1/2,1]$ and $d=1$. Suppose that  Assumptions \ref{Assumptions_m}, \ref{Assumptions_phi} and \ref{Assumptions_noise_global} are satisfied and \begin{equation}
\label{eq:initial_data_global_1d}
u_0\in  L_{\mathscr{F}_0}^{0}(\Omega,H^{s}(\Tor)) \quad \text{ satisfies }\quad \inf_{\Tor} u_0>0\text{ a.s. }
\end{equation}
Then there exists a unique progressively measurable process $u\colon \llbracket 0,\infty\rrparenthesis\to H^{s+2}(\T)$, such that a.s.\
\begin{align}
\label{eq:regularity_1d_global_0}
&\inf_{[0,t)\times \T} u>0 \ \ \text{ for all }\ t<\infty,\\
\label{eq:regularity_1d_global_1}
&u\in L^2_{\loc}([0,\infty); H^{s+2}(\Tor)) \,\cap\, C([0,\infty); H^{s}(\Tor)),
\end{align} 
and, a.s.\ for all $t>0$ and $\vp\in C^\infty(\T)$,
\begin{align}
\label{eq:thin_film_ito_formulation}
	\int_{\T} (u(t)-u_0)\vp \,\dd  x \,&=\,   \int_0^t 
	 \bigl\langle \vp_x  m(u), (u_{xx} - \phi'(u))_{x} \bigr\rangle_{H^{1-s}(\T)\times H^{s-1}(\T)}
	\, \dd r
	\\&
	-\, \sum_{k\in \N} \int_0^t  \int_{\T} \vp_x  m^{1/2}(u) \psi_k   \,\dd  x\, \dd \beta^{(k)}.
\end{align}
Finally, the solution $u$ instantaneously regularizes in time and space:
\begin{align}
\label{eq:regularity_1d_global_2}
&u\in H^{\theta,r}_{{\rm loc}}(0,\infty;H^{3-4\theta,\zeta}(\T)) \    \text{ for all }\theta\in [0,\tfrac{1}{2}),\ r,\zeta\in (2,\infty),\\
\label{eq:regularity_1d_global_3}
&u\in C_{\loc}^{\theta_1,\theta_2}((0,\infty)\times \T)  \   \text{ for all }\theta_1\in [0,\tfrac{1}{2}), \ \theta_2\in (0,3).
\end{align}
\end{theorem}

Proposition \ref{Prop_regularization} improves the assertions \eqref{eq:regularity_1d_global_2}--\eqref{eq:regularity_1d_global_3} in the case that the noise is more regular than required in Assumption \ref{Assumptions_noise_global}. In any case,
it follows from the instantaneous regularization result of \eqref{eq:regularity_1d_global_2}--\eqref{eq:regularity_1d_global_3} that we can replace the term $\langle \vp_x  m(u), (u_{xx} - \phi'(u))_{x} \rangle_{H^{1-s}(\T)\times H^{s-1}(\T)}$ in \eqref{eq:thin_film_ito_formulation} by 
\begin{equation}
\label{eq:replace_term_Hs}
\int_{\T} \vp_x  m(u) (u_{xx} - \phi'(u))_{x} \,\dd  x 
\end{equation}
for the solution $u$.

\smallskip 

Next, we turn our attention to the stochastic thin-film equation with noise in the Stratonovich form in one dimension:
\begin{equation}\label{Eq100_stra}
	\dd  u\,=\, -( m(u)  ( u_{xx}- \phi'(u))_x)_x\, \dd  t+\, \sum_{k\in \N}( m^{1/2}(u) \psi_k \circ \dd  \beta^{(k)} )_x,
	\quad u(0)=u_0.
\end{equation}
Note that, as $d=1$, the Stratonovich correction takes the form
\begin{align}
	\tfrac{1}{4}\sum_{k\in \N} (m'(u)m^{-1/2}(u) \psi_k ( m^{1/2}(u)\psi_k)_x )_x\,=\, \tfrac{1}{8}
	\sum_{k\in \N} ((m'(u))^2m^{-1}(u)u_x \psi_k^2 )_x\,+\, 
		\tfrac{1}{4}\sum_{k\in \N} (m'(u) \psi_k  \psi_k')_x.
\end{align}
Whenever $\sum_{k\in \N} \psi_k^2$ sums up to a constant $C$ independent of $x$ the above simplifies to 
\[\tfrac{C}{8}  ((m'(u))^2m^{-1}(u)u_x )_x,\]
which can be included in the second order term of \eqref{Eq101} if one modifies $\Phi$ appropriately.
We remark that, as observed by one of the anonymous referees, the latter condition expresses that 
	the noise intensity 
	\[
	\E \biggl|   
	\sum_{k\in \N} \psi_k(x) \beta^{(k)}_t
	\biggr|^2 \,=\, t \sum_{k\in \N}\psi_k^2(x)\,\equiv \, Ct
	\]
	is independent of the spatial location $x$ and 
	is satisfied for instance by any mollification of an $L^2(\T^d)$-cylindrical Brownian motion. In any case, the local well-posedness result Theorem \ref{Thm_local} holds then also for the Stratonovich interpretation of \eqref{Eq100}. Since the aforementioned a-priori estimates on the $\alpha$-entropy \eqref{Eq46} and the energy \eqref{Eq2} can be carried out analogously as for the It\^o interpretation \eqref{Eq100}, we also obtain the following result.

\begin{theorem}[Global well-posedness in one dimension -- Stratonovich]\label{cor_global_stratonovich}
Fix $s\in (1/2,1]$ and $d=1$. Let $u_0$ be as in \eqref{eq:initial_data_global_1d}. Suppose that Assumptions \ref{Assumptions_m}--\ref{Assumptions_noise_global} are satisfied, and that there exists $C>0$ such that  
\begin{equation}
\label{eq:assumption_noise_stratonovich}
\sum_{k\in \N}\psi_k^2(x)\equiv C\ \ \text{ for all }\ x\in \T.
\end{equation}
Then there exists a unique progressively measurable process $u\colon \llbracket 0,\infty\rrparenthesis\to H^{s+2}(\T)$ satisfying \eqref{eq:regularity_1d_global_0}, \eqref{eq:regularity_1d_global_1} and, for all $t>0$ and $\vp\in C^\infty(\T)$,
\begin{align}
\label{eq:thin_film_stra_formulation}
	\int_{\T} (u(t)-u_0)\vp \,\dd  x \,&=\,  \int_0^t 
	\bigl\langle \vp_x  m(u), (u_{xx} - \phi'(u))_{x} \bigr\rangle_{H^{1-s}(\T)\times H^{s-1}(\T)}
	\, \dd r
	\\& -\, \tfrac{C}{8}\int_0^t  \int_{\T} \vp_x  (m'(u))^2 m^{-1}(u) u_x  \, \dd x \, \dd r\, 
	-\, \sum_{k\in \N} \int_0^t  \int_{\T} \vp_x  m^{1/2}(u) \psi_k  \, \dd x \, \dd \beta^{(k)}.
\end{align}
Moreover, the solution $u$ enjoys the additional regularity \eqref{eq:regularity_1d_global_2} and \eqref{eq:regularity_1d_global_3}.
\end{theorem}

As for the It\^o formulation of the stochastic thin-film equation, we can replace the integrand of the deterministic integral on the right-hand side of \eqref{eq:thin_film_stra_formulation} by \eqref{eq:replace_term_Hs} for the solution $u$.
\smallskip

Finally, we remark that in the case of Stratonovich noise, additional cancellations occur, which allow for closing the ($\alpha$-) entropy estimate and subsequently the energy estimate for \eqref{Eq100} even if no interface potential is present, as carried out in \cite{dareiotis2021nonnegative, sauerbrey2023solutions}. Moreover, the $\alpha$-entropy function $h_\beta(r)$ behaves like $r^{\beta-n+2}$ near $0$, and the latter can become as singular as the interface potential $\phi(r)$ for mobility functions with an exponent of degeneracy $n>7/2$. Therefore, we expect preservation of positivity in the spirit of \cite[Theorem 4.1 (iii)]{Beretta_Bertsch_DalPasso_95} and consequently also global well-posedness of \eqref{Eq100_stra} without the interface potential for sufficiently degenerating mobility functions. However, this lies beyond the scope of the current article.

Global well-posedness in the physical dimension $d=2$ seems out of reach using our approach. The main obstacle is that the blow-up criteria in  Proposition \ref{prop:blow_up_criteria} require an a-priori estimate at least in $C^{\alpha}$ for some $\alpha>0$. However, the energy functional only provides an estimate in $H^{1}(\T^d)$, and the latter embeds in a space of H\"older continuous functions precisely if $d=1$.

We lastly point out that there is a large body of research concerned with maximal regularity estimates and corresponding well-posedness results for the deterministic thin-film equation for the case in which a contact line is present, see \cite{gnann_wisse} and the references therein. It is an intriguing problem to derive stochastic counterparts of these results, which usually rely on suitable transformations of the equation, to deduce well-posedness of  \eqref{Eq100} with $\phi=0$ also in situations in which the film height becomes $0$.

\subsection{Further comments on the literature}
\label{ss:review_literature}
The stochastic thin-film equation was independently derived in \cite{DMES2005} and \cite{GruenMeckeRauscher2006}.
We review here some of the mathematical works conducted on it in recent years. A main part of it is concerned with the construction of martingale, i.e., probabilistically weak solutions to \eqref{Eq100} based on the stochastic compactness method. As in the current article, the driving noise $B$ is usually assumed to be spatially correlated to close suitable a-priori estimates for \eqref{Eq100}. This line of research was initiated in \cite{fischer_gruen_2018}, where martingale solutions to \eqref{Eq100} in the It\^o sense on the one-dimensional torus $\T$ with quadratic mobility $m(u)=u^2$ are constructed using a finite difference scheme. In particular, it is assumed that the interface potential $\phi$ is sufficiently repulsive near zero, leading to a strictly positive film height for all times. In \cite{GessGann2020, KleinGruen22}, martingale solutions to the Stratonovich interpretation of \eqref{Eq100} with $m(u)=u^2$ and $\phi=0$ are constructed using a time splitting scheme and an approximation based on a vanishing interface potential, respectively, again in $d=1$. In \cite{dareiotis2021nonnegative, dareiotis2023solutions,sauerbrey2023solutions} martingale solutions to the Stratonovich interpretation of \eqref{Eq100} for $m(u)=u^n$, $\phi=0$, $d=1$ for various $n\in (2,4)$ are constructed using respective less degenerate approximations of \eqref{Eq100}. In particular, these results allow for a nonlinear noise term in \eqref{Eq100}. The physically relevant, two-dimensional setting was addressed in \cite{metzger2022existence,Sauerbrey_2021}, where higher-dimensional versions of \cite{fischer_gruen_2018, GessGann2020} are obtained. Martingale solutions to perturbations of the Stratonovich interpretation of \eqref{Eq100} with $m(u)=u^2$, $\phi=0$, $d=1$ are constructed in \cite{kapustyan2023film}. For \eqref{Eq100} with spatially white $B$, additional computations and numerical simulations are provided in \cite{Gvalani_19, GGKO21}. Additionally, an effort to treat this singular stochastic PDE using tree-free regularity structures is initiated in \cite{gvalani2023stochastic}.

\subsection{Notation} 
\label{ss:notation}
We write $\T^d$ for the $d$-dimensional torus $\R^d/ \Z^d$. The usual Sobolev space on an open subset $\mathcal{O}$ of $\R^d$ or $\T^d$ is denoted by $W^{s,q}(\mathcal{O})$ for an integer smoothness index $s\in \N_0$. For the definition of the Besov spaces $B^{s}_{q,p}(\R^d)$ and Bessel-potential spaces $H^{s,q}(\R^d)$ the reader is referred to \cite[Section 2.1.2]{Runst_Sickel} and for the periodic spaces $B^{s}_{q,p}(\T^d)$ and $H^{s,q}(\T^d)$ to \cite[Section 3.5.4]{schmeisser1987topics}.  The Besov and Bessel-potential space on an open subset $\mathcal{O}$ of $\R^d$ or $\T^d$ is defined as the set of restrictions of functions from the Besov and Bessel-potential space on the whole space and equipped with the induced quotient norm. The corresponding spaces of vector fields $W^{s,q}(\mathcal{O};\R^d)$, $B^{s}_{q,p}(\mathcal{O};\R^d)$ and $H^{s,q}(\mathcal{O}, \R^d)$ are defined as the $d$-fold direct sum of these spaces. In any of these situations the symbol $H^{s}$ stands for $H^{s,2}$.

We write $L^p(S,\mu;\mathscr{X})$ for the Bochner space of strongly measurable, $p$-integrable $\mathscr{X}$-valued functions for a measure space $(S,\mu)$ and a Banach space $\mathscr{X}$ as defined in \cite[Section 1.2b]{Analysis1}. If $\mathscr{X}=\R$, we write $L^p(S,\mu)$, and if it is clear which measure we refer to, we also leave out $\mu$. Moreover, if $S$ is countable and equipped with the counting measure, we write $\ell^p(S)$ instead of $L^p(S)$. If on the other hand $I$ is an open interval and $w$ a density, we write $L^p(I,w;\mathscr{X})$ for $L^p(I,w\, \dd t;\mathscr{X})$. In particular, we are interested in power weights of the form $w_\a^s(t)=|t - s|^\kappa$. The corresponding $\mathscr{X}$-valued fractional Sobolev space with weight $w_\a^s$ is denoted  by $H^{\theta,p}(I,w_\kappa^s;\mathscr{X})$ (see Appendix \ref{app:Sobolev_spaces} for the definition) and $H_{\loc}^{\theta,p}(I,w_\kappa^s;\mathscr{X})$ is the intersection of $H^{\theta,p}(J,w_\kappa^s;\mathscr{X})$ for all intervals $J$, which are compactly contained in $I$. Whenever we write '$a,b$' instead of an interval in the above spaces we mean the open interval $I=(a,b)$, e.g., $H^{\theta,p}(a,b,w_\kappa^s;\mathscr{X})$ stands for $H^{\theta,p}((a,b),w_\kappa^s;\mathscr{X})$.

If $(S,d)$ is a metric space we write $C(S;\mathscr{X})$ for the continuous function and $C^{\theta}(S;\mathscr{X})$ for the subset of $\theta$-H\"older continuous functions for $\theta \in (0,\infty)\setminus \N$, where we leave $\mathscr{X}$ out again if $\mathscr{X}=\R$. If $I$ is an open interval, we define the anisotropic H\"older space 
\begin{equation}
	C^{\theta_1,\theta_2}(I\times \T^d)\,=\, C^{\theta_1}(I;C(\T^d))\cap C(I;C^{\theta_2}(\T^d))
\end{equation}
for $\theta_i \in (0,\infty)\setminus \N$, $i\in \{1,2\}$ and accordingly $C_{\loc}^{\rho_1-,\rho_2-}(I\times \T^d)$ as the intersection of the spaces $C^{\theta_1,\theta_2}(J\times \T^d)$ for all $\theta_i\in (0,\rho_i)\setminus \N$ and intervals $J$ which are compactly contained in $I$. 

To ease the notation of these spaces, we introduce for a given quadruple $(p,\a,s,q)$ the notation $X_0=H^{s-2,q}(\T^d)$ and $X_1 = H^{s+2,q}(\T^d)$. Following \cite{AV19_QSEE_1,AV19_QSEE_2}, we write
\begin{equation}
\label{eq:interpolation_spaces_definition_notation_section}
	X_{\theta}\,:=\,[X_0,X_1]_{\theta},\qquad \Xap\,:=\,(X_0,X_1)_{1-\frac{1+\a}{p},p},\qquad  \Xp\,:=\,\Xzp
\end{equation}
for $\theta\in (0,1)$ and $p,\kappa$ as in \eqref{Eq1} for the complex and real interpolation spaces of $X_0$ and $X_1$, see  \cite[Chapter 1 and Chapter 2]{Luninterp} for a definition and Proposition \ref{prop:tracespace} for their appearance in the study of traces of solutions to SPDE with path regularity as in Theorem \ref{Thm_local}. In particular,  it holds
\begin{align}\label{Eq72}
	X_{\theta} \,=\, H^{s+2-4\theta,q}(\T^d),\qquad \Xap \,=\, B^{s+2-4\frac{1+\kappa}{p}}_{q,p}(\Tor^d)
\end{align}
by \cite[Section 3.6.1]{schmeisser1987topics} meaning that the Banach spaces coincide as sets and carry equivalent norms. 

We also fix throughout the manuscript a filtered probability space $(\Omega, \mathfrak{A},\mathscr{F}, \P)$ carrying a family of independent Brownian motions $(\beta^{(k)})_{k\in \N}$. We write $\E$ for the expectation on $(\Omega, \mathfrak{A}, \P)$ and $\mathscr{P}$ for the progressive $\sigma$-field. For a Banach space  $\mathscr{X}$ we denote the subspaces of $\mathscr{F}_t$-measurable and progressively measurable random variables in $L^p(\Omega; \mathscr{X})$ by $L_{\mathscr{F}_t}^p(\Omega;\mathscr{X})$ and $L_{\mathscr{P}}^p(\Omega;\mathscr{X})$, respectively. If additionally $\mathscr{Y}$ is another Banach space and $H$ is a Hilbert space, we write $\LLL(\mathscr{X},\mathscr{Y})$ for the space of bounded linear operators and $\gamma(H,\mathscr{X})$ for the space of $\gamma$-radonifying operators, see \cite[Chapter 9]{Analysis2} or Appendix \ref{app:gamma_radonifying}. 
For a stopping time $\tau\colon \Omega\to [s,T]$, we define 
\begin{align}
\llbracket s,\tau\rrparenthesis \,:=\, \bigl\{(\omega,t)\in \Omega\times [s,T]\,\big|\, t\,<\,  \tau(\omega)\bigr\}.
\end{align}
Let $(\mathscr{X}_t)_{t\in [s,T]}$ be a family of Banach spaces such that $\mathscr{X}_t$ consists of functions from $[s,T]$ to another Banach space $\mathscr{X}$ such that $f|_{[s,t]}\in \mathscr{X}_t$ for each $f\in \mathscr{X}_T$ and $\|f|_{[s,t]}\|_{\mathscr{X}_t}$ is increasing in $t$. Then we define $L_{\Progress}^p(\Omega; \mathscr{X}_\tau)$ as the restrictions of processes from $L_{\Progress}^p(\Omega;\mathscr{X}_T)$ to $\llbracket s,\tau\rrparenthesis$ and equip it with the norm
\[
\|u|_{\llbracket s,\tau\rrparenthesis} \|_{L^p(\Omega;\mathscr{X}_\tau)}\,:=\, \E\bigl[
\|u|_{[s,\tau]}\|_{\mathscr{X}_\tau}^p
\bigr]^\frac{1}{p},
\]
which is well-defined by \cite[Lemma 2.15]{AV19_QSEE_1}. In the situation that $\mathscr{X}_t = L^p((s,t),w_\a^s;\mathscr{X})$, we write $L_{\Progress}^p( (s,\tau)\times \Omega ,w_\a^s;\mathscr{X})$ for $L_{\Progress}^p(\Omega; \mathscr{X}_\tau)$.  

For two quantities $x$ and $y$, we write $x\lesssim y$ if there exists a universal constant $C$ such that $x\le Cy$. If a constant depends on parameters $(p_1,\dots)$ we either mention it explicitly or indicate this by writing $C_{(p_1,\dots)}$ and correspondingly $x\lesssim_{(p_1,\dots)}y$ whenever $x\le C_{(p_1,\dots)}y$. Finally, we write $x\eqsim_{(p_1,\dots)} y$, whenever $x\lesssim_{(p_1,\dots)} y$ and $y\lesssim_{(p_1,\dots)}x$.
Other recurring symbols introduced throughout the manuscript are collected in the following list:
\begin{itemize}
	\item The coefficients $m, \Phi, g$ and $(\psi_k)_{k\in\N}$ of the quasilinear SPDE: \eqref{Eq101}--\eqref{eq:assumption_psi}\noeqref{eq:assumption_coeff}
	\item The coefficients $m,\phi$ and $B$ of the stochastic thin-film equation: \eqref{Eq100}--\eqref{eq:assumption_STFE_coeff}
	\item The coefficients $A,F$ and $G$ of the stochastic evolution equation: \eqref{Eq8}--\eqref{Eq125}
	\item The parameters $(p,\kappa,s,q)$ are always taken according to Assumption \ref{Assumptions_coefficients}.
	\item The functionals $\EE$ and  $\H_\beta$: \eqref{Eq2}--\eqref{Eq46}
	\item The parameters $n,\nu$ and $\vt$ are introduced in Assumptions \ref{Assumptions_m}--\ref{Assumptions_phi}.
	\item The coefficients  $m_j, \Phi_j$ and $ g_j$ of the regularized quasilinear SPDE: \eqref{eq:regularization_eta_j}--\eqref{Eq102}
	\item The coefficients $A_j,F_j$ and $G_j$ of the regularized stochastic evolution equation \eqref{eq:reg_coeff_QSEE}--\eqref{Eq119}
	\item The operator $\Psi $ is defined before \eqref{eq:identification_gamma_norms}.
	\item The operator $\mathcal{A}$ is defined in and below \eqref{Eq12}.
	\item The classes $\MRtash$ and $\MRtaszh$ are defined in Definition \ref{def:SMRgeneralized}.
	\item The constants $C_{\mathcal{A}}^{l,\theta,p,\hat{\kappa}}$ and $K_{\mathcal{A}}^{l,\theta,p,\hat{\kappa}}$ for $l\in \{\det,\stoc\}$ are defined in and above \eqref{eq:constKellthetapa}.
	\item The  operator $\mathcal{A}_0$ is defined in \eqref{eq:bilaplacian_H_infty_calculus}.
	\item The operators $E, E_{y,r}^{\R^d}, E_{y,r}^{\T^d}$ and $\tilde{E}_{y,r}^{\T^d}$ are introduced at the start of the proof of Lemma \ref{l:estimates_small_interval}.
	\item The quantity $
	\mathcal{N}^\a (u;\sigma)
	$  is defined below \eqref{Eq30_stronger_proof}.
	\item The range of the exponent $\gamma$ is specified in \eqref{Eq10}.
	\item 
	The modified exponents $\tilde{\nu}$ and $\tilde{n}$ are defined above the equations \eqref{Eq78} and \eqref{Eq61}, respectively.
\end{itemize}

\section{Local well-posedness of thin-film type equations in any dimension}
The purpose of this section is to show local well-posedness, blow-up criteria, and instantaneous regularization for \eqref{Eq101} as stated in Subsection \ref{ss:local_all_dimensions_intro}. To this end, we fix in this subsection the smooth coefficients $m\colon(0,\infty)\to (0,\infty)$ and $g,\Phi\colon(0,\infty)\to \R$ and do not indicate if an implicit constant depends on them. Our strategy is to use the general theory for quasilinear parabolic stochastic evolution equations developed in \cite{AV19_QSEE_1,AV19_QSEE_2} and apply the results \cite[Theorem 4.7]{AV19_QSEE_1} and \cite[Theorem 4.9, Theorem 6.3]{AV19_QSEE_2} to \eqref{Eq101}. However, since we allow for a degeneracy of the operator $A[u]$ when $u$ approaches $0$, we need to consider regularized versions of \eqref{Eq101} instead. Specifically, we fix a smooth and increasing function $\eta\colon \R\to \R$ with $\eta(r)=1$ for $r\ge 2$ and  $\eta(r)=0$ for $r\le 1$ and define $\eta_j(r) = \eta(jr)$ for $j\in \N$. Setting 
\begin{align}&
	m_j(r)\,=\, \eta_j(r)m(r)\,+\, (1- \eta_j(r)), \qquad
	\label{eq:regularization_eta_j}
	\Phi_j(r)\,=\, \eta_j(r)\Phi(r),\qquad
	g_j(r)\,=\, \eta_j(r)g(r),
\end{align}
gives rise to smooth functions $m_j, \Phi_j, g_j\colon \R\to \R$. The new coefficient $m_j$ is bounded away from $0$, and consequently, the leading order operator of
\begin{align}\label{Eq102}
	\dd u\,+\,\div( m_j(u) \nabla \Delta u)\, \dd t\,=\, \div(\Phi_j(u)\nabla u )\, \dd t \,+\, \sum_{k\in \N}\div( g_j(u) \psi_k )\, \dd \beta^{(k)}, \quad u(0)=u_0,
\end{align}
is non-degenerate, so the theory for parabolic stochastic evolution equations developed in \cite{AV19_QSEE_1,AV19_QSEE_2} becomes applicable. Analogously to \eqref{Eq8}, we define
\begin{align}\label{eq:reg_coeff_QSEE}&
	A^{(j)}[u](f)\,=\, \div( m_j(u) \nabla \Delta f), \qquad
	F^{(j)}(u)\,=\,  \div(\Phi_j(u)\nabla u ),\qquad
	G_k^{(j)}(u)\,=\, \div( g_j(u) \psi_k )
\end{align}
and $G^{(j)}[u](e_k)  = G_k^{(j)}(u)$,
so that \eqref{Eq102} takes the form
\begin{align}\label{Eq119}
	\dd u\,+\,A^{(j)}[u](u)\,\dd t \,=\,  F^{(j)}(u)\, \dd t \,+\, G^{(j)}[u]\,\dd W, \qquad u(0)=u_0,
\end{align}
of \cite[Eq. (1.1)]{AV19_QSEE_1}.

\smallskip

The rest of this section is organized as follows. In Subsection \ref{Sec_lipschitz} we check the local Lipschitz conditions on the coefficients $A^{(j)}$, $F^{(j)}$ and $G^{(j)}$ from \cite[Hypothesis (H')]{AV19_QSEE_1}. Subsection \ref{Sec_SMR} is devoted to proving the stochastic maximal regularity of the linear problem
\begin{equation}
	\dd u\, +\,\div(a \nabla \Delta u)\,\dd t\,=\,f\, \dd t+ g\, \dd W, \quad u(0)=u_0,
\end{equation}
with a positive and bounded coefficient $a:[0,\infty)\times \O\to B^{s+2-4\frac{1+\kappa}{p}}_{q,p}(\Tor^d)$,
which allows us to prove local well-posedness and the blow-up criteria for \eqref{Eq102} in Subsection \ref{Sec_truncated_problem}. In Subsection \ref{Sec_reg}, we show instantaneous regularization of solutions to \eqref{Eq102} for sufficiently smooth noise. Finally, in Subsection  \ref{Sec_orig_problem} we transfer these statements to the original equation \eqref{Eq101} and prove Theorem \ref{Thm_local} and Propositions \ref{Prop_regularization}-\ref{prop:blow_up_criteria}.

\smallskip

Before proceeding, we point out that Assumption \ref{Assumptions_noise_local} is equivalent to demanding that the operator $\Psi:\ell^2(\N)\to H^{1+s_{\psi},q_{\psi}}(\T^d;\R^d)$ defined by $\Psi(e_k)=\psi_k$ is \emph{$\gamma$-radonifying}, see \cite[Chapter 9]{Analysis2} or Appendix \ref{app:gamma_radonifying}. Indeed, Proposition \ref{prop:identification_gamma_spaces} ensures 
\begin{equation}
\label{eq:identification_gamma_norms}
\g(\ell^2(\N),H^{s,q}(\T^d))=H^{s,q}(\T^d;\ell^2(\N))  \text{
with equivalent norms},
\end{equation} 
where $\g(\ell^2(\N),H^{s,q}(\T^d))$ denotes the set of $\g$-radonifying operators from $\ell^2(\N)$ to $H^{s,q}(\T^d)$, cf., \cite[Definition 9.1.4]{Analysis2} or Appendix \ref{app:gamma_radonifying}. The identification \eqref{eq:identification_gamma_norms} in particular implies 
\begin{equation}
	\|\Psi \|_{\gamma(\ell^2(\N), H^{1+s_\psi, q_\psi}(\T^d;\R^d))}\, \eqsim \, 
	\|
\psi 
\|_{H^{1+s_\psi,q_{\psi}}(\Tor^d;\ell^2(\N;\R^d))}\,\eqsim\, \biggl\|
\biggl(\sum_{k\in \N} \bigl| (1-\Delta)^{(1+s_{\psi})/2}\psi_k  \bigr|^2 \biggr)^{1/2}
\biggr\|_{{L^{q_{\psi}}(\Tor^d)}}.
\end{equation}
Below, we will frequently use the equivalence \eqref{eq:identification_gamma_norms} without further mentioning it.
For more information on the role of $\gamma$-radonifying operators in the context of the $L^p$-theory of stochastic evolution equations, the reader is referred to \cite{AV19,AV19_QSEE_1,AV19_QSEE_2,Brz2,LoVer,MaximalLpregularity,NVW13,VP18} and the references therein. 

\smallskip

Additionally, we recall the shorthand notation introduced in \eqref{Eq72} for the functions for function spaces: 
\[
X_{\theta} \,=\, H^{s+2-4\theta,q}(\T^d),\quad\text{and }\quad \Xap \,=\, B^{s+2-4\frac{1+\kappa}{p}}_{q,p}(\Tor^d).\]

\subsection{Local Lipschitzianity of the regularized coefficients}\label{Sec_lipschitz}
As laid out earlier, this subsection is devoted to showing local  Lipschitz and growth estimates of the coefficients of \eqref{Eq119}. The main ingredients are composition and product estimates in Bessel-potential estimates, which result from a frequency decomposition of the involved functions, see \cite[Chapter 2]{ToolsPDEsTaylor} or \cite[Chapters 4 and 5]{Runst_Sickel} for an introduction. More precisely, we repeatedly use  \cite[Proposition 4.1 (1), (3)]{AV21_max_reg_torus} and  \cite[Theorem 1, p.373]{Runst_Sickel}, which we recall for this reason in Appendix \ref{app:para_est}. To proceed, we fix a quadruple $(p,\a,s,q)$ and a sequence $(\psi_k)_{k\in \N}$ subject to Assumptions \ref{Assumptions_coefficients} and \ref{Assumptions_noise_local} with $(s_{\psi},q_{\psi})=(s,q)$ for the remainder of this subsection and allow for all constants to depend on this particular choice.
\begin{lemma}\label{Lemma_aj}
	For all $j,n\in \N$ there exist constants $C_{j,n}, L_{j,n}\in (0,\infty)$ such that
	\begin{align}&\label{Eq110}
		\|	A^{(j)}[u]\|_{\LLL(X_1,X_0)}\,\le \, C_{j,n} \bigl(1+\|u \|_{\Xap}\bigr),\\&\label{Eq106}
		\|	A^{(j)}[u]-	A^{(j)}[v]\|_{\LLL(X_1,X_0)}\,\le \, L_{j,n} \|u - v \|_{\Xap}
	\end{align}
	for all $u,v\in \Xap$ with $\|u\|_{\Xap}, \|v\|_{\Xap}\le n$.
\end{lemma}
\begin{proof}
	We first verify the local Lipschitz condition \eqref{Eq106} and observe that
	\begin{align}\label{Eq107}
		\|	(A^{(j)}[u]-	A^{(j)}[v])f\|_{X_0} \,\lesssim \, \| (m_j(u) - m_j(v))\nabla \Delta f\|_{H^{s-1,q}(\Tor^d; \R^d)}.
	\end{align}
	
	We distinguish several cases and start with $s-1 = 0$. Due to H\"older's inequality, we can bound
	\eqref{Eq107} further  by
	\begin{align}&
		\| m_j(u) - m_j(v) \|_{L^\infty(\Tor^d)}\|\nabla \Delta f\|_{L^q(\Tor^d; \R^d)}
		\,\lesssim \, 
		\| m_j(u) - m_j(v) \|_{L^\infty(\Tor^d)}\|f\|_{H^{s+2,q}(\Tor^d)}
	\end{align}
	and it remains to use the Sobolev embedding $\Xap\hookrightarrow L^\infty(\Tor^d)$ guaranteed by \eqref{Eq104} together with local Lipschitz continuity of $m_j$ to deduce \eqref{Eq106}.
		
	Next, we assume that $s-1>0$, in which case we employ the paraproduct estimate Proposition \ref{prop:para}~\eqref{item:prop_para_1} to bound \eqref{Eq107} by
	\begin{align}\begin{split}\label{Eq108}&
			\| m_j(u) - m_j(v) \|_{L^\infty(\Tor^d)}\|\nabla \Delta f\|_{H^{s-1,q}(\Tor^d; \R^d)}
			\,+\, 
			\| m_j(u) - m_j(v) \|_{H^{s-1,l}(\Tor^d)}\|\nabla \Delta f\|_{L^{r}(\Tor^d; \R^d)}
		\end{split}
	\end{align}
	for $l\in (1,\infty)$, $r\in (1,\infty]$ subject to 
	\begin{equation}\label{Eq109}
		\tfrac{1}{q} \,=\, \tfrac{1}{l}\,+\, \tfrac{1}{r}.
	\end{equation}
	We claim that we can choose  $l,r$ such that additionally $\Xap\hookrightarrow H^{s-1,l}(\Tor^d)$ and $H^{s-1,q}(\Tor^d)\hookrightarrow L^r(\Tor^d)$. 
	Indeed, if $s-1-\frac{d}{q}>0$, the choice $l=q$, $r=\infty$ is feasible. Otherwise, if $s-1-\frac{d}{q}\le 0$, we can choose $r$ such that
	\[s-1-\tfrac{d}{q}\,>\,
	\tfrac{-d}{r}\,>\, s-1-\tfrac{d}{q}-\epsilon
	\] 
	for $\epsilon>0$. Then, by \eqref{Eq109}
	\[
	\tfrac{d}{l}\,>\, s-1-\epsilon \quad \iff\quad s-1-\tfrac{d}{l} < \epsilon
	\]
	and by \eqref{Eq104} we can choose $\epsilon$ smaller than
	\begin{equation}\label{Eq121}
		s\,+\,2\,-\,4\,\tfrac{1+\kappa}{p}\,-\,\tfrac{d}{q}
	\end{equation}
	resulting in the embedding $\Xap\hookrightarrow H^{s-1,l}(\Tor^d)$.
	Thus, we can estimate \eqref{Eq108} by
	\begin{align}&
		\bigl(
		\| m_j(u) - m_j(v) \|_{L^\infty(\Tor^d)} \,+\, \| m_j(u) - m_j(v) \|_{H^{s-1,l}(\Tor^d)}  \bigr)\|\nabla \Delta f\|_{H^{s-1,q}(\Tor^d; \R^d)}
		\\&\quad \lesssim\,
		\| m_j(u) - m_j(v) \|_{L^\infty(\Tor^d)\cap H^{s-1,l}(\Tor^d)} \|f\|_{H^{s+2,q}(\Tor^d)}.
	\end{align}The desired estimate \eqref{Eq106} follows by local Lipschitz continuity of $u\mapsto m_j(u)$ in $L^\infty(\Tor^d)\cap  H^{s-1,l}(\Tor^d)$, see Proposition \ref{prop:comp}~\eqref{item:prop_comp_H}, together with the embedding 
	$\Xap\hookrightarrow L^\infty(\Tor^d)\cap  H^{s-1,l}(\Tor^d)$.
	
	Lastly, we consider the case $s-1<0$, in which the condition \eqref{Eq103} becomes relevant. Then, an application of Proposition \ref{prop:para}~\eqref{item:prop_para_3} yields that the right-hand side of \eqref{Eq107} is bounded by 
	\begin{align}\label{Eq23}&
		\| m_j(u) - m_j(v) \|_{L^\infty(\Tor^d)}\|\nabla \Delta f\|_{H^{s-1,q}(\Tor^d; \R^d)}
		\,+\, 
		\| m_j(u) - m_j(v) \|_{H^{\tau,\zeta}(\Tor^d)}\|\nabla \Delta f\|_{H^{s-1,q}(\Tor^d; \R^d)}
	\end{align}
	 for any $\tau>\max\{\frac{d}{\zeta}, 1-s\}$ and $\zeta \in [q',\infty)$. If we can choose $\tau,\zeta$ such that $\Xap\hookrightarrow H^{\tau,\zeta}(\Tor^d)$ the claimed estimate \eqref{Eq106} follows as in the previous case. But we can choose simply $\zeta = q$ and $\tau$ slightly smaller than 
	\[
	s\,+\,2\,-\,4\,\frac{1+\kappa}{p}
	\]
	by \eqref{Eq104} and \eqref{Eq103}.
	The growth condition \eqref{Eq110} follows in all cases by choosing $v=0$ in \eqref{Eq106} and noticing that $A^{(j)}[0]= \Delta^2$, completing the proof.
\end{proof}
\begin{lemma}\label{Lemma_fj}
	For all $j,n\in \N$ there exist constants $C_{j,n}, L_{j,n}\in (0,\infty)$ such that
	\begin{align}&\label{Eq111}
		\|	F^{(j)}(u)\|_{X_0}\,\le \, C_{j,n} \|u \|_{\Xap},\\&\label{Eq112}
		\|	F^{(j)}(u)-	F^{(j)}(v)\|_{X_0}\,\le \, L_{j,n} \|u - v \|_{\Xap}
	\end{align}
	for all $u,v\in \Xap$ with $\|u\|_{\Xap}, \|v\|_{\Xap}\le n$.
\end{lemma}
\begin{proof}We start again with \eqref{Eq112} and  calculate 
	\begin{align}&
		\|F^{(j)}(u)-	F^{(j)}(v)\|_{X_0}\,\lesssim\, \|\Phi_j(u)\nabla u - \Phi_j(v)\nabla v \|_{H^{s-1,q}(\Tor^d;\R^d)}
		\\&\quad \le\, \|\Phi_j(u)(\nabla u - \nabla v )\|_{H^{s-1,q}(\Tor^d;\R^d)} \,+\, 
		\|(\Phi_j(u)- \Phi_j(v))\nabla v \|_{H^{s-1,q}(\Tor^d;\R^d)}.
	\end{align}
	Since these terms compare structurally to \eqref{Eq107}, we continue as in the proof of Lemma \ref{Lemma_aj}. 
	If $s-1=0$, we proceed by estimating 
	\begin{align}&
		\|\Phi_j(u)(\nabla u - \nabla v )\|_{H^{s-1,q}(\Tor^d;\R^d)} \,\lesssim \,
		\|\Phi_j(u)\|_{L^\infty(\Tor^d)}\|\nabla u - \nabla v \|_{L^{q}(\Tor^d;\R^d)}
		\\&\quad \lesssim \,  
		\|\Phi_j(u)\|_{L^\infty(\Tor^d)}\|u-v \|_{H^{s,q}(\Tor^d)}
		\,\lesssim \,  
		\|\Phi_j(u)\|_{L^\infty(\Tor^d)}\|u-v \|_{\Xap},
	\end{align}
	where we used in the last step that $\Xap\hookrightarrow X_{1/2}$. Moreover, we have
	\begin{align}&
		\|(\Phi_j(u)- \Phi_j(v))\nabla v \|_{H^{s-1,q}(\Tor^d;\R^d)}
		\,\lesssim \, \|\Phi_j(u)- \Phi_j(v)\|_{L^\infty(\Tor^d)}\|\nabla v \|_{L^{q}(\Tor^d;\R^d)}
		\\&\quad \lesssim \, \|\Phi_j(u)- \Phi_j(v)\|_{L^\infty(\Tor^d)}\| v \|_{\Xap}
	\end{align}
	by the same argument. The desired estimate \eqref{Eq112} follows by local Lipschitz continuity of $\Phi_j$ and the embedding $\Xap\hookrightarrow L^\infty(\Tor^d)$ by \eqref{Eq104}.
	
	If instead $s-1>0$, we choose $l,r$ again such that \eqref{Eq109} holds and additionally $\Xap\hookrightarrow H^{s-1,l}(\Tor^d)$ and $H^{s-1,q}(\Tor^d)\hookrightarrow L^r(\Tor^d)$. Then Proposition \ref{prop:para}~\eqref{item:prop_para_1} yields that
	\begin{align}
		&
		\|\Phi_j(u)(\nabla u - \nabla v )\|_{H^{s-1,q}(\Tor^d;\R^d)} \\&\quad \lesssim \,
		\|\Phi_j(u )\|_{L^\infty(\Tor^d)}\|\nabla u-\nabla v\|_{H^{s-1,q}(\Tor^d; \R^d)}
		\,+\, 
		\| \Phi_j(u ) \|_{H^{s-1,l}(\Tor^d)}\|\nabla u-\nabla v\|_{L^{r}(\Tor^d; \R^d)}
		\\&\quad\lesssim \,
		\bigl(
		\|\Phi_j(u )\|_{L^\infty(\Tor^d)}\,+\, 	\| \Phi_j(u ) \|_{H^{s-1,l}(\Tor^d)}
		\bigr)
		\|\nabla u-\nabla v\|_{H^{s-1,q}(\Tor^d; \R^d)}
		\\&\quad\le \, L_{j,n} \|u-v\|_{H^{s,q}(\Tor^d)},
	\end{align}
	where in the last step, we used 
	$\Xap\hookrightarrow L^\infty(\Tor^d)\cap  H^{s-1,l}(\Tor^d)$ together with local Lipschitz continuity of $u\mapsto \Phi_j(u)$ in the latter space, see again Proposition \ref{prop:comp}~\eqref{item:prop_comp_H}. Because of $\Xap\hookrightarrow X_{1/2}$, we obtain
	\begin{align}\label{Eq113}
		\|\Phi_j(u)(\nabla u - \nabla v )\|_{H^{s-1,q}(\Tor^d;\R^d)} \,\lesssim\, L_{j,n} \|u-v\|_{\Xap}.
	\end{align}
	Analogously, we derive that
	\begin{align}&
		\|(\Phi_j(u)- \Phi_j(v))\nabla v \|_{H^{s-1,q}(\Tor^d;\R^d)}
		\\&\quad\lesssim \,
		\|\Phi_j(u ) - \Phi_j(v)\|_{L^\infty(\Tor^d)}\|\nabla v\|_{H^{s-1,q}(\Tor^d;\R^d)}
		\,+\, 
		\| \Phi_j(u ) - \Phi_j(v) \|_{H^{s-1,l}(\Tor^d)}\|\nabla v\|_{L^{r}(\Tor^d; \R^d)}
		\\&\quad
		\lesssim\,
		\bigl(
		\|\Phi_j(u ) - \Phi_j(v)\|_{L^\infty(\Tor^d)}\,+\, 	\| \Phi_j(u ) - \Phi_j(v) \|_{H^{s-1,l}(\Tor^d)}
		\bigr)\|v \|_{H^{s,q}(\Tor^d)}
		\\&\quad
		\lesssim\,
		L_{j,n} \|u-v \|_{\Xap}.
	\end{align}
	Together with \eqref{Eq113}, we conclude that \eqref{Eq112} holds.
	
	Lastly, if $s-1<0$, we choose $\tau$ and $\zeta$ such that $\tau>\max\{\frac{d}{\zeta}, 1-s\}$, $\zeta \in [q',\infty)$ and $\Xap\hookrightarrow H^{\tau,\zeta}(\Tor^d)$ and apply Proposition \ref{prop:para}~\eqref{item:prop_para_3} to estimate 
	\begin{align}
		&
		\|\Phi_j(u)(\nabla u - \nabla v )\|_{H^{s-1,q}(\Tor^d;\R^d)} \\&\quad\lesssim \,
		\|\Phi_j(u )\|_{L^\infty(\Tor^d)}\|\nabla u-\nabla v\|_{H^{s-1,q}(\Tor^d; \R^d)}
		\,+\, 
		\| \Phi_j(u ) \|_{H^{\tau,\zeta}(\Tor^d)}\|\nabla u-\nabla v\|_{H^{s-1,q}(\Tor^d; \R^d)}
		\\&\quad\lesssim \,
		\bigl(
		\|\Phi_j(u )\|_{L^\infty(\Tor^d)}\,+\, 	\| \Phi_j(u ) \|_{H^{\tau,\zeta}(\Tor^d)}
		\bigr)
		\|\nabla u-\nabla v\|_{H^{s-1,q}(\Tor^d; \R^d)}
		\\&\quad
		\le \,L_{j,n} \|u-v\|_{\Xap}
	\end{align}
	and 
	\begin{align}&
		\|(\Phi_j(u)- \Phi_j(v))\nabla v \|_{H^{s-1,q}(\Tor^d;\R^d)}
		\\&\quad\lesssim \,
		\|\Phi_j(u ) - \Phi_j(v)\|_{L^\infty(\Tor^d)}\|\nabla v\|_{H^{s-1,q}(\Tor^d; \R^d)}
		\,+\, 
		\| \Phi_j(u ) - \Phi_j(v) \|_{H^{\tau,\zeta}(\Tor^d)}\|\nabla v\|_{H^{s-1,q}(\Tor^d; \R^d)}
		\\&\quad
		\lesssim\,
		\bigl(
		\|\Phi_j(u ) - \Phi_j(v)\|_{L^\infty(\Tor^d)}\,+\, 	\| \Phi_j(u ) - \Phi_j(v) \|_{H^{\tau,\zeta}(\Tor^d)}
		\bigr)\|v \|_{H^{s,q}(\Tor^d)}
		\\&\quad
		\lesssim\,
		L_{j,n} \|u-v \|_{\Xap}
	\end{align}
	using once more Proposition \ref{prop:comp}~\eqref{item:prop_comp_H}. Also in this case we deduce \eqref{Eq112} and the growth estimate \eqref{Eq111} can be obtained by inserting $v=0$ in \eqref{Eq112} and using that $F^{(j)}(0)=0$.
\end{proof}

\begin{lemma}\label{Lemma_gj}
	For all $j,n\in \N$ there exist constants $C_{j,n}, L_{j,n}\in (0,\infty)$ such that
	\begin{enumerate}[(i)]
		\item \label{Item11} if $s+1-\frac{d}{q}>0$, then
		\begin{align}&\label{Eq123}
			\|G^{(j)}[u]\| _{\gamma(\ell^2(\N),X_{1/2})} \,\le \, C_{j,n} \|u\|_{X_{{3}/{4}}} 
			,
			\\ \begin{split}&
				\label{Eq114}
				\| G^{(j)}[u] -G^{(j)}[v]\|_{\gamma(\ell^2(\N), X_{1/2})} \,\le \, L_{j,n} \bigl(\|u-v\|_{X_{{3}/{4}}}\,+\, \|u-v\|_{\Xap} \bigl(
				\|u\|_{X_{3/4}}\,+\, 	\|v\|_{X_{3/4}}
				\bigr) \bigr)
				,
			\end{split}
		\end{align}
		\item \label{Item12} and if $s+1-\frac{d}{q}\le 0$, then
		\begin{align}&\label{Eq124}
			\|G^{(j)}[u]\| _{\gamma(\ell^2(\N), X_{1/2})} \,\le \, C_{j,n} \|u\|_{\Xap}
			,
			\\& \label{Eq115}
			\| G^{(j)}[u] -G^{(j)}[v]\|_{\gamma(\ell^2(\N), X_{1/2})} \,\le \, L_{j,n} \|u-v\|_{\Xap}
		\end{align}
	\end{enumerate}
	for all $u,v\in \Xap$ with $\|u\|_{\Xap}, \|v\|_{\Xap}\le n$.
\end{lemma}
\begin{proof}
	To prove the assertions concerning the local Lipschitz estimates, we introduce the operator
	\begin{equation}
		H^{(j)}[u] \colon H^{s+1,q}(\Tor^d;\R^d)\to H^{s,q}(\Tor^d), \, f\mapsto \div(g_j(u)f )
	\end{equation}
	so that $G^{(j)}[u]  = H^{(j)}[u]\circ\Psi $ for the operator $\Psi$ introduced below \eqref{Eq140}. Therefore, by the ideal property \cite[Theorem 9.1.10]{Analysis2} of $\gamma$-radonifying operators we deduce that
	\begin{align}&
		\| G^{(j)}[u] -G^{(j)}[v]\|_{\gamma(\ell^2(\N), H^{s,q}(\Tor^d))} \,\le \, 
		\| H^{(j)}[u] -H^{(j)}[v]\|_{\LLL (H^{s+1,q}(\Tor^d;\R^d), H^{s,q}(\Tor^d))} \|\Psi\|_{\gamma(\ell^2(\N), H^{s+1,q}(\Tor^d;\R^d))}
	\end{align}
	and recall that the latter term is finite by Assumption \ref{Assumptions_noise_local}.
	First, we treat the case from \eqref{Item11}, namely that $s+1-\frac{d}{q}>0$. Then we can apply Proposition \ref{prop:para}~\eqref{item:prop_para_1} to obtain
	\begin{align}&
		\| (H^{(j)}[u] -H^{(j)}[v]) f\|_{ H^{s,q}(\Tor^d)}
		\\&
		\quad\lesssim\, \|(g_j(u)  - g_j(v)) f\|_{ H^{s+1,q}(\Tor^d;\R^d)}
		\\&\quad
		\lesssim\, \|g_j(u)  - g_j(v)\|_{ H^{s+1,q}(\Tor^d)} \|f\|_{L^\infty(\Tor^d;\R^d)}
		\,+\, \|g_j(u)  - g_j(v)\|_{L^\infty(\Tor^d)}   \|f\|_{ H^{s+1,q}(\Tor^d;\R^d)}
		\\&\quad\lesssim \,  \|g_j(u)  - g_j(v)\|_{H^{s+1,q}(\Tor^d)}   \|f\|_{ H^{s+1,q}(\Tor^d;\R^d)}.
	\end{align}
	By Proposition \ref{prop:comp}~\eqref{item:prop_comp_H}, we can estimate 
	\begin{align}&
		\|g_j(u)  - g_j(v)\|_{H^{s+1,q}(\Tor^d)}\,\le \, L_{j,n}\bigl( \|u-v\|_{H^{s+1,q}(\Tor^d)}\,+\, \|u-v\|_{\Xap} \bigl(\|u\|_{H^{s+1,q}(\Tor^d)}\,+\, \|v\|_{H^{s+1,q}(\Tor^d)}\bigr)\bigr)
	\end{align}
	using the embedding $\Xap \hookrightarrow L^\infty(\Tor^d)$ implied by \eqref{Eq104}. Hence, \eqref{Eq114} follows. 
	
	Next, we assume that $s+1-\frac{d}{q}\le 0$ as in  \eqref{Item12}  and recall that 
	$s+1>0$ by Assumption \ref{Assumptions_coefficients}. Thus, we can use again Proposition \ref{prop:para}~\eqref{item:prop_para_1} to estimate
	\begin{align}\begin{split}\label{Eq122}
			&
			\| (H^{(j)}[u] -H^{(j)}[v]) f\|_{ H^{s,q}(\Tor^d)}
			\\&\quad 
			\lesssim\, \|g_j(u)-g_j(v)\|_{L^\infty(\Tor^d)} \|f\|_{H^{s+1,q}(\Tor^d;\R^d)}\,+\, 
			\|g_j(u)-g_j(v)\|_{H^{s+1,l}(\Tor^d)} \|f\|_{L^r(\Tor^d;\R^d)}
		\end{split}
	\end{align}
	for $l\in (1,\infty)$, $r\in (1,\infty]$ subject to \eqref{Eq109}. As in the proof of Lemma \ref{Lemma_aj} we find that $l,r$ can be chosen such that $\Xap\hookrightarrow H^{s+1,l}(\Tor^d)$ and $H^{s+1,q}(\Tor^d)\hookrightarrow L^r(\Tor^d)$.
	Indeed for each $\epsilon>0$, we can choose $r$ such that
	\[
	s+1-\tfrac{d}{q}\,>\,
	\tfrac{-d}{r}\,>\, s+1-\tfrac{d}{q}-\epsilon.
	\]
	Then we have 
	\[
	\tfrac{d}{l}\,>\, s+1-\epsilon\quad \iff\quad s+1 -  \tfrac{d}{l}\,<\,\epsilon 
	\]
	due to \eqref{Eq109}. Invoking additionally \eqref{Eq104} and choosing $\epsilon$ smaller than \eqref{Eq121} yields that also $\Xap\hookrightarrow H^{s+1,l}(\Tor^d)$. We estimate \eqref{Eq122} further by
	\[
	\|g_j(u)-g_j(v)\|_{L^\infty(\Tor^d)\cap H^{s+1,l}(\Tor^d)} \|f\|_{H^{s+1,q}(\Tor^d;\R^d)}
	\]
	and since $u\mapsto g_j(u)$ is locally Lipschitz continuous  in $L^\infty(\Tor^d)\cap  H^{s+1,l}(\Tor^d)$ by Proposition \ref{prop:comp}~\eqref{item:prop_comp_H} the embedding 
	$\Xap\hookrightarrow L^\infty(\Tor^d)\cap  H^{s+1,l}(\Tor^d)$ yields \eqref{Eq115}.
	The corresponding growth estimates \eqref{Eq123} and \eqref{Eq124} follow by inserting $v=0$ in \eqref{Eq114} and \eqref{Eq115} and using that $G^{(j)}[0]=0$.
\end{proof}
\begin{remark}\label{Rem1}
	We convince ourselves that the assertions from Lemmas \ref{Lemma_aj}--\ref{Lemma_gj} imply \cite[Hyptohesis (H')]{AV19_QSEE_1} for fixed $j\in \N$. Indeed, Lemma \ref{Lemma_aj} yields \cite[Hypothesis (HA)]{AV19_QSEE_1} and Lemma \ref{Lemma_fj} yields 
	the estimate in terms of the trace space from \cite[Hypothesis (HF')]{AV19_QSEE_1}. Similarly, if $s+1-\frac{d}{q}\le 0$, Lemma \ref{Lemma_gj} \eqref{Item12} implies the estimate on the trace part from \cite[Hypothesis (HG')]{AV19_QSEE_1}. If instead $s+1-\frac{d}{q}> 0$, Lemma \ref{Lemma_gj} \eqref{Item11} yields 
	\begin{align}&
		\|G^{(j)}[u]\| _{\gamma(\ell^2(\N),X_{{ 1}/{2}})} \,\le \, C_{j,n} \sum_{l=1}^2 	(1+\|u\|_{X_{\vp_l}}^{\rho_l})\|u\|_{X_{\beta_l}},
		\\&
		\|G^{(j)}[u] -G^{(j)}[v] \| _{\gamma(\ell^2(\N),X_{{ 1}/{2}})} \,\le \, L_{j,n} \sum_{l=1}^2 	(1+\|u\|_{X_{\vp_l}}^{\rho_l}+\|v\|_{X_{\vp_l}}^{\rho_l})\|u-v\|_{X_{\beta_l}}
	\end{align}
	for  $u,v\in \Xap$ with $\|u\|_{\Xap}, \|v\|_{\Xap}\le n$,
	if we choose 
	\begin{equation}
		\vp_l \,\in \, \bigl(\max \bigl\{ 1-\tfrac{1+\kappa}{p},\tfrac{3}{4}\bigr\},1\bigr),\quad l\in \{1,2\},
	\end{equation}
	$\beta_1 =  \vp_1$, $\rho_1=0$, $\beta_2\in (1-\frac{1+\kappa}{p},\vp_2]$ close to $1-\frac{1+\kappa}{p}$ and $\rho_2=1$. This implies the estimate on the (possibly) critical part from \cite[Hypothesis (HG')]{AV19_QSEE_1}. In particular, we have
	\begin{equation}\label{Eq69}
		\rho_l\bigl(\vp_l -  1+\tfrac{1+\kappa}{p}\bigr) \,+\,\beta_l \, <\, 1,
		\quad l\in \{1,2\}
	\end{equation}
	for sufficiently small $\beta_2$ and are therefore in the subcritical regime of the theory developed in  \cite{AV19_QSEE_1, AV19_QSEE_2}.
\end{remark}

\subsection{Stochastic maximal regularity of thin-film type operators}
\label{Sec_SMR} Next to the local Lipschitz estimates established in the previous subsection, the local well-posedness theory from \cite{AV19_QSEE_1,AV19_QSEE_2} also requires optimal regularity estimates for linear problems of the form
\begin{equation}
	\label{eq:diffAB_s}
	\begin{cases}
		\dd u\, +\,\A(u)\,\dd t\,=\,f\, \dd t+ g\, \dd W,\quad t\in [t_0,T],\\
		u(t_0)\,=\,u_{t_0},
	\end{cases}
\end{equation}
called stochastic maximal regularity. 
Here $\A:[t_0,T]\times \O\to \calL(X_1,X_0)$ is strongly progressively measurable, $X_0,X_1$ are UMD spaces of type 2 (see, e.g., \cite[Chapter 7]{Analysis2}) and $T<\infty$, $t_0\in [0,T]$ are fixed. 

The purpose of this subsection is to verify that if $(p,\a,s,q)$ are admissible,
\begin{align}\label{Eq12}
	\A(u)\,=\, \div(a\nabla \Delta u)
\end{align}
for a positive and $\Progress$-measurable coefficient $a:[0,\infty)\times \O\to \Xap$ and $(X_0,X_1)$ is as in Subsection \ref{ss:notation}, then the problem \eqref{eq:diffAB_s} indeed admits stochastic maximal regularity estimates, which can be seen as a generalization of \cite[Theorem 5.2]{AV21_max_reg_torus} to fourth-order operators. 
Let us note that the time dependence of the coefficient $a$ in \eqref{Eq12} appears naturally when considering the nonlinear problem \eqref{Eq102} by choosing $a=m(u)$. One of the key points here is that we will allow coefficients that are only measurable in time. This will be crucial in the proof of the blow-up criteria of Proposition \ref{prop:blow_up_criteria} (cf., \eqref{Eq30} below) used in the proof of the global well-posedness result in Subsection \ref{ss:global_1d_intro}.
\smallskip

Before going into the details, we recall some definitions. 
For a stopping time $\tau\colon \Omega \to [t_0,T]$, an initial value
$u_{t_0}\in L^0_{\F_{t_0}}(\O;X_0)$ and inhomogeneities $f\in L^0_{\Progress}(\O;L^1({t_0},\tau;X_0))$ and $g\in L^0_{\Progress}(\O;L^2({t_0},\tau;\g(l^{2}(\N),X_{1/2})))$   of \eqref{eq:diffAB_s}, we call a strongly progressive measurable map $u:\llbracket {t_0},\tau\rrbracket\to X_1$ a \textit{strong solution} to \eqref{eq:diffAB_s} on $\llbracket {t_0},\tau\rrbracket$, if $u\in L^0(\O;L^2({t_0},\tau;X_1))$ and
\begin{equation*}
	u(t)\,-\,u_{{t_0}}\,+\,\int_{t_0}^t \A(u)\, \dd r\,=\,\int_{{t_0}}^t
	f\, \dd r\,+\, \int_{t_0}^t g\,\dd W_r
\end{equation*}
a.s.\ for all $t\in [{t_0},\tau]$.
\begin{definition}[Stochastic maximal regularity]
	\label{def:SMRgeneralized}Let $\hat{\kappa}\in [0,\frac{p}{2}-1)$ for $p>2$ and $\hat{\kappa} = 0$ for $p=2$ be possibly different from $\kappa$ and  $\|\A\|_{\calL(X_1,X_0)}\le
	C_{\A}$ for some constant $C_{\A}<\infty$. We write $\A\in \MRtaszh$ if for every
	\begin{equation}\label{Eq77}
		f\in  L^p_{\Progress}( (t_0,T)\times\Omega,w_{\hat{\a}}^{t_0};X_0),
		\quad\text{ and }\quad
		g\in L^p_{\Progress}((t_0,T)\times\Omega,w_{\hat{\a}}^{t_0};\g(H,X_{1/2}))
	\end{equation}
	there exists a strong solution $u$ to \eqref{eq:diffAB_s} with initial value $u_{t_0}=0$ such that $u\in L^p_{\Progress}((t_0,T)\times\Omega, w_{\hat{\a}}^{t_0};X_{1})$, and moreover for all stopping time $\tau\colon \Omega \to [t_0,T]$ and strong solutions $u\in L^p_{\Progress}((t_0,\tau)\times \O,w_{\hat{\a}}^{t_0};X_1)$ to \eqref{eq:diffAB_s} with $u_{t_0}=0$ the estimate
	\begin{equation}\label{Eq201}
		\begin{aligned}
			\|u\|_{L^p((t_0,\tau)\times\Omega,w_{\hat{\a}}^{{t_0}};X_1)}
			\,\lesssim\,\|f\|_{L^p((t_0,\tau)\times\Omega,w_{\hat{\a}}^{{t_0}};X_0)}
			\,+\,\|g\|_{L^p((t_0,\tau)\times\Omega,w_{\hat{\a}}^{{t_0}};\g(H,X_{1/2}))}
		\end{aligned}
	\end{equation}
	holds,
	where the implicit constant is independent of $f$, $g$ and $\tau$. In addition:
	\begin{enumerate}[{\rm(1)}]
		\item\label{it:SMR_regularity_0} If $p\in (2,\infty)$, we write $\A\in \MRtash$ if $\A\in \MRtaszh$ and
		\begin{align}\begin{split}\label{Eq141}
				&
				\|u\|_{L^p(\O;H^{\theta,p}(t_0,T,w_{\hat{\a}}^{{t_0}};X_{1-\theta}))}
				\,\lesssim\,\|f\|_{L^p((t_0,T)\times\Omega,w_{\hat{\a}}^{{t_0}};X_0)}\,+\,\|g\|_{L^p((t_0,T)\times\Omega,w_{\hat{\a}}^{{t_0}};\g(H,X_{1/2}))}
			\end{split}
		\end{align}
	 	for each $\theta\in [0,\frac{1}{2})$,
		where the implicit constant is independent of $f$ and $g$.
		\item If $p=2$, we write $\A\in \MRttwoszero$ if $\A\in \MRttwosz$ and it holds
		\begin{align}\begin{split}
				&\label{Eq142}
				\|u\|_{ L^2(\O;C([{t_0},T];X_{1/2}))}
				\,\lesssim\, \|f\|_{L^2(({t_0},T)\times\Omega ;X_0)}\,+\,\|g\|_{L^2(({t_0},T)\times\Omega;\g(H,X_{1/2}))},
			\end{split}
		\end{align}
		where the implicit constant is independent of $f$ and $g$.
	\end{enumerate}
\end{definition}
For a definition of the Sobolev space $H^{\theta,p}(t_0,T,w_{\hat{\a}}^{{t_0}};X_{1-\theta})$ we refer to Appendix \ref{app:Sobolev_spaces} and we remark that the choice $u_{t_0}=0$ in the above definition is not essential. Indeed, stochastic maximal $L^p$-regularity estimates with zero initial data directly imply corresponding ones with non-trivial data \cite[Proposition 3.10]{AV19_QSEE_1}.

If $\A\in \MRtash$, we define $C_{\mathcal{A}}^{\det,\theta,p,\hat{\a}}({t_0},T)$ as the smallest implicit constant such that \eqref{Eq201} for $\theta=0$ and \eqref{Eq141}--\eqref{Eq142} for $\theta>0$ holds for all $	f\in  L^p_{\Progress}( ({t_0},T)\times\Omega,w_{\hat{\a}}^{{t_0}};X_0) $ and $g=0$, i.e., the constant accounting for the inhomogeneity $f\,\dd t$ on the right-hand side of \eqref{eq:diffAB_s}. Analogously, we define $C_{\mathcal{A}}^{\stoc,\theta,p,\hat{\a}}({t_0},T)$ as the smallest implicit constant such that \eqref{Eq201}--\eqref{Eq142}, respectively, holds for all $g\in L^p_{\Progress}(({t_0},T)\times\Omega,w_{\hat{\a}}^{t_0};\g(H,X_{1/2}))$ and $f=0$, accounting for the fluctuating part of the right-hand side. For $\ell\in \{\det,\stoc\}$, we set finally
\begin{align}\label{eq:constKellthetapa}
	K^{\ell,\theta,p,\hat{\a}}_{\A}({t_0},T)\,:=\,C_{\A}^{\ell,\theta,p,\hat{\a}}({t_0},T)\,+\,C_{\A}^{\ell,0,p,\hat{\a}}({t_0},T).
\end{align}

\smallskip
We are ready to state the main result of this subsection.

\begin{theorem}[Stochastic maximal regularity -- Thin-film type operators]
\label{Theorem_SMR}
Let the parameters  $(p,\a,s,q)$ be admissible (i.e., Assumption \ref{Assumptions_coefficients} holds). Assume that the mapping $a:[t_0,T]\times \O\to \Xap$ is $\Progress$-measurable and there exist positive constants $\lambda,\mu>0$ such that, for a.e.\ $(t,\om)\in [t_0,T]\times \O$,  
\begin{equation}
\label{eq:assumption_smr_smoothness_coefficients}
\lambda\leq \inf_{\T^d} a(t,\om,\cdot)\quad \text{ and }\quad
 \|a(t,\om,\cdot)\|_{\Xap}\leq \mu .
\end{equation}
Then $\A$ defined in \eqref{Eq12} satisfies $\A\in \MRtash$ and 
	\begin{equation}
		\max\bigl\{K^{\det,\theta,p,\hat{\a}}_{\A}({t_0},T), K^{\stoc,\theta,p,\hat{\a}}_{\A}({t_0},T)\}\,\lesssim_{\lambda,\mu,\theta,T}\,1
	\end{equation}
	for each $\theta\in [0,\frac{1}{2})$ and $\hat{\a}\in \{0,\a\}$.
\end{theorem}

Note that no regularity of the coefficient $a$ is assumed w.r.t.\ the time variable.
From Definition \ref{def:SMRgeneralized}, we have the inclusion $  \MRtash\subseteq\MRtaszh$, where the equality $ \MRtaszh= \MRtash$ holds however provided that $ \MRtash\neq \emptyset$, see \cite[Proposition 3.8]{AV19_QSEE_1}. According to \cite[Theorem 1.2]{MaximalLpregularity}, the periodic version of \cite[Theorem 10.2.25]{Analysis2} applied to the operator
\begin{equation}
\label{eq:bilaplacian_H_infty_calculus}
\mathcal{A}_0:=\Delta^2:H^{s+2,q}(\T^d)\subseteq H^{s-2,q}(\T^d)\to H^{s-2,q}(\T^d),
\end{equation} 
implies that $\MRtash\neq \emptyset$ in our setting, since $X_1=H^{s+2,q}(\T^d)$ and $X_0=H^{s-2,q}(\T^d)$. 
Therefore, to prove Theorem \ref{Theorem_SMR} it suffices to show that $\mathcal{A}\in \MRtaszh$. 

The proof of the above result is given at the end of this subsection, and some preparation is required. We begin by considering the situation of a spatially constant coefficient $a$.

\begin{lemma}\label{Lemma_SMR_bilplace}
	Assume that $a(t,x) = \bar{a}(t)$ for a progressive measurable random variable $\bar{a}\colon [t_0,T]\times \Omega \to \R$ satisfying $\lambda\le \bar{a} \le \lambda^{-1}$ for some $\lambda\in (0,1)$. Then $\mathcal{A} \in \MRtaszh$ and 
	\begin{align}\label{Eq19}
		\max\bigl\{ C^{\det,0,p,\hat{\a}}_{\A}({t_0},T), C^{\stoc,0,p,\hat{\a}}_{\A}({t_0},T)\bigr\}\, \lesssim_{\lambda, T} \,1
	\end{align}
	for $\hat{\kappa}\in \{0,\a\}$. 
\end{lemma}

\begin{proof}
The claim can be proved by applying existing results in the literature, but to keep this manuscript as self-contained as possible, we base our line of argument on the results from Appendix \ref{app:MR_Lp_time_measurable}.

First we note that it is enough to show the claim for $s=2$, which follows from the fact that $a$ is $x$-independent and $(1-\Delta)^{r/2}: H^{r+t,\zeta}(\T^d)\to H^{t,\zeta}(\T^d)$ is an isomorphism for all $r,t\in \R$ and $\zeta\in (1,\infty)$. Next, we mimic the transference argument in \cite[Theorem 3.9]{VP18}, where we can avoid measurability concerns 
due to dealing with an explicit SPDE \eqref{eq:diffAB_s}, cf. \cite[Theorem 3.9, Step 1]{VP18}. Indeed, by the \emph{method of continuity} \cite[Proposition 3.13]{AV19_QSEE_2} applied to the family of operators
\begin{equation}
	\mathcal{A}_r(u)\,=\, r \mathcal{A}(u)\,+\, (1-r)\mathcal{A}_0u\,=\, 
	(r\bar{a}+(1-r)) \Delta^2 u  , \qquad r\in [0,1],
\end{equation}
we obtain the existence part of Definition \ref{def:SMRgeneralized} as soon as we establish the estimate \eqref{Eq201} since the coefficients $a_r=ra+(1-r)$ satisfy our assumption with the same, uniform $\lambda$.

In the case $p=q=2$ and $\a=0$, \eqref{Eq201} follows from an application of the It\^o's formula, see \cite[Lemma 4.1]{AV24_variational} and \cite[Chapter 4]{LiuRock}, so  we only consider $p>2$ and $\kappa\in [0,\frac{p}{2}-1)$ in the following (recall that $p=2$ is allowed only if $q=2$, see \eqref{Eq1} in Assumption \ref{Assumptions_coefficients}).
To proceed in the latter case, we notice moreover that given any $g\in  L^p_{\Progress}((t_0,\tau )\times\Omega,w_{\hat{\a}}^{t_0};\g(H,X_{1/2}))$ there exists a strong solution $v$ to 
\begin{equation}
	\label{eq:diffAB_s_0_new}
	\begin{cases}
		\dd v\, +\,\A_0(v)\,\dd t\,=\, g\, \dd W,\quad t\in [t_0,\tau ],\\
		v(t_0)\,=\,0,
	\end{cases}
\end{equation}
satisfying the estimate 
\begin{equation}\label{Eq73}
	\|v\|_{L^p((t_0,\tau)\times\Omega,w_{\hat{\a}}^{{t_0}};X_1)}\,\lesssim_T \,\|g\|_{L^p((t_0,\tau)\times\Omega,w_{\hat{\a}}^{{t_0}};\g(H,X_{1/2}))},
\end{equation}
where $\A_0=\Delta^2\in \MRtash$  by  the discussion preceding \eqref{eq:bilaplacian_H_infty_calculus}. Consequently, if $u\in L^p_{\Progress}((t_0,\tau)\times \O,w_{\hat{\a}}^{t_0};X_1)$ is any strong solution to the original equation \eqref{eq:diffAB_s} on $\llbracket {t_0},\tau\rrbracket$ with $u_{t_0} = 0$ and $f \in L^p_{\Progress}((t_0,\tau)\times \O,w_{\hat{\a}}^{t_0};X_0)$, then we find that the difference $w = u - v$ solves a.s.\ the random PDE
\begin{equation}
	\label{eq:diffAB_s_01_new}
	\begin{cases}
		\partial_t w\, +\,\A(w)\,=\, (f+ (\A_0-\A)v), \quad t\in [t_0,T],\\
		w(t_0)\,=\,0.
	\end{cases}
\end{equation}
 Now Proposition \ref{prop:deterministic_SMR} yields that $w$ satisfies 
\[
\|w\|_{L^p((t_0,\tau)\times \O,w_{\a}^{t_0};X_1)} \,\lesssim_{\lambda, T} \, \|f\|_{ L^p((t_0,\tau)\times \O,w_{\a}^{t_0};X_0)} \,+\, 
\|(\mathcal{A}_0 - \mathcal{A})v \|_{ L^p((t_0,\tau)\times \O,w_{\a}^{t_0};X_0)},
\]
which together with \eqref{Eq73} and the continuity of $\A,\A_0\colon X_1 \to X_0$ yields the desired 
\begin{equation}
	\label{eq:estimate_to_be_proven_easy_case}
	\|u\|_{L^p((t_0,\tau)\times\Omega,w_{\hat{\a}}^{{t_0}};X_1)}\,\lesssim_{\lambda, T} 
	\,\|f\|_{L^p((t_0,\tau)\times\Omega,w_{\hat{\a}}^{{t_0}};X_0)}
	\,+\,\|g\|_{L^p((t_0,\tau)\times\Omega,w_{\hat{\a}}^{{t_0}};\g(H,X_{1/2}))}
\end{equation}
for the original   $u = v+w$.
\end{proof}

Using estimates similar to the ones in the proof of Lemma \ref{Lemma_aj}, we can also include small perturbations of a spatially constant coefficient. 
\begin{lemma}\label{Lemma_SMR_small_pert}
	Let $\bar{a}$ and $\lambda$ be as in Lemma \ref{Lemma_SMR_bilplace} and let $a$ be as in the statement of Theorem \ref{Theorem_SMR}. Then there exists some $\epsilon>0$ depending on $\lambda,T$ such that if 
	\begin{equation}\label{Eq21}
		\esssup_{[t_0,T]\times \O}\|a \,-\, \bar{a} \|_{L^\infty(\Tor^d)}\,<\,\epsilon,
	\end{equation}
	then we have $\mathcal{A} \in \MRtaszh$ with 
	\begin{equation}
	\max\bigl\{	C^{\det,0,p,\hat{\kappa}}_{\mathcal{A}}({t_0},T),C^{\stoc,0,p,\hat{\kappa}}_{\mathcal{A}}({t_0},T) \bigr\}\,\lesssim_{\lambda,\mu,T}\, 1
	\end{equation}
	for $\hat{\a}\in \{0,\a\}$.
\end{lemma}
\begin{proof}
	By repeating the proof of the estimate \eqref{Eq110} for the operator $\A$, one obtains that, a.e.\ on $[t_0,T]\times \O$,
	\begin{equation}\label{Eq25}
		\|\A\|_{\LLL(X_1,X_0)}\,\lesssim\,\|a\|_{\Xap}\,\le \, \mu.
	\end{equation}
	By the previous Lemma \ref{Lemma_SMR_bilplace}, we have  $\bar{a}\Delta^2\in \MRtaszh$ and $\bar{a}\Delta^2\in \MRtaszh$. We define $C_{\lambda,T}^*$ as the maximum of the corresponding upper bounds on $C^{\det,0,p,\hat{\a}}_{\A}({t_0},T)$ and $C^{\det,0,p,0}_{\A}({t_0},T)$ from \eqref{Eq19}. Then, since we can write 
	\begin{equation}
		\mathcal{A}(u)\,=\, \div( \bar{a}\nabla \Delta u) \,+\,  \div( (a-\bar{a})\nabla \Delta u)
	\end{equation}
	the claim follows by the perturbation result \cite[Theorem 3.2]{AV21_max_reg_torus} as soon as we can verify that
	\begin{equation}\label{Eq22}
		\| \div( (a-\bar{a})\nabla \Delta u)\|\,\le \,({2C_{\lambda,T}^*})^{-1}\|u\|_{X_1} \,+\, L\|u\|_{X_0}
	\end{equation}
 for a suitable constant $L\in (0,\infty)$. To this end, we assume \eqref{Eq21} and obtain
	\begin{align}
		\|\div( (a-\bar{a})\nabla \Delta u) \|_{X_0}\,\le \, \| a-\bar{a} \|_{L^\infty(\Tor^d)} \|u\|_{X_1}\,<\, \epsilon \|u \|_{X_1}
	\end{align}
	for $s-1=0$ implying \eqref{Eq22} for sufficiently small $\epsilon>0$. For $s-1>0$, we derive
	\begin{align}
		\|\div( (a-\bar{a})\nabla \Delta u) \|_{X_0}\,\lesssim\, 
		 \|a-\bar{a}\|_{L^\infty(\Tor^d)} \|\nabla\Delta u\|_{H^{s-1,q}(\Tor^d;\R^d)}\,+\, \|a-\bar{a}\|_{H^{s-1,l}(\Tor^d)}\|\nabla \Delta u\|_{L^r(\Tor^d;\R^d)}
	\end{align}
	for  $l,r$ as in the proof of Lemma \ref{Lemma_aj}, analogously to \eqref{Eq108}. We observe moreover that the embedding $H^{s-1,q}(\Tor^d)\hookrightarrow L^r(\Tor^d)$ is not sharp and hence it also holds $H^{s-1-\delta,q}(\Tor^d)\hookrightarrow L^r(\Tor^d)$ for a sufficiently small $\delta>0$. Then, using \eqref{Eq21}, that $\Xap\hookrightarrow H^{s-1,l}(\Tor^d)$, the interpolation inequality 
	\begin{equation}\label{Eq143}
		\|u\|_{X_{1-{\delta}/{4}}}\,\le\, \|u\|_{X_{0}}^{{\delta}/{4}} \|u\|_{X_{1}}^{1-{\delta}/{4}}
	\end{equation} and Young's inequality, we can estimate
	\begin{align}&
		\|\div( (a-\bar{a})\nabla \Delta u) \|_{X_0}\,\lesssim \, \epsilon \|u\|_{X_1}\,+\, \|a-\bar{a}\|_{\Xap}\|u\|_{X_{1-{\delta}/{4}}}
		\\&\quad \le \, \epsilon \|u\|_{X_1} \,+\, C_{\lambda,\mu}  \|u\|_{X_{0}}^{{\delta}/{4}} \|u\|_{X_{1}}^{1-{\delta}/{4}}
		\,\le \, 2\epsilon \|u\|_{X_1} \,+\, C_{\delta,\epsilon,\lambda,\nu}\|u\|_{X_0}.
	\end{align}
	Therefore, \eqref{Eq22} follows also in this case as long as $\epsilon>0$ is small enough.
	For $s-1<0$, we proceed again as in Lemma \ref{Lemma_aj} and apply Proposition \ref{prop:para}~\eqref{item:prop_para_3} to obtain
	\begin{align}&
		\|\div( (a-\bar{a})\nabla \Delta u) \|_{X_0}\,\lesssim \, \|a-\bar{a}\|_{L^\infty(\Tor^d)} \|\nabla \Delta u\|_{H^{s-1,q}(\T^d;\R^d)}\,+\, \|a-\bar{a}\|_{H^{\tau,\zeta}(\Tor^d)}\|\nabla \Delta u\|_{H^{s-1-\delta,q}(\Tor^d;\R^d)}
	\end{align}
	for $\tau,\zeta$ as in Lemma \ref{Lemma_aj} and some $\delta>0$. We remark that we used the slightly worse estimate \eqref{Eq23} previously because it was sufficient to verify Lemma \ref{Lemma_aj}. Employing again \eqref{Eq21}, the inequality \eqref{Eq143}, Young's inequality and the embedding $\Xap\hookrightarrow H^{\tau,\zeta}(\T^d)$ we conclude
	\begin{align}
		\|\div( (a-\bar{a})\nabla \Delta u) \|_{X_0}\,\lesssim \, 2\epsilon \|u\|_{X_1} \,+\, C_{\delta,\epsilon,\lambda,\mu}\|u\|_{X_0}.
	\end{align}
	The desired estimate \eqref{Eq22} follows again for sufficiently small $\epsilon>0$.
\end{proof}
Before giving a proof of stochastic maximal regularity in the general situation, we first provide an a a-priori estimate on sufficiently small time intervals. The proof relies on a spatial localization procedure known as freezing the coefficients, which allows us to apply Lemma \ref{Lemma_SMR_small_pert} locally in space.
\begin{lemma}
	\label{l:estimates_small_interval}
	Let $a$ be as in the statement of Theorem \ref{Theorem_SMR}.
	Then there exists a time $\tT>0$ depending on $\lambda,  \mu,T$ for which the following holds.
	For any $t\in [{t_0},T]$, $\hat{\a}\in \{0,\a\}$,  stopping time $\tau\colon \O\to [t,T\wedge(t+\tT)]$,  $f\in L^p_{\Progress}((t,\tau)\times\Omega,w_{\hat{\a}}^t;X_0)$, $g\in L^p_{\Progress}((t,\tau)\times\Omega,w_{\hat{\a}}^t;\g(H,X_{{1}/{2}}))$ and  strong solution $u\in L^p_{\Progress}((t,\tau)\times\Omega,w_{\hat{\a}}^t;X_1)$ to 
	\begin{equation}
		\label{Eq150}
		\begin{cases}
			\dd u\, +\,\mathcal{A}(u)\,\dd r\,=\,f\, \dd r+ g\, \dd W,\quad r\in [t,\tau],\\
			u(t)\,=\,0,
		\end{cases}
	\end{equation}
	it holds
	\begin{align}\label{eq:lemaprioriTorus}
		\|u\|_{L^p((t,\tau)\times\Omega,w_{\hat{\a}}^t;X_1)}\,\lesssim_{\lambda,  \mu,T}\, \|f\|_{L^p((t,\tau)\times\Omega,w_{\hat{\a}}^t;X_0)}\,+\,\|g\|_{L^p((t,\tau)\times\Omega,w_{\hat{\a}}^t;\g(H,X_{{1}/{2}}))}.
	\end{align}
\end{lemma}
\begin{proof}We follow the strategy employed to prove \cite[Lemma 5.4]{AV21_max_reg_torus} and split the proof into three steps to improve its readability.
We do not display the time dependence of $a$ as it does not play any role here.
	
	\textit{Construction of suitable extension operators.}
	We first derive the existence of extension operators with a uniform operator norm, which can be seen as a variant of \cite[Lemma 5.8]{AV21_max_reg_torus} for Besov spaces. The main difference compared to \cite[Lemma 5.8]{AV21_max_reg_torus} is that constant functions are contained in H\"older spaces on $\R^d$, but not necessarily in $B_{q,p}^s(\R^d)$. Let $E$ be Stein's total extension operator from the unit ball $B_{\R^d}(0,1)$  to $\R^d$ as introduced in \cite[Theorem 5, p.181]{Stein70}, then 
	\begin{align}&
		E\colon W^{l,r}(B_{\R^d}(0,1))\to W^{l,r}(\R^d)
	\end{align}
	continuously for each $l\in \N_0$ and $r\in [1,\infty]$. We let $\vp\in C_c^\infty(\R^d)$ such that $\vp=1$ on $B_{\R^d}(0,1)$ and $\vp=0$ outside of $B_{\R^d}(0,2)$ and define
	\begin{align}
		E_{y,r}^{\R^d}[f](x)\,=\, (\vp \cdot E[ f(y+r\cdot)])\bigl(
		\tfrac{x-y}{r}
		\bigr)
	\end{align}
	for $y\in \R^d$ and $r\in (0,\frac{1}{8})$
	and obtain extension operators from $B_{\R^d}(y,r)$ to $\R^d$ with
	\begin{equation}
		\| E_{y,r}^{\R^d}\|_{\LLL(L^\infty(B_{\R^d}(y,r)), L^\infty(\R^d))}
	\end{equation}
	independent of $y$ and $r$. Identifying $ \T^d \eqsim (-1/2,1/2]^d$ we set now 
	\begin{equation}\label{Eq79}
	E_{y,r}^{\Tor^d}f (x) \,=\, \bigl[ E_{0,r}^{\R^d} ( f(\cdot-y)) \bigr] \bigl(( x-y) \mod 1 \bigr)
	\end{equation}
	with the convention that $(z\mod 1 )_j  \in (-1/2,1/2]$ for $1\le j\le d$  and  $z\in \R^d$ is the componentwise modulo operation. Due to the assumption that $r<\frac{1}{8}$, the left-hand side of \eqref{Eq79}
	satisfies periodic boundary conditions and therefore 
	yields extension operators $E_{y,r}^{\Tor^d}$ from $B_{\Tor^d}(y,r)$ to $\Tor^d$ with the same properties as $E_{y,r}^{\R^d}$. To ensure that constant functions are mapped to constant functions, we define new extension operators by setting 
	\begin{align}
		\tilde{E}_{y,r}^{\Tor^d}f \,=\, E_{y,r}^{\Tor^d}\biggl[
		f-\fint_{B_{\Tor^d}(y,r)} f\, \dd x
		\biggr]\,+\,\fint_{B_{\Tor^d}(y,r)} f\, \dd x,
	\end{align}
	where $\fint_{B_{\T^d}(y,r)}:=|B_{\T^d}(y,r)|^{-1}\int$, 
	and observe that again
	\begin{equation}
		\|	\tilde{E}_{y,r}^{\Tor^d}\|_{\LLL(L^\infty(B_{\Tor^d}(y,r)),  L^\infty(\Tor^d))}
	\end{equation}
	is independent of $y$ and $r$. By real interpolation, we conclude that also
	\begin{equation}
		\tilde{E}_{y,r}^{\Tor^d} \colon B^{s+2-4\frac{1+\a}{p}}_{q,p}(B_{\Tor^d}(y,r))\to \Xap
	\end{equation} 
	continuously.
	
	\textit{Spatial localization.}
	Next, analogously to the proof of \cite[Lemma 5.4]{AV21_max_reg_torus}, we introduce the operators
	\begin{align}
		\mathcal{A}_y(u)\,=\, \div(a(y) \nabla \Delta u ),\qquad \mathcal{A}^E_{y,r}(u)\,=\,  \div(\tilde{E}_{y,r}^{\Tor^d}[ a|_{B_{\Tor^d}(y,r)} ] \nabla \Delta u).
	\end{align}
	Let $(y_\iota)_{\iota\in \mathcal{I}}\subset \Tor^d$ and $r\in (0,\frac{1}{8})$ such that $(B_{\Tor^d}(y_\iota,r))_{\iota\in \mathcal{I}}$ is a finite cover of $\Tor^d$. Moreover, let $(\vp_\iota)_{\iota\in \mathcal{I}}$ be a partition of unity subordinate to $(B_{\Tor^d}(y_\iota,r))_{\iota\in \mathcal{I}}$. We observe that for sufficiently small $\gamma>0$ we have $\Xap\hookrightarrow C^{\gamma}(\Tor^d)$ by \eqref{Eq104} and consequently
	\begin{align}
		\|a(y_\iota) - a|_{B_{\Tor^d}(y_\iota,r)} \|_{L^\infty ({B_{\Tor^d}(y_\iota,r)} )}\,\lesssim_{\mu}\,r^\gamma.
	\end{align}
	After applying the extension operators $\tilde{E}_{y_\iota,r}^{\Tor^d}$ we constructed, we conclude that
	\begin{align}&
		\bigl\|a(y_\iota) - \tilde{E}_{y_\iota,r}^{\Tor^d} [a|_{B_{\Tor^d}(y_\iota,r)}] \bigr\|_{L^\infty (\Tor^d)}\,=\, 	\bigl\|\tilde{E}_{y_\iota,r}^{\Tor^d} [a(y_\iota) - a|_{B_{\Tor^d}(y_\iota,r)}] \bigr\|_{L^\infty (\Tor^d)}\,\lesssim_{\mu}\,r^\gamma. 
	\end{align}
	Thus, if we choose $\epsilon$ as in Lemma \ref{Lemma_SMR_small_pert} according to the given constant $\lambda$, we can ensure that 
	\begin{equation}
		\bigl\|a(y) - \tilde{E}_{y_\iota,r}^{\Tor^d} [a|_{B_{\Tor^d}(y,r)}] \bigr\|_{L^\infty (\Tor^d)}\,<\, \epsilon
	\end{equation}
	if we choose $r<r_{\lambda,\mu,T}^*\in (0,\frac{1}{8})$ small enough.	Setting then 
	\begin{equation}
		C_{\lambda,\mu,T}^{*}\,=\, \max_{\iota\in \mathcal{I}}\,\,\bigl\| \tilde{E}_{y_\iota,r}^{\Tor^d}\bigr\|_{\LLL\bigl(B^{s+2-4\frac{1+\kappa}{p}}_{q,p}(B_{\Tor^d}(y_\iota,r)),\Xap\bigr)},
	\end{equation}
	we deduce that
	\begin{align}&
		\bigl\|
		\tilde{E}^{\Tor^d}_{y_\iota,r} [a|_{B_{\Tor^d}(y_\iota,r)} ]
		\bigr\|_{\Xap}\,
		\le\, C_{\lambda,\mu,T}^* \|a|_{B_{\Tor^d}(y_\iota,r)}\|_{B^{s+2-4\frac{1+\kappa}{p}}_{q,p}(B_{\Tor^d}(y_\iota,r))}\,
		\le \, C_{\lambda,\mu,T}^* \|a\|_{\Xap}\,\le \, C_{\lambda,\mu,T}^*\mu.
	\end{align}
	Therefore, the functions $\tilde{E}^{\Tor^d}_{y_\iota,r} [a|_{B_{\Tor^d}(y_\iota,r)} ]$ satisfy the assumptions of Lemma \ref{Lemma_SMR_small_pert} with uniform constants and thus $\mathcal{A}_{y_\iota,r}^E\in \MRtash$ with 
	\begin{align}\label{Eq151}
		C^{\ell,0,p,\hat{\kappa}}_{\mathcal{A}_{y_\iota,r}^E}(s,T)\,\lesssim_{\lambda,\mu,T}\, 1,\qquad \ell\in \{\det,\stoc\}.
	\end{align}
	With this at hand, we finally consider the strong solution $u$ to \eqref{Eq150} and using the partition of unity $(\vp_\iota)_{\iota\in \mathcal{I}}$, we can write $
	u\,=\, \sum_{\iota\in \mathcal{I}}u_\iota$, if we set $u_\iota = u\vp\iota$. Introducing analogously $f_\iota = f\vp_\iota$ and $g_\iota = g\vp_\iota$, we find that
	\begin{align}
		\dd u_\iota \,+\, \mathcal{A}_{y_\iota, r}^E(u_\iota)\,\dd t \,&=\, \dd  u_\iota \,+\, \mathcal{A}(u_\iota)\,\dd t\\
		&=\, \vp_\iota \bigl(\dd u \,+\, \mathcal{A}(u)\, \dd t \bigr)\,+\, [\mathcal{A},\vp_\iota] u\, \dd t\,
		=\, \bigl(f_\iota\,+\, [\mathcal{A},\vp_\iota] u\bigr)\, \dd t\,+\, g_\iota\, \dd W,
	\end{align}
	where we use the commutator notation $[\mathcal{A},\vp_\iota]u= \mathcal{A}(\vp_\iota u) - \vp_\iota\mathcal{A} u $. Hence, applying \eqref{Eq151}, we deduce that
	\begin{align}\begin{split}\label{Eq154}
			&
			\|u\|_{L^p((t,\tau)\times \O, w_{\hat{\a}}^t;X_1)}\,\le\, \sum_{\iota\in \mathcal{I}} \|u_\iota\|_{L^p((t,\tau)\times \O, w_{\hat{\a}}^t;X_1)}\\&\quad \lesssim_{\lambda,\mu,T}
			\,
			\sum_{\iota\in \mathcal{I}}\Bigl(\|[\mathcal{A},\vp_\iota] u \|_{L^p((t,\tau)\times \O, w_{\hat{\a}}^t;X_0)}\,+\,\|f_\iota\|_{L^p((t,\tau)\times\Omega,w_{\hat{\a}}^t;X_0)}\,+\,\|g_{\iota}\|_{L^p((t,\tau)\times\Omega,w_{\hat{\a}}^t;\g(H,X_{1/2}))} \Bigr)
			\\&\quad\lesssim_{\lambda,\mu,T}	\,
			\sum_{\iota\in \mathcal{I}}\Bigl(\|[\mathcal{A},\vp_\iota] u \|_{L^p((t,\tau)\times \O, w_{\hat{\a}}^t;X_0)}\Bigr)\,+\,\|f\|_{L^p((t,\tau)\times\Omega,w_{\hat{\a}}^t;X_0)}\,+\,\|g\|_{L^p((t,\tau)\times\Omega,w_{\hat{\a}}^t;\g(H,X_{1/2}))},
		\end{split}
	\end{align}
	where in the last step we used that multiplication $\vp_\iota$ is a bounded linear mapping on $H^{s+2,q}(\Tor^d)$ and $ H^{s,q}(\Tor^d)$ for each $\iota\in \mathcal{I}$. 
	
	\textit{Absorbing the commutator terms.} To conclude the proof, we show that the commutator terms $[\mathcal{A},\vp_\iota]u$ are of lower order  and expand
	\begin{align}&
		\mathcal{A}(\vp_\iota u)\,=\, \div(a\nabla \Delta (\vp_\iota u))
		\,=\, \div(a \nabla ( \vp_\iota \Delta u \,+\, 2\nabla \vp_\iota \cdot \nabla u\,+\, u \Delta \vp_\iota ))
		\\&\quad=\, 
		\div(a(\vp_\iota \nabla \Delta u\,+\, \nabla \vp_\iota \Delta u )) \,+\,
		\div(a \nabla (  2\nabla \vp_\iota \cdot \nabla u\,+\, u \Delta \vp_\iota ))
		\\&\quad=\, \vp_\iota \mathcal{A} u\,+\, a\nabla \vp_\iota \cdot \nabla \Delta u \,+\, \div(a\nabla \vp_\iota \Delta u) \,+\,\div(a \nabla (  2\nabla \vp_\iota \cdot \nabla u\,+\, u \Delta \vp_\iota ))
	\end{align}
	to conclude 
	\begin{align}\label{Eq152}
		[\mathcal{A}, \vp_\iota ] u\,=\, a\nabla \vp_\iota \cdot \nabla \Delta u \,+\, \div(a\nabla \vp_\iota \Delta u) \,+\,\div(a \nabla (  2\nabla \vp_\iota \cdot \nabla u\,+\, u \Delta \vp_\iota )).
	\end{align}
	We claim that
	\begin{equation}\label{Eq155}
		\| 	[\mathcal{A}, \vp_\iota ] u\|_{X_0}\,\lesssim_{\lambda,\mu,T}\, \|u\|_{X_{1-{\delta}/{4}}}
	\end{equation}
	for sufficiently small $\delta>0$ and estimate only the first term on the right-hand side of \eqref{Eq152}, since the others are easier to treat. Firstly, we observe that
	\begin{equation}\label{Eq156}
		\|a\nabla \vp_\iota \|_{B_{q,p}^{s+2-4\frac{1+\a}{p}}(\T^d;\R^d)}\,\lesssim_{\lambda,\mu,T}\, \|a\|_{\Xap}\,\lesssim_{\mu}\, 1
	\end{equation}
	since the $\vp_\iota$ are smooth. If $s-1\le 0$, we choose $\delta\in (0,1)$ such that
	\begin{equation}\label{Eq153}
		s\,+\,2\,-\,4\,\tfrac{1+\kappa}{p}\,>\, 1+\delta-s
	\end{equation}
	by \eqref{Eq103}. Since 
	$s-1-\delta<0$, we can use Proposition \ref{prop:para}~\eqref{item:prop_para_3} to estimate
	\begin{align}&
		\|a\nabla \vp_\iota \cdot \nabla \Delta u \|_{X_0}\,\lesssim\, 
		\|a\nabla \vp_\iota \cdot \nabla \Delta u \|_{H^{s-1-\delta,q}(\Tor^d)}\\&\quad \lesssim \,
		\|a\nabla \vp_\iota \|_{L^\infty(\Tor^d;\R^d)}\|\nabla \Delta u\|_{H^{s-1-\delta,q}(\Tor^d;\R^d)} \,+\, \| a\nabla \vp_\iota\|_{H^{\sigma,\zeta}(\Tor^d;\R^d)} \|\nabla \Delta u\|_{H^{s-1-\delta,q}(\Tor^d;\R^d)}
	\end{align}
	as soon as $\sigma>\max \{ \frac{d}{\zeta}, 1+\delta-s\}$ and $\zeta\in [q',\infty)$. Arguing as in Lemma \ref{Lemma_aj}, we see that we can choose $\sigma$ and $\zeta$ such that $\Xap\hookrightarrow H^{\sigma,\zeta}(\Tor^d)$ due to \eqref{Eq104} and \eqref{Eq153}. Because also $\Xap \hookrightarrow L^\infty(\Tor^d)$ by \eqref{Eq104}, we obtain 
	\begin{align}
		\|a\nabla \vp_\iota \cdot \nabla \Delta u \|_{X_0}\, \lesssim_{\lambda,\mu,T}\, \|u\|_{H^{s+2-\delta,q}(\Tor^d)}
	\end{align}
	by \eqref{Eq156}.
	If instead $s-1>0$, we choose $\delta\in (0,\min\{1,s-1\})$. Then an application of Proposition \ref{prop:para}~\eqref{item:prop_para_1} as in the proof of Lemma \ref{Lemma_aj} yields 
	\begin{align}&
		\|a\nabla \vp_\iota \cdot \nabla \Delta u \|_{X_0}\,\le\, \|a\nabla \vp_\iota \cdot \nabla \Delta u \|_{H^{s-1-\delta,q}(\Tor^d)}
		\\&\quad \lesssim\, \|a\nabla \vp_\iota\|_{L^\infty(\Tor^d;\R^d)}\|\nabla \Delta u\|_{H^{s-1-\delta,q}(\Tor^d;\R^d)}\,+\, 
		\|a\nabla \vp_\iota\|_{H^{s-1-\delta,l}(\Tor^d;\R^d)}\|\nabla \Delta u\|_{L^r(\Tor^d;\R^d)}
	\end{align}
	for $l\in (1,\infty)$, $r\in (1,\infty]$
	whenever \eqref{Eq109} holds. If we can show that $l,r$ can be chosen in a way such that $\Xap\hookrightarrow H^{s-1-\delta,l}(\Tor^d)$ and $H^{s-1-\delta,q}(\Tor^d)\hookrightarrow L^r(\Tor^d)$, we can conclude as in Lemma \ref{Lemma_aj} that
	\begin{align}
		\|a\nabla \vp_\iota \cdot \nabla \Delta u \|_{X_0}\, \lesssim_{\lambda,\mu,T}\, \|u\|_{H^{s+2-\delta,q}(\Tor^d)},
	\end{align}
	using additionally $\Xap\hookrightarrow L^\infty(\T^d)$ and \eqref{Eq156}.  If $s-1-\delta-\frac{d}{q}>0$, the choice $l=q$ and $r=\infty$ suffices. If $s-1-\delta -\frac{d}{q}\le 0$, we take $r$ in accordance with
	\begin{equation}
		s-1-\delta-\tfrac{d}{q}\,>\,
		\tfrac{-d}{r}\,>\, s-1-\delta-\tfrac{d}{q}-\eta
	\end{equation}
	for $\eta>0$. Using \eqref{Eq109} we obtain that 
	\begin{equation}
		\tfrac{d}{l}\,>\, s-1-\delta-\eta\quad \iff\quad s-1-\delta-\tfrac{d}{l}\,<\,\eta.
	\end{equation}
	Choosing $\eta>0$ smaller than the left-hand side of \eqref{Eq104} ensures that the desired embeddings hold. We conclude that \eqref{Eq155} is valid in any case.
	
	With this at hand, we can proceed as in the proof of \cite[Lemma 5.4]{AV21_max_reg_torus}, namely, we use that
	\begin{align}&
			\| 	[\mathcal{A}, \vp_\iota ] u\|_{L^p((t,\tau)\times\O, w_{\hat{\a}}^t;X_0)}\,\lesssim_{\lambda,\mu,T}\, \|u\|_{L^p((t,\tau)\times\O, w_{\hat{\a}}^t;X_{1-{\delta}/{4}})}\,\le \, 
		\|u\|_{L^p((t,\tau)\times\O, w_{\hat{\a}}^t;X_0)}^{{\delta}/{4}} \|u\|_{L^p((t,\tau)\times\O, w_{\hat{\a}}^t;X_1)}^{1-{\delta}/{4}}
		\\&\quad  \le \,C_{\delta,\eta} \|u\|_{L^p((t,\tau)\times\O, w_{\hat{\a}}^t;X_0)} \,+\, \eta \|u\|_{L^p((t,\tau)\times\O, w_{\hat{\a}}^t;X_1)}
	\end{align}
	for any $\eta>0$
	by the interpolation inequality for Bochner spaces \cite[Theorem 2.2.6]{Analysis1} and   Young's inequality. Choosing $\eta$ sufficiently small and inserting this in \eqref{Eq154} we obtain that
	\begin{align}
		&
		\|u\|_{L^p((t,\tau)\times\O, w_{\hat{\a}}^t;X_1)}\,\lesssim_{\lambda,\mu,T}	\,
		\| u \|_{L^p((t,\tau)\times\O, w_{\hat{\a}}^t;X_0)}\,+\,\|f\|_{L^p((t,\tau)\times\Omega,w_{{\hat{\a}}}^t;X_0)}\,+\,\|g\|_{L^p((t,\tau)\times\Omega,w_{{\hat{\a}}}^t;\g(H,X_{1/2}))}.
	\end{align}
	If $\tau \le t+T^*$ for sufficiently small $T^*$ depending on $\lambda,\mu,T$, the term containing $u$ on the right-hand side can once more be absorbed by the first statement of \cite[Lemma 3.13]{AV19_QSEE_1} together with \eqref{Eq25} leading to the desired a-priori estimate \eqref{eq:lemaprioriTorus}.
\end{proof}
Having established the necessary a-priori estimate, Theorem \ref{Theorem_SMR} now follows by the method of continuity together with a partition of the time interval $[t_0,T]$  into pieces $t_0<t_1<\dots<t_N =T$ of the length $T^*$ from Lemma \ref{l:estimates_small_interval}.

\begin{proof}[Proof of Theorem \ref{Theorem_SMR}]
We firstly remind ourselves of the uniform estimate  \eqref{Eq25} on the operator norm of $\mathcal{A}$. By the comments below \eqref{eq:bilaplacian_H_infty_calculus}, by \cite[Proposition 3.8]{AV19_QSEE_1}, it suffices to show that
	$\mathcal{A}\in \MRtaszh$ with 
	\begin{equation}
		\max\bigl\{C^{\det,0,p,\hat{\a}}_{\A}({t_0},T), C^{\stoc,0,p,\hat{\a}}_{\A}({t_0},T)\}\,\lesssim_{\lambda,\mu,T}\,1.
	\end{equation}
	This again can be deduced by \cite[Proposition 3.1]{AV21_max_reg_torus}, if we can provide a partition ${t_0}<t_1<\dots<t_N =T$ such that $\mathcal{A}\in \mathcal{SMR}_{p,\hat{\a}}(t_0,t_{1}) $ and $\mathcal{A}\in \mathcal{SMR}_{p,0}(t_i,t_{i+1}) $ for $i\ge 1$ with 
	\begin{equation}
		\max\bigl\{C^{\det,0,p,\a_i}_{\A}(t_i,t_{i+1}), C^{\stoc,0,p,\a_i}_{\A}(t_i,t_{i+1})\}\,\lesssim_{\lambda,\mu,T}\,1, \qquad \kappa_i \,=\, \hat{\kappa} \delta_{i0}
	\end{equation}
	for $i\in \{0,\dots, N-1\}$. The required a-priori estimate follows from Lemma \ref{l:estimates_small_interval} if we set $t_i=(t_0+iT^*)\wedge T$ for $i\ge 1$. Moreover, the existence of strong solutions  on the respective intervals $[t_i,t_{i+1}]$ follows from the method of continuity \cite[Proposition 3.13]{AV19_QSEE_2}, this time applied to the family of operators
	\begin{equation}
		\tilde{\mathcal{A}}_r(u)\,=\, r \mathcal{A}(u)\,+\, (1-r)\lambda\Delta^2u\,=\, 
		\div((ra+(1-r)\lambda)\nabla \Delta u  ), \qquad r\in [0,1],
	\end{equation}
	because the coefficients $a_r=ra+(1-r)\lambda$ satisfy the assumptions of Lemma \ref{l:estimates_small_interval}  with the same constants $\lambda,\mu$ as the original coefficient $a$.
\end{proof}
\subsection{Local well-posedness and blow-up criteria I -- Regularized problem}\label{Sec_truncated_problem}
In this subsection, we use the preceding findings to apply the framework for quasilinear stochastic evolution equations from \cite{AV19_QSEE_1,AV19_QSEE_2} to the regularized equation \eqref{Eq102}. To this end, we fix again compatible parameters $(p,\a,s,q)$ and a sequence $(\psi_k)_{k\in \N}$ subject to Assumptions \ref{Assumptions_coefficients} and \ref{Assumptions_noise_local} with $(s_\psi,q_\psi)=(s,q)$, and start with the definition of maximal local solutions to \eqref{Eq102}, similar to the one of \eqref{Eq101}. We remark that this definition deviates from the corresponding Definition \ref{Defi_Lpk} and Definition \ref{Defi_max} for  \eqref{Eq101} in which the solution is required to be positive.
\begin{definition}\label{Defi_Lpk_cutoff}
Let $(p,\a,s,q)$ be as in Assumption \ref{Assumptions_coefficients}. 
	Let $\sigma\colon \Omega \to [0,\infty]$ be a stopping time and $u\colon \llbracket0,\sigma\rrparenthesis \to X_1$ progressively measurable. Then the tuple $(u,\sigma)$ is a \emph{local $(p,\a,s,q)$-solution} to \eqref{Eq102}, if there exists a localizing sequence $0\le \sigma_l\nearrow \sigma$ of stopping times such that for all $l\in\N$, we have a.s.\
	\begin{align}
		&u\,\in \, L^p(0,\sigma_l, w_\kappa; X_1) \,\cap\, C([0,\sigma_l]; \Xap),\\
		&F^{(j)}(u)\in L^p(0,\sigma_l, w_\kappa; X_1),\quad
		G^{(j)}[u]\in L^p(0,\sigma_l, w_\kappa; \gamma( \ell^2(\N),X_{1/2})),
	\end{align}
and a.s.\ for all $t\in [0,\sigma_l]$
	\begin{align}
	\label{eq:stoch_integrated_eq_cutoff}
		u(t)\,-\, u(0)\,+\, \int_0^t A^{(j)}[u(r)](u(r))\, \dd r\,=\, \int_0^t F^{(j)}(u(r))\, \dd r
		\,+\, \int_0^t 
		G^{(j)}[u(r)]
		\, \dd W_r\,.
	\end{align}
\end{definition}

As for Definition \ref{Defi_Lpk}, all the integrals in  \eqref{eq:stoch_integrated_eq_cutoff} are well-defined due to the required integrability conditions.

\begin{definition}\label{Defi_max_cutoff}We call a local $(p,\a,s,q)$-solution $(u,\sigma)$  to \eqref{Eq102} \emph{maximal unique $(p,\a,s,q)$-solution}, if for every local $(p,\a,s,q)$-solution $(v,\tau)$ to \eqref{Eq102}, one has $\tau\leq \sigma$ a.s.\ and $u=v$ a.e.\ on $\llbracket 0,\sigma\wedge \tau\rro$.
\end{definition} 

The existence of a maximal unique $(p,\a,s,q)$-solution to \eqref{Eq102} follows from our preceding findings.

\begin{proposition}\label{Wellposedness_truncated_problem}
Let the Assumptions \ref{Assumptions_coefficients} and \ref{Assumptions_noise_local} with $(s_\psi,q_\psi)=(s,q)$ be satisfied.
	Let  $u_0\in L_{\mathscr{F}_0}^{0}(\Omega;\Xap)$. Then there exists a maximal unique $(p,\a,s,q)$-solution $(u,\sigma)$ to \eqref{Eq102} in the sense of Definition \ref{Defi_max_cutoff} such that a.s.\ $\sigma>0$ and
	\begin{equation}\label{Eq144}
	u\in H^{\theta,p}_{\loc}([0,\sigma), w_\kappa, X_{1-\theta}) \cap C((0,\sigma); \Xp)
	\end{equation}
	for all $\theta\in [0,\frac{1}{2})$, if  $p>2$. Additionally, for all $T<\infty$, the following blow-up criterion holds
	\begin{equation}\label{Eq30}
		\P \Bigl(
		\sigma<T\,,\, \sup_{t\in [0,\sigma)} \| u(t)\|_{ \Xap}<\infty	\Bigr)\,=\, 0.	
		\end{equation}
\end{proposition}

Note that the blow-up criterion in \eqref{Eq30} is stronger than those in \cite{AV19_QSEE_2}, as it only involves a time supremum norm, whereas the ones in \cite[Theorem 4.9]{AV19_QSEE_2} require the existence of a limit as $t \nearrow \sigma$.
This improvement is enabled by Theorem \ref{Theorem_SMR}, which imposes no assumptions on the time regularity of the coefficients. Our arguments are independent of the specific structure of the thin-film operators, allowing the blow-up criterion \eqref{Eq30} to extend, with minimal changes, to a broader class of quasilinear SPDEs.

To keep this paper self-contained, we derive \eqref{Eq30} directly from the blow-up criterion \cite[Lemma 5.4]{AV19_QSEE_2}, restated below in \eqref{Eq30_stronger_proof}, which is the most elementary among those presented in \cite{AV19_QSEE_2} as it follows directly from the local well-posedness of \eqref{Eq119} initialized at an arbitrary stopping time.

\begin{proof}[Proof of Proposition \ref{Wellposedness_truncated_problem}]
Let $T<\infty$ be arbitrary. 
	We check the assumptions of \cite[Theorem 4.7]{AV19_QSEE_1} and define the localization 
	\[u_{0,n} =\mathbf{1}_{\{\|u_0\|_{\Xap} \le n\}}u_0,\qquad n\in \N\] 
	of the initial value. Since $m_j\colon \R\to \R$ is smooth, we deduce that $m_j(u_{0,n})$ lies in $ L_{\mathscr{F}_0}^\infty (\Omega;\Xap)$ by Proposition \ref{prop:comp}~\eqref{item:prop_comp_B} and moreover uniformly bounded away from $0$. Thus, Theorem \ref{Theorem_SMR} applies to the operator $A^{(j)}[u_{0,n}](f)= \div(m_j(u_{0,n} ) \nabla \Delta f )$ and yields that
	\begin{equation}
		A^{(j)}[u_{0,n}]\,\in \, \MRtzeroa
	\end{equation}
	for each $n\in \N$. Therefore, the assumption regarding stochastic maximal regularity holds, and we convinced ourselves already in Remark \ref{Rem1} that \cite[Hypothesis (H')]{AV19_QSEE_1} is satisfied. 
Furthermore, due to Theorem \ref{Theorem_SMR} and \cite[Remark 5.6]{AV19_QSEE_2}, our definition of maximal unique $(p,\a,s,q)$-solution to \eqref{Eq102} is equivalent to \cite[Definition 4.4]{AV19_QSEE_1} and therefore an application of \cite[Theorem 4.7]{AV19_QSEE_1} yields the existence of a maximal unique $(p,\a,s,q)$-solution $(u,\sigma)$ to \eqref{Eq102} and the regularity assertions. 
Next, we prove the blow-up criterion \eqref{Eq30}. 
Our starting point is the blow-up criterion which follows from \cite[Lemma 5.4]{AV19_QSEE_2} applied to the regularized SPDE \eqref{Eq102} in the abstract form \eqref{Eq119}:
\begin{equation}
\label{Eq30_stronger_proof}
		\P \Bigl(
		\sigma<T\,,\, \lim_{t\nearrow \sigma}  u(t)\text{ exists in }\Xp,\,\|u\|_{L^p(0,\sigma,w_{\a};X_1)}+ \mathcal{N}^\a (u;\sigma)<\infty\Bigr)\,=\, 0,
\end{equation}
where 
$
\mathcal{N}^\a (u;\sigma):= \|F(u)\|_{L^p(0,\sigma,w_{\a};X_0)}+\|G(u)\|_{L^p(0,\sigma,w_{\a};\g(\ell^2(\N),X_{1/2}))}.
$
Before going further, let us comment on the technical requirements needed to apply \cite[Lemma 5.4]{AV19_QSEE_2}. The condition from \cite[Assumption 4.5]{AV19_QSEE_2} regarding stochastic maximal regularity follows as before from Theorem \ref{Theorem_SMR}. Following again Remark \ref{Rem1} we can choose $\beta_1=\vp_1$ and $\rho_2=1$ to estimate the nonlinearity $G^{(j)}$ and therefore \cite[Assumption 4.7]{AV19_QSEE_2} is satisfied by \cite[Remark 4.8]{AV19_QSEE_2}.
	
To conclude the proof, we show that \eqref{Eq30_stronger_proof}  implies the seemingly stronger statement of \eqref{Eq30}. 
To this end, we use the stochastic maximal regularity estimates of Theorem \ref{Theorem_SMR} that hold for coefficients with measurable dependence on time. 
To prove \eqref{Eq30}, it is enough to show that, for all $M<\infty$,
\begin{equation}\label{Eq30_Mclaim}
		\P (\O_M)\,=\, 0\quad \text{ where }\quad \O_M\,:=\,\Bigl\{
		\sigma<T\,,\, \sup_{t\in [0,\sigma)} \| u(t)\|_{ \Xap}\leq M \Bigr\}.
	\end{equation}
For the sake of clarity, we split the proof of \eqref{Eq30_Mclaim} into three steps. Moreover, we define $\tilde{u}_0=\mathbf{1}_{\{\|u_0\|_{\Xap} \le M\}}u_0$ which coincides with $u_0$ on $\Omega_M$ and satisfies $\tilde{u}_0\in L^p_{\F_0}(\O;\Xap)$.

\emph{$(\varepsilon$-linearity in $X_1$ for the nonlinearities$)$ For all $\varepsilon>0$ there exists $C_{\varepsilon}>0$ such that, for $v\in X_1$,
\begin{equation}
\label{eq:blow_up_varepsilon_proof}
\|F^{(j)}(v)\|_{X_0}\,+\,\|G^{(j)}[v]\|_{\g(\ell^2(\N),X_{1/2})}\,\leq \, \varepsilon\|v\|_{X_1} \,+\, 
C_{\varepsilon}(1+ \|v\|_{\Xap}^{\delta})
\end{equation}
where $\delta>0$ is independent of $\varepsilon>0$.}
The estimate \eqref{eq:blow_up_varepsilon_proof} is a straightforward consequence of the \emph{subcriticality} of the nonlinearity for the regularized thin-film equation \eqref{Eq102}, see Remark \ref{Rem1}. Indeed, by the reiteration theorem for real interpolation (see, e.g., \cite[Theorem 3.5.3]{BeLo}), for all $\beta\in (1-\frac{1+\a}{p},1)$ one has 
$$
\|v\|_{X_{\beta}}\,\lesssim \,\|v\|_{\Xap}^{1-\theta}\|v\|_{X_1}^{\theta} \ \ \text{ for } v\in X_1, \ \ \  \text{ with }\ \ \  \theta\,=\,\frac{\beta-1+{(1+\a)}/{p}}{{(1+\a)}/{p}}.
$$
One can readily check that \eqref{eq:blow_up_varepsilon_proof} follows from the above and the subcriticality condition \eqref{Eq69}.

\emph{For all $M<\infty$, it holds that
\begin{equation}\label{Eq70}
\E \|u\|_{L^p(0,\tau,w_{\a};X_1)}^p\,<\,\infty, \ \ \ \text{ where } \ \ \ \tau\,:=\,\inf\big\{t\in [0,\sigma)\,:\,   \|u\|_{\Xap}\geq M \big\}\wedge T
\end{equation}
with $\inf\emptyset:=\sigma$ and in particular a.s.\ $\mathcal{N}^\a (u;\tau)<\infty$.}

Note that the last assertion follows from the first one and \eqref{eq:blow_up_varepsilon_proof}, and therefore it suffices to show \eqref{Eq70}.
To this end, let $\sigma_n$ be the stopping time given by 
\begin{align}
\sigma_{n}&:= \inf\big\{t\in [0,\sigma)\,:\,  \|u\|_{L^p(0,t,w_{\a};X_1)}\geq n \big\}\wedge T,
\end{align} 
where, as above, $\inf\emptyset:=\sigma$. Now the idea is to apply stochastic maximal regularity estimates of Theorem \ref{Theorem_SMR} to the problem \eqref{Eq102}. To this end, note that 
$u\in L^p(\llbracket 0,\sigma_n\wedge \tau\rrbracket ,w_{\a};X_1)$ solves, on $\llbracket 0,\tau\rro$,
 \begin{equation}
	\label{eq:diffAB_s_1}
	\begin{cases}
		\dd u\, +\,\div(a\nabla \Delta u)\,\dd t\,=\,\one_{\llbracket 0,\tau\rro}F^{(j)}(u)\, \dd t+ \one_{\llbracket 0,\tau\rro}G^{(j)}[u]\, \dd W,\\
		u(0)\,=\,\tilde{u}_{0},
	\end{cases}
\end{equation}
 where $a:=\one_{\llbracket 0,\tau\rro }m_j(u)+ \one_{\llbracket \tau,T\rr}$. Recall that $\Xap\embed L^{\infty}(\T^d)$ by \eqref{Eq104}. Hence, by Proposition \ref{prop:comp}~\eqref{item:prop_comp_B},
$$
\essinf_{[0,T]\times \O}a\,\gtrsim_j \,1 \qquad \text{ and }\qquad 
\esssup_{[0,T]\times \O}\|a\|_{\Xap}\,<\,\infty.
$$
Combining \eqref{eq:diffAB_s_1},
Theorem \ref{Theorem_SMR} and \cite[Proposition 3.12(b)]{AV19_QSEE_1}, we obtain the existence of a constant $C^{(0)}_M>0$ such that, for all $n\geq 1$, 
\begin{align*}
\E \|u\|_{L^p(0,\tau\wedge \sigma_n,w_{\a};X_1)}^p 
\,&\leq\, C^{(0)}_M\E\|\tilde{u}_0\|_{\Xap}^p \,
+ \,C^{(0)}_M \E\int_0^{\tau\wedge \sigma_n} \big(\|F^{(j)}(u)\|_{X_0}^p+ \|G^{(j)}[u]\|^p_{\g(\ell^2(\N),X_{1/2})}\big) \,w_\a\dd t\\
&\leq\, C^{(0)}_M\E\|\tilde{u}_0\|_{\Xap}^p\, 
+ \,\tfrac{1}{2} 
\E \|u\|_{L^p(0,\tau\wedge \sigma_n,w_{\a};X_1)}^p
\,+\, C_{\delta,M}^{(1)}
\end{align*}
where  we applied \eqref{eq:blow_up_varepsilon_proof} with $\varepsilon=(2C^{(0)}_M)^{-1}$ and used that $\sup_{t\in [0,\tau)}\|u\|_{\Xap}\leq M$. The above implies
\begin{equation}
\label{eq:LpX_1_part_estimate_proof}
\E \|u\|_{L^p(0,\tau\wedge \sigma_n,w_{\a};X_1)}^p \,\leq\, C^{(0)}_M\E\|\tilde{u}_0\|^p_{\Xap} \,+\,C_{\delta,M}^{(1)}.
\end{equation}
The claimed estimate \eqref{Eq70} now follows by letting $n\to \infty$ and using Fatou's lemma.

\emph{Let $\tau$ be as in \eqref{Eq70}. Then, a.s.,}
\begin{equation}
\label{eq:sup_lim_Xp_proof_blow_up}
\lim_{\tau\nearrow \tau} u(t)\text{\emph{ exists in }}\Xp.
\end{equation}
Here, we need another localization argument. For all $k\geq 1$, let 
$$
\tau_k :=\inf \{t\in [0,\tau)\,:\, \mathcal{N}^\a(u;t)\geq k\} \ \ \  \text{ with } \ \ \ \inf\emptyset:=\tau, 
$$
and where $\tau$ is as in \eqref{Eq70}. From the comments below \eqref{Eq70}, it follows that $\lim_{k\to \infty}\P(\tau_k=\tau)=1$. In particular, to obtain \eqref{eq:sup_lim_Xp_proof_blow_up} it is enough to show that
\begin{equation}
\label{eq:sup_lim_Xp_proof_blow_up_new}
\lim_{\tau\nearrow \tau_k} u(t)\text{ exists in }\Xp \ \text{ for all } k\geq 1,
\end{equation}
for which we fix $k\geq 1$ in the following. Similar to \eqref{eq:diffAB_s_1}, let us consider 
\begin{equation}
	\label{eq:diffAB_s_1_k}
	\begin{cases}
		\dd v\, +\,\div(a\nabla \Delta v)\,\dd t\,=\,\one_{\llbracket 0,\tau_k\rro}F^{(j)}(u)\, \dd t+ \one_{\llbracket 0,\tau_k\rro}G^{(j)}[u]\, \dd W,\\
		u(0)\,=\,\tilde{u}_{0},
	\end{cases}
\end{equation}
 where, again, $a:=\one_{\llbracket 0,\tau\rro }m_j(u)+ \one_{\llbracket \tau,T\rr}$. 
Now, from Theorem \ref{Theorem_SMR} (see also Definition \ref{def:SMRgeneralized}\eqref{it:SMR_regularity_0}), there exists a unique solution $v\in L^p((0,T)\times \O,w_{\a};X_1)$ satisfying, a.s., 
\begin{equation}
\label{eq:regularity_path_v_proof}
\textstyle
v\in \bigcap_{\theta\in [0,1/2)} H^{\theta,p}(0,T,w_{\a};X_{1-\theta})\subseteq C((0,T];\Xp),
\end{equation}
where the embedding follows from the instantaneous regularization part of the trace embedding in Proposition \ref{prop:tracespace}\eqref{it:trace_without_weights_Xp}.
Now, $u|_{\llbracket 0,\tau_k\rro}\in L^p(\llbracket 0,\tau_k\rro,w_{\a};X_1)$ is also a solution to \eqref{eq:diffAB_s_1_k} on the stochastic interval $\llbracket 0,\tau_k\rro$. From Definition \ref{def:SMRgeneralized} and Theorem \ref{Theorem_SMR}, it follows that $u=v$ a.e.\ on $\llbracket 0,\tau_k\rro$, which together with \eqref{eq:regularity_path_v_proof} implies \eqref{eq:sup_lim_Xp_proof_blow_up_new} as desired.

\emph{Proof of \eqref{Eq30_Mclaim}}. Let $\tau$ be defined as in \eqref{Eq70}.  From  \eqref{eq:sup_lim_Xp_proof_blow_up} we know that $\lim_{t\nearrow \sigma} u(t)$ exists in $\Xap$ a.s.\ on $\{\tau=\sigma\}$. Moreover, from \eqref{Eq70} and the comments below it, we also have $\|u\|_{L^p(0,\sigma,w_{\a};X_0)}+ \mathcal{N}^\a (u;\sigma)<\infty$ a.s.\ on $\{\tau=\sigma\}$. 
Since $\{\tau=\sigma\}\supseteq \O_M$ by definition of $\tau$, we have 
\begin{align*}
\P(\O_M)
&=\,\P\Big(\O_M \cap \Big\{
		\sigma<T\,,\, \lim_{t\nearrow \sigma}  u(t)\text{ exists in }\Xp,\, \|u\|_{L^p(0,\sigma,w_{\a};X_1)}+ \mathcal{N}^\a (u;\sigma)<\infty\Big\}\Big)\\
&\leq\, \P \Bigl(
		\sigma<T\,,\, \lim_{t\nearrow \sigma}  u(t)\text{ exists in }\Xp,\, \|u\|_{L^p(0,\sigma,w_{\a};X_1)}+ \mathcal{N}^\a (u;\sigma)<\infty\Bigr)\,\stackrel{\eqref{Eq30_stronger_proof}}{=}\, 0.
\end{align*}
Thus \eqref{Eq30_Mclaim} is proved and the claimed blow-up criterion \eqref{Eq30} follows from the arbitrariness of $M$.
\end{proof}

\subsection{Instantaneous regularization and blow-up criteria II -- Regularized problem}\label{Sec_reg}
This subsection is dedicated to understanding how the regularity of the noise affects the regularity of solutions to the regularized problem \eqref{Eq102} and its consequences in terms of blow-up criteria.

The following result is the key ingredient in the proof of Proposition \ref{Prop_regularization}.

\begin{proposition}[Instantaneous regularization -- Regularized problem]
\label{Prop_regularization_trunc}
Let $(p,\a,s,q)$ be as in Assumption \ref{Assumptions_coefficients}. 
	Let $u_0\in L_{\mathscr{F}_0}^{0}(\Omega;\Xap)$ and $(u,\sigma)$ the maximal unique $(p,\a,s,q)$-solution to \eqref{Eq102} provided by Proposition \ref{Wellposedness_truncated_problem}. 
	Suppose that Assumption \ref{Assumptions_noise_local} holds for some $s_{\psi}\geq s$ and all $q_\psi\in [2,\infty)$.
	Then, a.s.,
	\begin{equation}
	\label{eq:improved_regularity_regularized_problem}
		u\,\in\, H_{\loc}^{\theta,r}(0,\sigma; H^{2+s_{\psi}-4\theta,\zeta}(\Tor^d)) \ \text{
	for all $\theta\in [0,\tfrac{1}{2})$ and $r, \zeta\in [2,\infty)$.}
	\end{equation}
		In particular $u\,\in\,C_{\loc}^{\theta_1,\theta_2}((0,\sigma)\times \T^d)$ a.s.\ for all $\theta_1\in [0,\frac{1}{2})$ and $\theta_2\in (0,2+s_{\psi})$.
\end{proposition}

\begin{proof}
The last assertion follows from \eqref{eq:improved_regularity_regularized_problem} and Sobolev embeddings. Thus, below we only prove \eqref{eq:improved_regularity_regularized_problem}.
The proof of \eqref{eq:improved_regularity_regularized_problem} can be obtained via the results in \cite[Section 6]{AV19_QSEE_2}. For the reader's convenience, however, we provide a self-contained proof based on the stochastic maximal $L^p$-regularity result of Theorem \ref{Theorem_SMR} and the instantaneous regularization of weighted anisotropic spaces, see Proposition \ref{prop:tracespace}\eqref{it:trace_without_weights_Xp}, which we 
structure as follows:
\begin{itemize}
\item Reduction to the case of a positive weight $\a>0$. 
\item Bootstrap time integrability.
\item Bootstrap spatial integrability.
\item Bootstrap spatial smoothness.
\end{itemize}	
The reduction in the first step is convenient since for positive $\kappa$ there is an immediate gain of regularity between time $0$ and $\epsilon$, cf. Proposition \ref{prop:tracespace}\eqref{it:trace_with_weights_Xap} and  \eqref{it:trace_without_weights_Xp}, respectively. 
	
	\textit{We can assume that $p>2$ and  $\kappa>0$.} 
	In this step, we assume that \eqref{eq:improved_regularity_regularized_problem} is valid in the case $p>2$ and $\a>0$ and show that it carries over to the general case $p\ge 2$, $\a\ge 0$. To this end, it suffices to show that if either $p=2$ or $\a=0$, there exists a new set of parameters $(\hat{p}, \hat{\kappa},\hat{s}, q)$ with $\hat{p}>2$ and $\hat{\a}>0$, such that $(u,\sigma)$ coincides with the maximal unique $( \hat{p},\hat{\kappa}, \hat{s}, q)$-solution to \eqref{Eq102} in the sense of Definitions \ref{Defi_Lpk_cutoff} and \ref{Defi_max_cutoff}. 
Let us start with the case $p=2$. By \eqref{Eq1} in Assumption \ref{Assumptions_coefficients}, this forces $\kappa=0$ and $q=2$.
By standard interpolation inequalities, a.s.,
\begin{align}
\label{eq:hat_mixed_derivatives_proof_regularity}
u
&\, \in\, L^2_{\loc}([0,\sigma);H^{s+2}(\T^d))\cap C([0,\sigma);H^{s}(\T^d))\\
\nonumber
&\embed 
L^{2/\theta}_{\loc}([0,\sigma);H^{s+2\theta}(\T^d))
\,{\embed}\,L^{2/\theta}_{\loc}([0,\sigma),w_{\hat{\a}};H^{s+2\theta}(\T^d)),
\end{align}
where $\theta\in (0,1)$ and in the last embedding we used that the weight $w_\a$ is bounded. 
Now, by continuity, there exist $\hat{\theta}\in (0,1)$ and $\hat{\a}>0$ for which Assumption \ref{Assumptions_coefficients} with $(p,\a,s)$ replaced by $(\hat{p},\hat{\a},\hat{s})$ holds, where $\hat{p}\,:=\,2/\hat{\theta}$, and that 
$
\hat{s}+2-4/\hat{p}<s.
$
The previous condition ensures that $H^{s}(\T^d)\embed B^{s+2-4\frac{1+\hat{\a}}{\hat{p}}}_{2,\widehat{p}}(\T^d)$. Thus, by  \eqref{eq:hat_mixed_derivatives_proof_regularity}, a.s., 
$$
u\, \in\, L^{\hat{p}}_{\loc}([0,\sigma);H^{\hat{s}+2}(\T^d))\cap C([0,\sigma);B^{s+2-4\frac{1+\hat{\a}}{\hat{p}}}_{2,\widehat{p}}(\T^d)).
$$
In particular $(u,\sigma)$ is a local $(\hat{p},\hat{\a},\hat{s},2)$-solution to \eqref{Eq102}. Let $(\hat{u},\hat{\sigma})$ be the maximal unique $(\hat{p},\hat{\a},\hat{s},2)$-solution to \eqref{Eq102} provided by Proposition \ref{Wellposedness_truncated_problem}. By maximality (cf., Definition \ref{Defi_max_cutoff}), we obtain 
\begin{equation}
\sigma\,\leq\, \hat{\sigma}\ \text{ a.s.\ }\qquad \text{ and }\qquad u \,=\,\hat{u} \ \text{ a.e.\ on }\llbracket 0,\sigma\rro.
\end{equation}
Hence, the claim of this step follows in this case, since the regularity of the $(\hat{p},\hat{\a},\hat{s},2)$-solution $(\hat{u},\hat{\sigma})$ transfers to $(u,\sigma)$. 
For completeness, let us show that such improved regularity also yields $\sigma=\widehat{\sigma}$ a.s. Indeed, if \eqref{eq:improved_regularity_regularized_problem} holds with $(u,\sigma)$ replaced by $(\hat{u},\hat{\sigma})$, it follows that 
$$
u\,=\,\hat{u}\,\in\, C((0,\sigma];H^s(\T^d)) \ \text{ a.s.\ on }\{\sigma<\widehat{\sigma}\}.
$$ 
Thus, for all $T<\infty$, 
$$
\P(\sigma<\hat{\sigma}\,,\, \sigma<T)\,\leq\,  \P\Big(\sigma<T\,,\, \sup_{t\in [0,\sigma)} \|u(t)\|_{H^s(\T^d)}<\infty\Big)\,=\,0,
$$
where the last equality follows from \eqref{Eq30} with $p=2$ and $\a=0$. 
The arbitrariness of $T<\infty$ yields $\sigma=\hat{\sigma}$ a.s.\ on $\{\sigma<\infty\}$. Since $\sigma\leq \hat{\sigma}$ a.s., it follows that $\sigma=\hat{\sigma}$ a.s., as desired.

For the case that $p>2$ and $\a=0$, the same procedure works by passing to the new parameters $(p,\hat{\a},s,q)$ with $\hat{\a}$ slightly larger than $0$.

	\emph{Temporal integrability.} We continue with the initial set of parameters $(p, \kappa,s,q)$ and can assume by the previous step that $p>2$ and $\kappa>0$. Here, we show that
	\begin{equation}
	\label{eq:temporal_regularity}
	\textstyle
	u\in \bigcap_{\theta\in [0,1/2)} H^{\theta,r}_{\loc}(0,\sigma;X_{1-\theta}) \text{ a.s.\ for all } r\in (2,\infty),
	\end{equation}
	for which it suffices to prove the existence of $\delta>0$ such that, for all $\zeta\in [p,\infty)$ and $\varepsilon>0$, 
	\begin{align}
	\label{eq:claim_step1_regularization}
	\textstyle
	u\in \bigcap_{\theta\in [0,1/2)} H^{\theta,\zeta} ([\varepsilon,\sigma);X_{1-\theta}) \ \ \text{a.s.\ on }\Dom_\varepsilon\qquad \Longrightarrow& \\
	\textstyle 
	u\in \bigcap_{\theta\in [0,1/2)} H^{\theta,\zeta+\delta} ([\varepsilon,\sigma),w_\a^\varepsilon;X_{1-\theta}) \ \ \text{a.s.\ on }\Dom_\varepsilon&,
	\end{align}
	where $\Dom_\varepsilon=\{\sigma>\varepsilon\}$ and $w_\a^{\varepsilon}(t)=|t-\varepsilon|^\a$ is the power weight centered in $\varepsilon$, see Subsection \ref{ss:notation}. Once \eqref{eq:claim_step1_regularization} is proven, the claim \eqref{eq:temporal_regularity} follows from a standard iteration argument, the arbitrariness of $\varepsilon$ and the fact that $w_{\a}^\varepsilon$ acts only near $t=\varepsilon$ (see \cite[Proposition 2.3]{AV19_QSEE_1}).
 	
	To prove \eqref{eq:claim_step1_regularization}, we exploit the subcriticality of our problem (see \cite[Corollary 6.5]{AV19_QSEE_2} for an alternative approach). Below, we assume that $u$ satisfies the left-hand side of \eqref{eq:claim_step1_regularization}. By Proposition \ref{prop:tracespace}, $u\in L^\zeta(\varepsilon,\sigma,w_\a;X_1)\cap C([\varepsilon,\sigma);X^{\mathrm{Tr}}_{\zeta})$ a.s.\ on $\Dom_\varepsilon$, where we recall that $X^{\mathrm{Tr}}_{\zeta}=X^{\mathrm{Tr}}_{0,\zeta}=(X_0,X_1)_{1-1/\zeta,\zeta}$, cf. \eqref{eq:interpolation_spaces_definition_notation_section}-\eqref{Eq72}.
	The interpolation argument in the first step in the proof of  Proposition \ref{Wellposedness_truncated_problem} ensures now the existence of $\iota\in (0,1)$ such that, for all $v\in X_1$ and $n\geq 1$ satisfying $\|v\|_{\Xap}\leq n$,
	$$
	\|F^{(j)}(v)\|_{X_0}+ \|G^{(j)}[v]\|_{\g(\ell^2(\N),X_{1/2})}\lesssim_n 1+\|v\|_{X_1}^\iota.
	$$
	Hence, from the assumed regularity of $u$, it follows that 
	\begin{align*}
	\|F^{(j)}(u)\|_{X_0}+ \|G^{(j)}[u]\|_{\g(\ell^2(\N),X_{1/2})}\in L^{\zeta/\iota}_{\loc}([\varepsilon,\sigma),w_\a^\varepsilon)\text{ a.s. }
	\end{align*}
	Next, note that $\one_{\Dom_\varepsilon}u(\varepsilon)\in X^{\mathrm{Tr}}_{\zeta}\subseteq X^{\mathrm{Tr}}_{\a,\zeta+\delta}$ a.s.\ provided $\zeta+\delta < \zeta(1+\kappa)$. As $\zeta \ge p>2$, the claim \eqref{eq:claim_step1_regularization} follows with the choice $\delta:= (\frac{1-\iota}{\iota})\wedge \kappa$ from Theorem \ref{Theorem_SMR} and a localization argument in the probability space to obtain $\omega$-integrability for the data, see, e.g., \cite[Proposition 3.11]{AVSurvey}. 
	
	\textit{Integrability in space.} 
	We again consider the maximal unique $(p,\kappa,s,q)$-solution $(u,\sigma)$ with $p>2$ and $\kappa>0$. In this step, we prove that 
	\begin{equation}
	\label{eq:temporal_regularity_spatial}
	\textstyle
	u\in \bigcap_{\theta\in [0,1/2)} H^{\theta,r}_{\loc}(0,\sigma;H^{s+2-4\theta,\zeta}(\T^d)) \text{ a.s.\ for all } r,\zeta\in (2,\infty).
	\end{equation}
	To establish this result, we take advantage of the time regularity just proven. Similar to the time regularization, it suffices to prove the existence of $\delta>0$ such that, for each $\zeta\in [q,\infty)$, $r\in [p,\infty)$ and $\varepsilon>0$, the following implication holds 
	\begin{align}
	\label{eq:claim_step1_regularization_space}
	\textstyle 
	u\in \bigcap_{\theta\in [0,1/2)} H^{\theta,r}_{\loc}( [\varepsilon,\sigma);H^{s+2-4\theta,\zeta}(\T^d)) \ \ \text{a.s.\ on $\Dom_\varepsilon$}\ \ \  \Longrightarrow& \\
	\nonumber
	\textstyle u\in \bigcap_{\theta\in [0,1/2)} H^{\theta,{r}}_{\loc}( [\varepsilon,\sigma),w_{{\alpha}}^\varepsilon;H^{s+2-4\theta,\zeta+\delta}(\T^d)) \ \ \text{a.s.\ on $\Dom_\varepsilon$}&,
	\end{align}
	where $\alpha\in [0,\frac{r}{2}-1)$ depends only on $p,\a$ and, as above, $\Dom_\varepsilon=\{\sigma>\varepsilon\}$ and $w_{{\alpha}}^\varepsilon(t)=|t-\varepsilon|^{\alpha}$. Indeed, if \eqref{eq:claim_step1_regularization_space}, then \eqref{eq:temporal_regularity_spatial} follows from an iteration argument and \cite[Proposition 2.3]{AV19_QSEE_1}.
	
	Now, fix $r\in [p,\infty)$ and $\varepsilon>0$, and let $\alpha\in [\a,\frac{{r}}{2}-1)$ be such that 
	\begin{equation}
		\label{eq:alpha_equal_hat_r}
	\tfrac{1+{\alpha}}{{r}}= \tfrac{1+\a}{p}.
	\end{equation}
	Note that the above choice is always possible as ${r}\geq p$ and $p>2$ (for the latter, see the first step of the current proof). From the admissibility of $(p,\a,s,q)$ and the above choice, it follows that $({r},{\alpha},s,\zeta)$ are admissible provided $\zeta\in [q,\infty)$ (see Assumption \ref{Assumptions_coefficients}). Next, for each $\zeta\in [q,\infty)$ and $\delta\geq 0$, it follows from Lemma \ref{Lemma_fj} and Lemma \ref{Lemma_gj} that, for all $n\geq 1$, and $v\in H^{s+2,\zeta+\delta}(\T^d)$ satisfying $\|v\|_{B^{s+2-4\frac{1+\alpha}{r}}_{\zeta+\delta,r}(\T^d)}\leq n$,
		\begin{align}
		\label{eq:estimate_improved_integrability_proof}
		\|F^{(j)}(v)\|_{H^{s-2,\zeta+\delta}(\T^d)}
		+
		\|G^{(j)}[v]\|_{\g(\ell^2(\N),H^{s,\zeta+\delta}(\T^d))}
		&\lesssim_n 1+\|v\|_{H^{s+2-\lambda,\zeta+\delta}(\T^d)},
	\end{align}	
	where $\lambda\in (0,4\frac{1+\a}{p}\wedge 1)$, and we used \eqref{eq:alpha_equal_hat_r}.
 	For convenience, below we let $\lambda=2\frac{1+\a}{p}\wedge \frac{1}{2}$. Sobolev embeddings ensure the existence of $\delta>0$ depending only on $p,\a$ and $d$ such that 
	\begin{equation}
	\label{eq:embeddings_gain_integrability}
	\begin{aligned}
	H^{s+2,\zeta}(\T^d)&\embed 
	H^{s+2-\lambda,\zeta+\delta}(\T^d),\\ 
	B^{s+2-\frac{4}{r}}_{\zeta,r}(\T^d)&
	\embed B^{s+2-\frac{4}{p}}_{\zeta,r}(\T^d)\embed B^{s+2-4\frac{1+\a}{p}}_{\zeta+\delta,r}(\T^d)
	=B^{s+2-4\frac{1+{\alpha}}{{r}}}_{\zeta+\delta,r}(\T^d).
	\end{aligned}
\end{equation}
	Assume now that the LHS\eqref{eq:claim_step1_regularization_space} holds. Proposition \ref{prop:tracespace} implies 
	$$u\in L^{r}_{\loc}([\varepsilon,\sigma);H^{s+2,\zeta}(\T^d))\cap C([\varepsilon,\sigma); B^{s+2-4/r}_{\zeta,r}(\T^d))$$ a.s.\ on $\Dom_\varepsilon$.  
	Thus, from \eqref{eq:estimate_improved_integrability_proof} and the embeddings in \eqref{eq:embeddings_gain_integrability}, 
	$$
	\|F^{(j)}(v)\|_{H^{s-2,\zeta+\delta}(\T^d)}+ \|G^{(j)}[v]\|_{\g(\ell^2(\N),H^{s,\zeta+\delta}(\T^d))}\in 
	L_{\loc}^{{r}}([\varepsilon,\sigma),w_{{\alpha}}^\varepsilon) \text{ a.s.\ on }\Dom_\varepsilon,
	$$
	and $\one_{\Dom_\varepsilon}u(\varepsilon)\in B^{s+2-4\frac{1+{\alpha}}{{r}}}_{\zeta+\delta,{r}}(\T^d)$ a.s.
	Thus, the RHS\eqref{eq:claim_step1_regularization_space} follows again from Theorem \ref{Theorem_SMR} and a localization argument in the probability space to obtain $\omega$-integrability for the data, see, e.g., \cite[Proposition 3.11]{AVSurvey}. 

	\textit{Smoothness in space.} To prove the improvement in the space regularity, we argue as in the previous step (an abstract version of the argument below can be found in \cite[Theorem 6.3]{AV19_QSEE_2}). 
	Here, we exploit the just-proven improvement in the time and space integrability of the maximal unique $(p,\a,s,q)$-solution $(u,\sigma)$ to \eqref{Eq102}. 
	To begin, note that for $r\in [p,\infty)$ and $\zeta\in [q,\infty)$  it follows that Assumption \ref{Assumptions_coefficients} holds with $(p,\a,s,q)$ replaced by $(r,0,s,\zeta)$. Moreover, by assuming $\zeta$ sufficiently large, we can assume $s+1>\frac{d}{\zeta}$, and therefore below we will be in the case $(i)$ of Lemma \ref{Lemma_gj}.
	From the previous steps, it suffices to show that for each $\tau\in [s,\infty)$ there are sufficiently large $r,\zeta$ such that, for all $\varepsilon>0$,  
	\begin{align}
	\label{eq:claim_step1_regularization_space_smoothness}
	\textstyle 
	u\in \bigcap_{\theta\in [0,1/2)} H^{\theta,r}_{\loc}( [\varepsilon,\sigma);H^{\tau+2-4\theta,\zeta}(\T^d)) \ \ \text{a.s.\ on $\Dom_\varepsilon$}\ \ \  \Longrightarrow& \\
	\nonumber
	\textstyle u\in \bigcap_{\theta\in [0,1/2)} H^{\theta,r}_{\loc}( [\varepsilon,\sigma),w_{\alpha}^\varepsilon;H^{\tau'+2-4\theta,\zeta}(\T^d)) \ \ \text{a.s.\ on $\Dom_\varepsilon$}&,
	\end{align}
	where $\tau'=(\tau+1)\wedge s_\psi$ and $\alpha=r\frac{(\tau'-\tau)}{4}-1$. Moreover, as above, $\Dom_\varepsilon=\{\sigma>\varepsilon\}$.  
	
	In the following, we assume that $\tau'>\tau$; otherwise, there is nothing to prove. It is clear that $\alpha<\frac{r}{2}-1$. By choosing $r$ sufficiently large, we can also ensure $\alpha\geq 0$. With the above choice of the weight $\alpha$, it holds that   
	$
	\tau'+2-4\frac{1+\alpha}{r}= \tau+2-\frac{4}{r}.
	$  
	Hence, the admissibility of $(r,\alpha,\tau',\zeta)$ follows from the one of $(r,0,\tau,\zeta)$. 
	Therefore, Lemma \ref{Lemma_fj} and Lemma \ref{Lemma_gj}$(i)$ imply that, for all $n\geq 1$ and $v\in H^{\tau'+2,\zeta}(\T^d)$ satisfying $\|v\|_{B^{\tau'+2-4\frac{1+\alpha}{r}}_{\zeta,r}(\T^d)}\leq n$, 
	\begin{align}
	\label{eq:estimate_improved_smoothness_proof_sobolev_smoothness}
	\|F^{(j)}(v)\|_{H^{\tau'-2,\zeta}(\T^d)}
	+
	\|G^{(j)}[v]\|_{\g(\ell^2(\N),H^{\tau',\zeta}(\T^d))}
	&\lesssim_n  1+\|v\|_{H^{\tau'+1,\zeta}(\T^d)}.
	\end{align}
	Now, assume that LHS\eqref{eq:claim_step1_regularization_space_smoothness} holds. As above, from Proposition \ref{prop:tracespace} and the choice of $\alpha$ and $\tau'$,
	\begin{align*}
	u
&\in L^{r}_{\loc}([\varepsilon,\sigma);H^{\tau+2,\zeta}(\T^d))\cap C([\varepsilon,\sigma); B^{\tau+2-4/r}_{\zeta,r}(\T^d))\\
&\subseteq L^{r}_{\loc}([\varepsilon,\sigma);H^{\tau'+1,\zeta}(\T^d))\cap C([\varepsilon,\sigma); B^{\tau'+2-4(1+\alpha)/r}_{\zeta,r}(\T^d))
	\end{align*}
	a.s.\ on $\Dom_\varepsilon$. 
	Therefore $\one_{\Dom_\varepsilon}u(\varepsilon)\in B^{\tau'+2-4\frac{1+\alpha}{r}}_{\zeta,r}(\T^d)$ a.s., and from \eqref{eq:estimate_improved_smoothness_proof_sobolev_smoothness}, 
	$$
	\|F^{(j)}(v)\|_{H^{\tau'-2,\zeta}(\T^d)}+ \|G^{(j)}[v]\|_{\g(\ell^2(\N),H^{\tau',\zeta}(\T^d))}
	\in L_{\loc}^{{r}}([\varepsilon,\sigma),w_{{\alpha}}^\varepsilon)\text{ a.s.\ on }\Dom_\varepsilon.
	$$ 
	Thus, the RHS\eqref{eq:claim_step1_regularization_space_smoothness} again follows from Theorem \ref{Theorem_SMR} and a localization argument in the probability space to obtain $\omega$-integrability for the data, see, e.g., \cite[Proposition 3.11]{AVSurvey}. 
	\end{proof}
	
As commented below the statement of Proposition \ref{prop:blow_up_criteria}, the independence of the regularity of solutions at positive times allows us to deduce the independence of the blow-up criteria from the original choice of the admissible parameters. This observation and the blow-up criterion in Proposition \ref{Wellposedness_truncated_problem} readily imply the following result, which is the key ingredient in the proof of Proposition \ref{prop:blow_up_criteria}.

	\begin{corollary}[Blow-up criteria -- Regularized problem]
\label{cor_blow_up_regularized}
Let $(p,\a,s,q)$ be as in Assumption \ref{Assumptions_coefficients}. Moreover, suppose that Assumption \ref{Assumptions_noise_local} holds for some $s_{\psi}\geq s$ and all $q_{\psi}\in [2,\infty)$.
	Let $u_0\in L_{\mathscr{F}_0}^{0}(\Omega;\Xap)$ and $(u,\sigma)$ the maximal unique $(p,\a,s,q)$-solution to \eqref{Eq102} provided by Proposition \ref{Wellposedness_truncated_problem}. Then, for any quadruple of admissible parameters $(p_0,\a_0,s_0,q_0)$ with $s_{0}\leq s_{\psi}$ and $0<\varepsilon<T<\infty$,
	\begin{equation}
	\label{eq:blow_up_extrapolated_regularized}
		\P \Bigl(
		\varepsilon<\sigma<T\,,\, \sup_{t\in [\varepsilon,\sigma)} \| u(t)\|_{ B^{\g_0}_{q_0,p_0}(\T^d)}<\infty	\Bigr)\,=\, 0
	\end{equation}
	where $\g_0 := s_0+2-4\frac{1+\kappa_0}{p_0}$.
\end{corollary}

Note that the norm in \eqref{eq:blow_up_extrapolated_regularized} can be evaluated even if $(p_0,\a_0,s_0,q_0)\neq (p,\a,s,q)$ due to Proposition \ref{Prop_regularization_trunc}.

\begin{proof}
By Proposition \ref{Prop_regularization_trunc} and the blow-up criterion in Proposition \ref{Wellposedness_truncated_problem}, \eqref{eq:blow_up_extrapolated_regularized} follows as in the proof of \cite[Theorem 2.10]{AV22_localRD}. We include some details for the reader's convenience.
We begin by collecting some useful facts. 
Let $(u,\sigma)$ be the maximal unique $(p,\a,s,q)$-solution to \eqref{Eq102} provided by Proposition \ref{Wellposedness_truncated_problem}. 
Proposition \ref{Prop_regularization_trunc} ensures that, for all $\varepsilon>0$,
$$
\one_{\{\sigma>\varepsilon\}}u(\varepsilon)\, \in\, L^0_{\F_{\varepsilon}}(\O;B^{\g_0}_{q_0,p_0}(\T^d)).
$$
Since Assumption \ref{Assumptions_noise_local} holds for some $s_{\psi}\geq s_0\vee s$ and all $q_{\psi}\in [2,\infty)$,  Proposition \ref{Wellposedness_truncated_problem} yields the existence of maximal unique $(p,\a,s,q)$-solution $(v,\tau)$ to 
\begin{equation}\label{Eq102_v_version}
\begin{cases}
	\dd v\,=\, -\div( m_j(v) \nabla \Delta v)\, \dd t\,+\, \div(\Phi_j(v)\nabla v )\, \dd t \,+\, \sum_{k\in \N}\div( g_j(v) \psi_k )\, \dd \beta^{(k)},\\
	v(0)\,=\,\one_{\{\sigma>\varepsilon\}}u(\varepsilon),
\end{cases}
\end{equation}
with 
\begin{equation}
\label{eq:blow_v_transference}
\P\Big(\varepsilon <\tau <T\,,\, \sup_{t\in [\varepsilon ,\tau)} \| v(t)\|_{ B^{\g_0}_{q_0,p_0}(\T^d)}<\infty	\Bigr)\,=\, 0.
\end{equation}
Finally, Proposition \ref{Prop_regularization_trunc} and a translation argument ensure that 
\begin{align}
\label{eq:v_improved_regularity_proof_blow_up}
v\in C([\varepsilon,\tau);C^{2+s_{\psi}-}(\T^d))\ \text{ a.s.\  }
\end{align}

We now turn to the proof of Corollary \ref{cor_blow_up_regularized}. By \eqref{eq:blow_v_transference}, it is enough to prove that 
\begin{equation}
\label{eq:sigma_equal_tau_v_u_proof_blow_up}
\sigma \, =\, \tau  \text{ a.s.\ on } \{\varepsilon<\sigma\leq T\} \quad \text{ and }\quad u\,=\,v  \text{ a.e.\ on } [\varepsilon,\sigma)\times \{\varepsilon<\sigma\leq T\}.
\end{equation}
To this end, we note that, due to Proposition \ref{Prop_regularization_trunc}, $(u|_{\llbracket \varepsilon,\sigma\rro},\sigma \one_{\{\sigma>\varepsilon\}}+\varepsilon\one_{\{\sigma\leq \varepsilon\}})$ is a local $(p_0,\a_0,s_0,q_0)$-solution to \eqref{Eq102_v_version}. The maximality of $(v,\tau)$ yields 
\begin{equation}
\label{eq:sigma_equal_tau_v_u_proof_blow_up_1}
\sigma \,\leq\, \tau  \text{ a.s.\ on } \{\varepsilon<\sigma\leq T\} 
\quad \text{ and }\quad u\,=\,v  \text{ a.e.\ on } [\varepsilon,\sigma)\times \{\varepsilon<\sigma\leq T\}.
\end{equation}
Therefore, since $C^{s_{\psi}-}(\T^d)\embed B^{2+s-4\frac{1+\a}{p}}_{q,p}(\T^d)= \Xap$ as $s\leq s_{\psi}$,  \eqref{eq:v_improved_regularity_proof_blow_up} and the previous yield
$$
\P(\varepsilon<\sigma<\tau\leq T)
\, \leq\,  \P\Big(\{\varepsilon<\sigma<\tau\}\cap \Big\{\sigma<T\,,\, \sup_{t\in [0,\sigma)} \| u(t)\|_{ \Xap}<\infty\Big\}\Big)\stackrel{\eqref{Eq30}}{=}0.
$$
Hence $\sigma = \tau$ a.s.\ on $\{\varepsilon<\sigma\leq T\}$, and therefore \eqref{eq:sigma_equal_tau_v_u_proof_blow_up} follows from \eqref{eq:sigma_equal_tau_v_u_proof_blow_up_1}.
\end{proof}

\subsection{Transference to the original equation}\label{Sec_orig_problem}
Finally, we are concerned with transferring the previously obtained results on \eqref{Eq102} to positive solutions to the original equation \eqref{Eq101}, which leads to the results stated in Subsection \ref{ss:local_all_dimensions_intro}. 
The main idea is that as long as a solution to \eqref{Eq102} remains above the threshold $2j^{-1}$, it is also a solution to \eqref{Eq119} and vice versa due to the choice of the regularization \eqref{eq:regularization_eta_j}. 

In the proof of Theorem \ref{Thm_local}, as an intermediate step to obtain maximal solutions, we also employ the notion of positive local \emph{unique} solutions to \eqref{Eq101}. Recall that positive local solutions of \eqref{Eq101} are defined in Definition \ref{Defi_Lpk}. 

\begin{definition}[Local unique positive solution for \eqref{Eq101}]\label{Defi_local_unique}
Let $(p,\a,s,q)$ be as in Assumption \ref{Assumptions_coefficients}. 
A positive local $(p,\a,s,q)$-solution $(u,\sigma)$ to \eqref{Eq101} is called \emph{positive local unique $(p,\a,s,q)$-solution}, if for every positive local $(p,\a,s,q)$-solution $(v,\tau)$ to \eqref{Eq101} one has $u=v$ a.s.\ on $[0,\sigma\wedge \tau)$. 
\end{definition}

In contrast to \emph{maximal solutions} as defined in Definition \ref{Defi_max}, the lifetime $\sigma$ of the unique solution $u$ does not necessarily extend that of the other solution $v$, i.e., $\tau$.

\begin{proof}[Proof of Theorem \ref{Thm_local}]
Recall that $(p,\a,s,q)$ are as in Assumption \ref{Assumptions_coefficients}, $(\psi_k)_{k\in \N}$ satisfy Assumption \ref{Assumptions_noise_local} with $(s_{\psi},q_{\psi})=(s,q)$ and $\inf_{ \Tor^d} u_0>0$ a.s., respectively.  
	
	\textit{Existence of a positive local unique solution.} 
	Set
	$$
	\O_1\, :=\,\Bigl\{\inf_{\T^d} u_0\geq 1\Bigr\}, 
	\quad \text{ and }\quad  \O_j\, :=\,\Bigl\{\tfrac{1}{j+1}\leq \inf_{\T^d}u_0 < \tfrac{1}{j}\Bigr\} \ \ \text{ for } j\geq 1.
	$$
	From the positivity assumption on the initial data, it follows that $\P( \cup_{j\in \N}\O_j )=1$. Hence, to construct a positive local $(p,\a,s,q)$-solution to \eqref{Eq101} it is enough to construct a solution on $\O_j$ for $j\in \N$.
	To this end, recall that, Proposition \ref{Wellposedness_truncated_problem} ensures the existence of maximal unique $(p,\a,s,q)$-solution $(u^{(j)},\sigma^{(j)})$ to 
	\begin{equation}\label{Eq102_modified_proof_maximal_sol_original}
	\begin{cases}
	\dd u^{(j)}\,=\, \big[-\div( m_{2j+3}(u^{(j)}) \nabla \Delta u^{(j)})\, 
	+\, \div(\Phi_{2j+3}(u^{(j)})\nabla u^{(j)} )\big]\, \dd t \,+\, \sum_{k\in \N}\div( g_{2j+3}(u^{(j)}) \psi_k )\, \dd \beta^{(k)}, \\ 
	u^{(j)}(0)\,=\,\one_{\O_j} u_0.
	\end{cases}
\end{equation}
	Now, let us define 
	$$
	\tilde{\sigma}^{(j)}\,:=\,\inf\Bigl\{t\in [0,\sigma^{(j)})\,:\, \inf_{\T^d} u(t,\cdot)\leq \tfrac{2}{2j+3}\Bigr\}
	$$
	where $\inf\emptyset :=\sigma^{(j)}$. Note that $\tilde{\sigma}^{(j)}>0$ a.s.\ on $\O_j$ by construction, and a.e.\ on $[0,\tilde{\sigma}^{(j)})$,
	$$
	m_{2j+3}(u^{(j)})\,=\,
	m(u^{(j)}), \quad
	\Phi_{2j+3}(u^{(j)})\,=\,
	\Phi(u^{(j)}), \quad
	\quad
	g_{2j+3}(u^{(j)})\,=\,
	g(u^{(j)}).
	$$
	In particular, for $j\in\N$, letting  
	\begin{equation}
	\tilde{\sigma}\,:=\,\tilde{\sigma}^{(j)}\ \text{ on }\O_j \quad \text{ and }\quad u\,:=\,u^{(j)}\ \text{ on }[0,\tilde{\sigma}_j)\times \O_j,
	\end{equation}
	one can check that $(u,\tilde{\sigma})$ is a positive local $(p,\a,s,q)$-solution to the original equation \eqref{Eq101}. To conclude this step, it remains to discuss its uniqueness, see Definition \ref{Defi_local_unique}. To this end, let $(v,\tau)$ be a positive local $(p,\a,s,q)$-solution to \eqref{Eq101}. We define 
	$$
	\tau^{(j)}\,:=\,\inf\Bigl\{t\in [0,\tau)\,:\, \inf_{\T^d} v(t,\cdot)\leq \tfrac{2}{2j+3}\Bigr\}
	$$
	where, as usual,  $\inf\emptyset :=\tau$. Now, arguing as above, it readily follows that $( v|_{\O_j}, \one_{\O_j} \tau^{(j)})$ is a local $(p,\a,s,q)$-solution to \eqref{Eq102_modified_proof_maximal_sol_original}. By maximality of the solution $(u^{(j)},\sigma^{(j)})$ we have $\tau^{(j)}\leq \sigma^{(j)}$ a.s.\ on $\O_j$ and $v=u^{(j)}$ a.e.\ on $[0,\tau^{(j)})\times \O_j$. Hence, for all $j\in \N$, we have $v=u^{(j)}$ a.e.\ on $[0,\tilde{\sigma}^{(j)}\wedge \tau^{(j)})\times \O_j$. The latter implies $\tilde{\sigma}^{(j)}\wedge \tau^{(j)}= \tilde{\sigma}^{(j)}\wedge \tau$ a.s.\ and therefore $v=u$ a.e.\ on $\llbracket 0,\tilde{\sigma}\wedge \tau\rro$.

\textit{Maximality within the class of local unique solutions.}
In this step, we build a positive unique local $(p,\a,s,q)$-solution to \eqref{Eq101} which is maximal within the set of positive unique local $(p,\a,s,q)$-solutions. To this end, we let
\begin{align}
\label{eq:set_of_local_solution_proof}
\mathcal{T}\,:=\, \Big\{\tau\,:\, &\text{ $\tau$ is a stopping time for which there exists a}\\
\nonumber
&\text{ positive local unique $(p,\a,s,q)$-solution $(v,\tau)$ to \eqref{Eq101}}\Big\}.
\end{align}
Note that the above set is not empty as $\tilde{\sigma}\in \mathcal{T}$, where $\tilde{\sigma}$ is as in the previous step.
Proceeding as in \cite[Step 5b, proof of Theorem 4.5]{AV19_QSEE_1}, the uniqueness requirement for local solutions in $\mathcal{T}$ yields $\tau_0,\tau_1\in \mathcal{T}$ implies $\tau_0\vee \tau_1\in \mathcal{T}$. In particular, by \cite[Theorem A.3]{KS98}, we conclude that  $\sigma\,:=\,\esssup_{\tau\in \mathcal{T}}\tau$ is a stopping time and there exists a positive local $(p,\a,s,q)$-solution $(u,\sigma)$ to \eqref{Eq101} with a localizing sequence $(\sigma_l)_{l\in\N}$, cf., Definition \ref{Defi_Lpk}. Let us conclude by noticing that $\sigma\geq \tilde{\sigma}>0$ a.s.

At this stage, we do not know whether $(u,\sigma)$ constructed in this step is a positive maximal unique $(p,\a,s,q)$-solution as we are not excluding the existence of a positive local (but not unique) $(p,\a,s,q)$-solution $(v,\tau)$ satisfying $\P(\tau>\sigma)>0$. To prove that $(u,\sigma)$ is actually a positive maximal unique $(p,\a,s,q)$-solution to \eqref{Eq101} we employ a blow-up criterion for $(u,\sigma)$ as constructed above, cf., \cite[Remark 5.6]{AV19_QSEE_2}.

\textit{A blow-up criterion.} Let $(u,\sigma)$ be the positive unique local $(p,\a,s,q)$-solution constructed in the previous step. Then, for all $T<\infty$,
\begin{equation}
\label{eq:blow_up_intermediate_local_proof_original_problem}
\P \Bigl(\sigma<T\,,\, \sup_{t\in [0,\sigma)} \| u(t)\|_{ \Xap}<\infty\,,\,\inf_{[0,\sigma)\times \T^d} u>0\Bigr)\,=\, 0.
\end{equation}
We prove the claim by contradiction using the maximality among unique solutions of $(u,\sigma)$ and the blow-up criterion \eqref{Eq30}. To begin, let us assume that \eqref{eq:blow_up_intermediate_local_proof_original_problem} is false, and therefore, for some $T_{\star}>0$ and $j_{\star}\in \N$,
$$
\P (\mathcal{O}_{\star})\,>\, 0 \quad \text{ where } \quad \mathcal{O}_{\star} \,:=\,
\Bigl\{\sigma<T_{\star}\,,\, \sup_{t\in [0,\sigma)} \| u(t)\|_{ \Xap}<\infty\,,\,\inf_{[0,\sigma)\times \T^d} u\geq \tfrac{2}{j_{\star}}\Bigr\}.
$$
Let $\tau_{\star}$ be the following stopping time
\begin{equation}\label{Eq71}
\tau_{\star}\,:=\,\inf\Bigl\{t\in [0,\sigma)\,:\, \inf_{\T^d} u(t,\cdot)\leq \tfrac{2}{j_{\star}}\Bigr\}
\end{equation}
with $\inf\emptyset\,:=\, \sigma$. By construction $\tau_{\star}\,=\,\sigma$ on $\mathcal{O}_{\star}$. By the choice of the cut-off in the regularized equation \eqref{Eq102}, $(u,\tau_{\star})$ is a local $(p,\a,s,q)$-solution to \eqref{Eq102} for all $j\,\geq \, j_{\star}$.

Now, by Proposition \ref{Wellposedness_truncated_problem}, there exists maximal unique $(p,\a,s,q)$-solution $(u_{\star},\sigma_{\star})$ to \eqref{Eq102} with $j\,=\, j_{\star}+1$. By maximality of $(u_{\star},\sigma_{\star})$, it follows that $\tau_{\star}\,\leq \,\sigma_{\star}$ a.s.\ and $u\,=\,u_{\star}$ a.e.\ on $\llbracket 0,\tau_{\star}\rro$. The latter fact and the definition of $\mathcal{O}_{\star}$ yield
\begin{align*}
\P(\{\tau_{\star}\,=\,\sigma_{\star}\}\cap \mathcal{O}_{\star})
\,\le\, \P\Bigl( \{\tau_{\star}\,=\,\sigma_{\star}\} \cap 
\Bigl\{\sigma_{\star}<T_{\star}\,,\, \sup_{t\in [0,\sigma_{\star})} \| u_{\star}(t)\|_{ \Xap}<\infty\Bigr\}\,
\Bigr)\stackrel{\eqref{Eq30}}{= }\,0.
\end{align*}
Thus, $\sigma_{\star}>\tau_{\star}$ a.s.\ on $\O_{\star}$. Consider the following stopping time 
\begin{equation*}
\tau_{\star\star} \,:=\,\inf\Bigl\{t\in [0,\sigma_{\star})\,:\, \inf_{\T^d} u_{\star}(t,\cdot)\leq \tfrac{2}{j_{\star}+1}\Bigr\}
\end{equation*}
with $\inf\emptyset:= \sigma_{\star}$. Since $\tau_{\star}<\sigma_{\star}$ and $\tau_{\star}=\sigma$ a.s.\ on $\mathcal{O}_{\star}$, it follows that $\tau_{\star\star}>\sigma$ a.s.\ on $\mathcal{O}_{\star}$. Arguing as in the first step of the current proof, by maximality of $(u_{\star},\sigma_{\star})$, one can check that $(u_{\star},\tau_{\star\star})$ is a positive local unique $(p,\a,s,q)$-solution to the original problem \eqref{Eq101} which extends $(u,\sigma)$ as $\tau_{\star\star}>\sigma$ a.s.\ on a set of positive probability $\mathcal{O}_{\star}$. This contradicts the maximality of $\sigma$ in the set of positive local unique solutions $\mathcal{T}$, see \eqref{eq:set_of_local_solution_proof}. Therefore \eqref{eq:blow_up_intermediate_local_proof_original_problem} holds.

\textit{Maximality in the class of positive local solutions.} Now we show that the positive local unique $(p,\a,s,q)$-solution $(u,\sigma)$ to \eqref{Eq101} constructed above is actually maximal. Indeed, let $(v,\tau)$ be another positive local $(p,\a,s,q)$-solution to \eqref{Eq101}. By uniqueness of $(u,\sigma)$ we have $u=v$ a.s.\ on $\llbracket 0,\tau\wedge \sigma\rro$. Hence, it remains to prove that $\tau\leq \sigma$ a.s. By the regularity of local $(p,\a,s,q)$-solutions, a.s.\ on 
$\{\sigma\,<\,\tau\}$, we have 
$$
u\,=\,v\in C((0,\sigma];\Xap)\quad \text{ and }\quad \inf_{(0,\sigma)\times \T^d} u\,=\, \inf_{(0,\sigma)\times \T^d} v\,>\,0.
$$ 
Therefore, for all $T<\infty$,
\begin{align*}
\P(\sigma<\tau\,,\,\sigma< T)\,=\,\P \Bigl(\{\sigma<\tau\}\cap \Bigl\{\sigma<T\,,\, \sup_{t\in [0,\sigma)} \| u(t)\|_{ \Xap}<\infty\,,\,\inf_{[0,\sigma)\times \T^d} u>0\Bigr\}\Bigr)\,\stackrel{\eqref{eq:blow_up_intermediate_local_proof_original_problem}}{=}\, 0.
\end{align*}
The arbitrariness of $T<\infty$ implies that $\tau\leq\sigma$ a.s.\ on $\{\sigma=\infty\}$. Hence $\tau\leq \sigma$ a.s.\ as desired.

\textit{Additional regularity.}
Next, we assume that $p>2$ and prove the additional assertions regarding the regularity of the positive maximal unique $(p,\a,s,q)$-solution $(u, \sigma)$. To this end, for all $j\geq 1$, let 
$$
\tau_{j}\,:=\,\inf\Bigl\{t\in [0,\sigma)\,:\, \inf_{\T^d} u(t,\cdot)\leq \tfrac{2}{j}\Bigr\} \quad \text{ with }\quad \inf\emptyset\,:=\,\sigma.
$$
Arguing as in the previous step, $(u|_{\llbracket 0,\tau_j\rro },\tau_{j})$ is a local $(p,\a,s,q)$-solution to \eqref{Eq102}. Hence, it is extended by the maximal unique $(p,\a,s,q)$-solution to \eqref{Eq102} provided by Proposition \ref{Wellposedness_truncated_problem} with $j$ replaced by $j+1$, and admits consequently the regularity stated in Proposition \ref{Wellposedness_truncated_problem}. 
\end{proof}

\begin{proof}[Proof of Proposition \ref{Prop_regularization}]  
Analogously to the proof of the regularity assertion of Theorem \ref{Thm_local}, by a stopping time argument and the maximality of solutions to \eqref{Eq102}, we conclude that $u$ inherits the regularity from the solutions to the regularized problems stated in Proposition \ref{Prop_regularization_trunc}.
\end{proof}

\begin{proof}[Proof of Proposition \ref{prop:blow_up_criteria}]
The proof is analogous to that of \eqref{eq:blow_up_intermediate_local_proof_original_problem}, where instead of the blow-up criterion \eqref{Eq30} in Proposition \ref{Wellposedness_truncated_problem}, one uses the blow-up criteria of Corollary \ref{cor_blow_up_regularized}.
\end{proof}

\section{Global well-posedness in one dimension}
The aim of this section is to show the global well-posedness of the stochastic thin-film equation \eqref{Eq100}. As laid out in the introduction, \eqref{Eq100} can be cast into the form \eqref{Eq101} if one sets
\begin{align}
\label{Eq47}
	\Phi(u)\,=\,m(u)\phi''(u),\qquad 
	g(u)\,=\, m^{1/2}(u),
\end{align}
so that Theorem \ref{Thm_local} yields that the equation is well-posed locally in time. It remains to use the blow-up criterion from Proposition \ref{prop:blow_up_criteria} to deduce that the unique solution exists even globally in time. We achieve this by first establishing an a-priori estimate on the $\alpha$-entropy \eqref{Eq46} and subsequently estimating the energy \eqref{Eq2} along the trajectory of a solution. To this end,  we restrict ourselves to $d=1$ and $s\in (1/2,1]$ and impose  Assumption \ref{Assumptions_noise_global} on the noise and Assumption \ref{Assumptions_m} and \ref{Assumptions_phi} on the mobility function $m$ and the effective interface potential $\phi$, respectively. Accordingly, we denote the exponent of degeneracy and growth exponent of $m$ by $n$ and $\nu$, and let $\vt$ and $c_0$  be as in \eqref{Eq44} and allow for all implicit constants in this section to depend on $(\psi_k)_{k\in \N}$, $m$ and $\phi$ and in particular on $n$, $\nu$, $\vt$, and $c_0$. Moreover, we fix an initial value $u_0\in L_{\mathscr{F}_0}^{0}(\Omega;H^{s}(\Tor))$ with $\inf_{\Tor} u_0>0$ a.s.\ and let $(u,\sigma)$ be the maximal unique positive local $(2,0,s,2)$-solution to \eqref{Eq101} with coefficients \eqref{Eq47} given by Theorem \ref{Thm_local}. We focus our analysis on the It\^o formulation, i.e., Theorem \ref{Thm_global}, which is more delicate since no cancellations occur. The proof of Theorem \ref{cor_global_stratonovich} on the Stratonovich formulation can be obtained analogously by estimating the terms due to the Stratonovich correction in the same way as the It\^o correction terms. We recall both the energy and the $\alpha$-entropy functional
\begin{align}
	&
	\EE(u)\,=\, \int_{\Tor}\Big[ \tfrac{1}{2}|u_x|^2 \,+\, \phi(u)\Big]\, \dd x,
	\\&
	\H_\beta (u)\,=\, \int_{\T} h_{\beta} (u)\, \dd x ,\qquad h_{\beta}(r)\,=\, \int_1^r \int_1^{r'}\frac{(r'')^\beta}{m(r'')} \, \dd r''\, \dd r',\qquad \beta \in (-{1}/{2},1),
\end{align}
which appear in the following a-priori estimate on the energy at the heart of Theorem \ref{Thm_global}.
\begin{lemma}\label{Lemma_Energy_Est}
	For $0<t_0<T<\infty$, $q\in [1,\infty)$ and $\max\{0,\nu-5\}<\beta<1$ holds the energy estimate
	\begin{align}\begin{split}\label{Eq67}&
			\E\biggl[
			\mathbf{1}_\Gamma \sup_{t_0\le t<\sigma\wedge T} \EE^q(u(t))
			\biggr] \,+\, \E\biggl[
			\biggl(\int_{t_0}^{\sigma\wedge T} \int_{\T} \mathbf{1}_\Gamma m(u)(u_{xx}-\phi'(u))_x^2\,\dd x\,\dd t \biggr)^q
			\biggr]
			\\&\quad \lesssim_{\beta,q,T} \E \bigl[\mathbf{1}_\Gamma \EE^q(u(t_0))\,+\, \mathbf{1}_{\Gamma}
			{\H}_\beta^{6q/(5+\beta-\tilde{\nu})}(u(t_0))\,+\,  \mathbf{1}_{\Gamma}
			{\H}_0^{(\vt-2)q/(\vt+2\tilde{n}-6)}(u(t_0))
			\bigr] 
			\\&\qquad \,+\,  \E\biggl[	\mathbf{1}_{\Gamma}\biggl(
			\int_{\Tor} u_0\, \dd x \biggr)^{\min \bigl\{\frac{6(\beta-\vt)q}{5+\beta-\tilde{\nu}}
				,\frac{-\vt(\vt-2)q}{\vt+2\tilde{n}-6}
				\bigr\}} \mkern-16mu+\,\mathbf{1}_{\Gamma}\biggl( 
			\int_{\Tor} u_0\, \dd x \biggr)^{\max\bigl\{\frac{9(\beta+3)q}{5+\beta-\tilde{\nu}}, \frac{9(\vt-2)q}{2\vt+4\tilde{n}-12}\bigr\}}
			\biggr],
		\end{split}
	\end{align}
	where $\tilde{n}=\min\{2,n\}$, $\tilde{\nu}=\max\{3,\nu\}$ and
	\begin{equation}
		\Gamma\,=\, \biggl\{ \EE(u(t_0))\,+\, \H_\beta(u(t_0)) \,+\,\H_0(u(t_0)) \,\le\, l,\, \tfrac{1}{l}\,\le \, \int_{\T} u_0\, \dd x \,\le\, l, \sigma >t_0 \biggr\}
	\end{equation}
	for some $l\in \N$.
\end{lemma}
We proceed to show an $\alpha$-entropy estimate in Subsection \ref{Sec_proof_alpha_entr}, which is used to give the proof of  Lemma \ref{Lemma_Energy_Est} in Subsection \ref{Sec_Energy_Est}. Before all this, however, we demonstrate how Lemma \ref{Lemma_Energy_Est} leads to the global well-posedness result Theorem \ref{Thm_global}.

\begin{proof}[Proof of Theorem \ref{Thm_global}]
Any process  satisfying \eqref{eq:regularity_1d_global_0}--\eqref{eq:thin_film_ito_formulation} constitutes a $(2,0,s,2)$-solution to \eqref{Eq101} with infinite lifetime. Therefore, the claim follows if we can verify that the positive maximal unique $(2,0,s,2)$-solution $(u,\sigma)$ to \eqref{Eq101} provided by Theorem \ref{Thm_local} satisfies a.s.\ $\sigma=\infty$. Indeed, this immediately yields the existence of a process satisfying \eqref{eq:regularity_1d_global_0}--\eqref{eq:thin_film_ito_formulation}, and uniqueness follows by the uniqueness part of Definition \ref{Defi_max}.  The additional regularity assertions \eqref{eq:regularity_1d_global_2} and \eqref{eq:regularity_1d_global_3} will then follow from Proposition \ref{Prop_regularization} and the fact that Assumption \ref{Assumptions_noise_global} implies Assumption \ref{Assumptions_noise_local} for $s_{\psi}=1$ and any $q_{\psi}<\infty$.

To this end, we recall that $u$ preserves mass, as follows by integrating \eqref{Eq101}, and therefore
\begin{equation}\label{Eq68}
\sup_{t\in (0,\sigma)}\int_{\T} u(t,x)\,\dd x\, =\, \int_{\T} u_0(x)\, \dd x.
\end{equation}
Moreover, Lemma \ref{Lemma_Energy_Est} and Proposition \ref{Prop_regularization} show that, for all $0<\varepsilon<T<\infty$ and a.s.\ on $\{\sigma>\varepsilon\}$, 
\begin{equation}
\label{eq:energy_bounds_proof_global}\noeqref{eq:energy_bounds_proof_global}
\sup_{t\in [\varepsilon,\sigma\wedge T)}\EE(u(t))<\infty.
\end{equation}
Combining the previous with \eqref{Eq45}, which is rigorously stated in Lemma \ref{Lemma_max_min_Energy}, and $\vt>2$ as imposed in Assumption \ref{Assumptions_phi} we deduce that
\begin{align}
\label{eq:positivity_bounds_proof_global}
\Bigl(\inf_{[\varepsilon,\sigma\wedge T)\times\T} u\Bigr)^{-1}\,
=\,
\sup_{t\in [\varepsilon,\sigma\wedge T)}\sup_{x\in \T} \, u^{-1}
\,\lesssim \,\sup_{t\in [\varepsilon,\sigma\wedge T)}\EE^{2/(\vt-2)}(u) \,+\, \biggl(\int_{\T} u_0(x)\, \dd x\biggr)^{-1}\,<\,\infty,
\end{align}
 a.s.\ on $\{\sigma>\varepsilon\}$.
In particular, $u$ is bounded from below on $(\varepsilon,\sigma\wedge T)\times\T$. 
Since Assumption \ref{Assumptions_noise_global} implies Assumption \ref{Assumptions_noise_local} for $s_{\psi}=1$ and all $q_{\psi}<\infty$, Proposition \ref{prop:blow_up_criteria} with $s_0=1$, $p_0=q_0=2$ and $\a_0=0$ is applicable, and combined with \eqref{Eq68}--\eqref{eq:positivity_bounds_proof_global} and the Poincar\'e-Wirtinger inequality, it yields
$$
\P(\varepsilon<\sigma<T)\,=\,0 \ \  \text{for all $0<\varepsilon<T<\infty$}.
$$ 
Now the fact that $\sigma=\infty$ a.s.\ follows by letting $\varepsilon\searrow 0$ and $T\nearrow\infty$ as well as $\sigma>0$ a.s.\ by Theorem \ref{Thm_local}.
\end{proof}

\subsection{$\alpha$-entropy estimates}\label{Sec_proof_alpha_entr}
As a tool to close an $\alpha$-entropy estimate, we show how to control the minimum and maximum of a function in terms of the $\alpha$-entropy dissipation, see \cite[Lemma 4.1 ]{fischer_gruen_2018} and \cite[Lemma 2.3]{dareiotis2023solutions} for similar estimates. 
\begin{lemma}\label{Lemma_max_min_est} Let $\beta\in (-1/2,1)$, then it holds 
	\begin{align}&\label{Eq48}
	\sup_{x\in \Tor} f^{\beta-\vt}(x)\, \le\, \tfrac{(\beta-\vt)^2}{2}\int_{\Tor} f^{\beta-\vt -2} f_x^2\, \dd x \,+\,2\biggl(
	\int_{\Tor} f\, \dd x
	\biggr)^{\beta-\vt}\mkern-26mu ,
	\\&
	\label{Eq13}
	\sup_{x\in \Tor}f^{\beta+5}(x)\,\lesssim_\beta \, \biggl(\int_{\Tor}
	f^{\beta-2}f_x^4\, \dd x
	\biggr)\biggl(\int_{\Tor} f\, \dd x
	\biggr)^{3}\!\!+\,\biggl( \int_{\Tor} f\, \dd x\biggr)^{\beta+5}
\end{align}
	for every positive function $f\in C^1(\T)$ bounded away from $0$.
\end{lemma}
\begin{proof}
By the fundamental theorem of calculus
	\begin{align}
		&
		\sup_{x\in \T} f^{(\beta-\vt)/2}(x)\,-\, \inf_{x\in \T} f^{(\beta-\vt)/2}(x)\,\le \, \biggl(
		\int_{\T} |(f^{(\beta-\vt)/2})_x|^2\, \dd x
		\biggr)^{1/2}
		\\&
		\quad = \,
		\tfrac{\vt-\beta}{2}\biggl(
		\int_{\T} f^{\beta-\vt-2} f_x^2\, \dd x
		\biggr)^{1/2}\mkern-18mu ,
	\end{align}
	and therefore 
	\begin{align}
			\sup_{x\in \T} f^{(\beta-\vt)/2}(x) \,\le\, 	\tfrac{\vt-\beta}{2}\biggl(
			\int_{\T} f^{\beta-\vt-2} f_x^2\, \dd x
			\biggr)^{1/2}\!\!+\, \biggl(
			\int_{\T} f\, \dd x
			\biggr)^{(\beta-\vt)/2}\mkern-50mu .
	\end{align}
	By squaring both sides, we conclude \eqref{Eq48}.
	
	For \eqref{Eq13}, we introduce the function
	\begin{equation}
		g\,=\, f^{(\beta+2)/4}\,-\,\int_{\T} f^{(\beta+2)/4}\, \dd x
	\end{equation}
	so that 
	\begin{equation}\label{Eq49}
		\int_{\T} |g|^{4/(\beta+2)}\,\dd x\,\lesssim_\beta \, \int_{\T} f\, \dd x.
	\end{equation}
	We can estimate again by the fundamental theorem of calculus
	\[
	\sup_{x\in \T} |g|^{(\beta+5 )/(\beta+2)}(x)\,\lesssim_\beta\, \int_{\T} |g_x||g|^{3/(\beta+2)}\, \dd x \, \le \, \biggl(
	\int_{\T} |g_x|^4\, \dd x
	\biggr)^{1/4} \biggl(\int_{\T}
	|g|^{4/(\beta+2)}\, \dd x
	\biggr)^{3/4}\!\!\!\!,
	\]
	because $g$ is mean free. Thus, we deduce that
	\begin{align}
		&
		\sup_{x\in \T} f^{(\beta+5)/4}(x)\,=\, \biggl(\sup_{x\in \T} f^{(\beta+2)/4}(x)\biggr)^{(\beta+5)/(\beta+2)}\mkern-20mu = \, 
		\biggl(\sup_{x\in \T}g(x) \,+\,  \int_{\T} f^{(\beta+2)/4}\, \dd x\biggr)^{(\beta+5)/(\beta+2)}
		\\&\qquad
		\lesssim_\beta \, \sup_{x\in \T} |g|^{(\beta+5)/(\beta+2)}(x)\,+\, \biggl(\int_{\T} f\, \dd x\biggr)^{(\beta+5)/4}
			\\&\qquad \lesssim_\beta
			\biggl(
			\int_{\T} |g_x|^4\, \dd x
			\biggr)^{1/4} \biggl(\int_{\T}
			|g|^{4/(\beta+2)}\, \dd x
			\biggr)^{3/4} \,+\, \biggl(\int_{\T} f\, \dd x\biggr)^{(\beta+5)/4}
	\end{align}
	and it remains to insert $g_x = \frac{\beta+2}{4} f^\frac{\beta-2}{4}f_x  $ and \eqref{Eq49} and raise both sides to the power $4$. 
\end{proof}
	To give the following proof, we recall the additional regularity properties of $u$ stated in Proposition \ref{Prop_regularization}, justifying all the performed integrations by parts.
	\begin{lemma}\label{Lemma_alpha_Entr}
		Let $0<t_0<T<\infty$, $q\in [1,\infty)$, $\beta\in (-1/2,1)$ and 
		\begin{equation}\label{Eq10}
			\gamma \, \in \,\Bigl[\tfrac{
				\beta+2 - \sqrt{(1-\beta)(1+2\beta)}}{3}, \tfrac{
				\beta+2 + \sqrt{(1-\beta)(1+2\beta)}}{3}
			\Bigr],
		\end{equation}
	then holds the $\alpha$-entropy estimate
	\begin{align}\begin{split}\label{Eq16}
			&
			\E \biggl[ \mathbf{1}_{\Gamma}
			\sup_{t_0\le t< \sigma\wedge T} 	\H_\beta^q(u(t))
			\biggr]\,+\, \E \biggl[\biggl( \int_{t_0}^{\sigma\wedge T } \int_{\Tor} \mathbf{1}_{\Gamma}
			u^{\beta-\vt-2} u_x^2
			\, \dd x\, \dd t\biggr)^q
			\biggr]
			\\&+\,\E\biggl[
			\biggl( \int_{t_0}^{\sigma\wedge T}\int_{\Tor}\mathbf{1}_{\Gamma} u^{\beta-2\gamma+2}(u^\gamma)_{xx}^2\, \dd x\, \dd t\biggr)^q
		\,
			+\, 
			\biggl( \int_{t_0}^{\sigma\wedge T }\int_{\Tor}\mathbf{1}_{\Gamma}
			u^{\beta-2}u_x^4\, \dd x\, \dd t\biggr)^q
			\biggr]
			\\&\quad 
			\lesssim_{\beta,\gamma ,q,T}\,
			\E\bigl[ \mathbf{1}_{\Gamma}
			{\H}_\beta^q(u(t_0))
			\bigr]
			\,+\,  \E\biggl[	\mathbf{1}_{\Gamma}\biggl(
			\int_{\Tor} u_0\, \dd x \biggr)^{(\beta-\vt)q} \mkern-16mu+\,\mathbf{1}_{\Gamma}\biggl( 
			\int_{\Tor} u_0\, \dd x \biggr)^{3(\beta+3)q/2}
			\biggr]
			,
		\end{split}
	\end{align}
for any  $\mathscr{F}_{t_0}$-measurable subset $\Gamma$ of  \begin{equation}
 \biggl\{ \H_\beta(u(t_0)) \,\le\, l,\, \tfrac{1}{l}\,\le \, \int_{\T} u_0\, \dd x \,\le\, l, \sigma >t_0 \biggr\}
\end{equation} for some $l\in \N$.
	\end{lemma}
\begin{proof}
	For a localizing sequence  $0\le \sigma_j \nearrow \sigma$  for $(u,\sigma)$ as in Definition \ref{Defi_Lpk} we define
	\begin{equation}\label{Eq55}
		\tilde{\sigma}_j \,=\, \mathbf{1}_\Gamma\inf\Bigl\{
		t\in [t_0,\sigma_j\wedge T ]:\, \inf_{x\in \Tor}u(t)\,\le \, \tfrac{1}{j}\,\text{ or }\, \|u(t)\|_{C^2(\T)} \,+\, \|u\|_{L^2((t_0, t); H^{3}(\T))}\,\ge \, j
		\Bigr\}\,+\, \mathbf{1}_{\Gamma^c}t_0,
	\end{equation}
	so that $\tilde{\sigma}_j \nearrow \sigma \wedge T$ as $j\to \infty$ on $\Gamma$ by Proposition \ref{Prop_regularization}. Correspondingly, we define the process
	\begin{equation}\label{Eq56}
		u^{(j)}(t)\,=\, \mathbf{1}_{\Gamma}u(t\wedge \tilde{\sigma}_j)\,+\, \mathbf{1}_{\Gamma^c}\mathbf{1}_{\T}
	\end{equation}
	for $t\in [t_0, T]$ and let $\tilde{h}_\beta\colon \R\to \R$ be twice continuously differentiable with bounded second derivative such that $\tilde{h}_\beta = h_\beta$ on $[{1}/{j}, j]$. Then the assumptions of \cite[Proposition A.1]{DHV_16} are satisfied and an application of It\^o's formula yields that
	\begin{align}&
		\int_{\T} \tilde{h}_\beta(u^{(j)}(t)) \, \dd x\,=\, \int_{\T} \tilde{h}_\beta(u^{(j)}(t_0)) \, \dd x \,+\, \int_{t_0}^{t} \int_{\T} \mathbf{1}_{[t_0,\tilde{\sigma}_j]\times \Gamma} \tilde{h}_\beta''(u^{(j)})u^{(j)}_x m(u^{(j)})(u^{(j)}_{xxx} - \phi''(u^{(j)})u^{(j)}_x)\, \dd x\, \dd r
		\\&\quad +\, \tfrac{1}{2} \sum_{k\in \N}  \int_{t_0}^t
		\int_{\T}\mathbf{1}_{[t_0,\tilde{\sigma}_j]\times \Gamma}
		\tilde{h}_\beta''(u^{(j)}) (g(u^{(j)})\psi_k)_x^2
		\, \dd x
		\, \dd r\,+\, \sum_{k\in \N} \int_0^{t}\int_{\T} \mathbf{1}_{[t_0,\tilde{\sigma}_j]\times \Gamma}
			\tilde{h}_\beta'(u^{(j)})  (g(u^{(j)})\psi_k)_x\, \dd x\, \dd \beta^{(k)}.
	\end{align}
	Moreover, we can replace again $\tilde{h}_\beta$ by $h_\beta$ since they coincide on the range of $u^{(j)}$ and $u^{(j)}$ by $u$ to conclude that
	\begin{align}\begin{split}\label{Eq50}
			&
		\mathbf{1}_\Gamma\H_\beta(u(t\wedge \tilde{\sigma}_j)) \,=\, \mathbf{1}_\Gamma \H_\beta(u(t_0))\,+\, \int_{t_0}^{t} \int_{\T} \mathbf{1}_{[t_0,\tilde{\sigma}_j]\times \Gamma} u^\beta u_x (u_{xxx} - \phi''(u)u_x)\, \dd x\, \dd r
		\\&
		\quad +\, 
		\tfrac{1}{2} \sum_{k\in \N}  \int_{t_0}^t
		\int_{\T}\mathbf{1}_{[t_0,\tilde{\sigma}_j]\times \Gamma}
		h_\beta''(u) (g(u)\psi_k)_x^2
		\, \dd x
		\, \dd r\,+\, \sum_{k\in \N} \int_0^{t}\int_{\T} \mathbf{1}_{[t_0,\tilde{\sigma}_j]\times \Gamma}
		{h}_\beta'(u)  (g(u)\psi_k)_x\, \dd x\, \dd \beta^{(k)}
		\end{split}
	\end{align}
	for $t\in [t_0, T]$. Using the classical calculation for the deterministic $\alpha$-entropy \cite[Proposition 2.1]{Beretta_Bertsch_DalPasso_95} and the assumption \eqref{Eq44} on $\phi''$, we deduce
	\begin{align}&
		\int_{\T} u^\beta u_x (u_{xxx} - \phi''(u)u_x)\, \dd x \, \le \,-\tfrac{1}{\gamma^2}\int_{\T} u^{\beta+2-2\gamma} (u^\gamma)_{xx}^2\, \dd x \,-\, c(\beta, \gamma)\int_{\T} u^{\beta-2}u_x^4\, \dd x\\&\quad-\, c_1 \int_{\T} u^{\beta-\vt-2}u_x^2\, \dd x\,+\, c_2\int_{\T} u^\beta u_x^2 \, \dd x
	\end{align}
	on $[t_0,\tilde{\sigma}_j]\times \Gamma$ for a constant $c(\beta,\gamma)>0$. To estimate the It\^o correction, we calculate that
	\begin{align}
		(g(u)\psi_k)_x^2\,\lesssim\, (g'(u))^2u_x^2 \psi_k^2\,+\, m(u)(\psi_k')^2.
	\end{align}  Since
	\begin{equation}\label{Eq57}
		(g'(u))^2\,=\, \Bigl(\tfrac{m'(u)}{2m^{1/2}(u)}\Bigr)^2\,\lesssim\, u^{-2}m(u)
	\end{equation}
	by  \eqref{Eq42}, we obtain
	\begin{align}
		\sum_{k\in \N}\int_{\T} h_\beta''(u)(g(u)\psi_k)_x^2\, \dd x\,\lesssim \,
		 \int_{\T} u^{\beta-2}u_x^2\,+\, u^\beta\, \dd x.
	\end{align}
	Using Young's inequality twice, we arrive at 
	\begin{align}&
	\int_{\T} u^\beta u_x (u_{xxx} - \phi''(u)u_x)\, \dd x\,+\, \sum_{k\in \N}\int_{\T} h_\beta''(u)(g(u)\psi_k)_x^2 \, \dd x \\&\quad\le \, -\tfrac{1}{\gamma^2}\int_{\T} u^{\beta+2-2\gamma} (u^\gamma)_{xx}^2\, \dd x \,-\, \tfrac{c(\beta, \gamma)}{2}\int_{\T} u^{\beta-2}u_x^4\, \dd x \,-\,\tfrac{c_1}{2} \int_{\T} u^{\beta-\vt-2}u_x^2\, \dd x\,+\, C \int_{\T} u^{\beta+2}\,+\, u^\beta  \, \dd x
	\end{align}
on $[t_0, \tilde{\sigma}_j]\times \Gamma$.
Inserting this in \eqref{Eq50}, taking the $q$-th power on both sides, the supremum in time, and using the Burkholder--Davis--Gundy inequality, we deduce
\begin{align}
\begin{split}\label{Eq51}&
		\E\biggl[
	\mathbf{1}_{\Gamma} \sup_{t_0\le t \le \tilde{\sigma}_j} \H_\beta^q(u(t))
	\biggr]\,+\, \E \biggl[\biggl( \int_{t_0}^{\tilde{\sigma}_j } \int_{\Tor} \mathbf{1}_{\Gamma}
	u^{\beta-\vt-2} u_x^2
	\, \dd x\, \dd t\biggr)^q
	\biggr]
	\\&+\,\E\biggl[
	\biggl( \int_{t_0}^{\tilde{\sigma}_j}\int_{\Tor}\mathbf{1}_{\Gamma} u^{\beta-2\gamma+2}(u^\gamma)_{xx}^2\, \dd x\, \dd t\biggr)^q
	\,
	+\, 
	\biggl( \int_{t_0}^{\tilde{\sigma}_j }\int_{\Tor}\mathbf{1}_{\Gamma}
	u^{\beta-2}u_x^4\, \dd x\, \dd t\biggr)^q
	\biggr]
	\\&\quad 
	\lesssim_{\beta,\gamma ,q}\,
	\E\bigl[ \mathbf{1}_{\Gamma}
	{\H}_\beta^q(u(t_0))
	\bigr]
	\,+\,  \E\biggl[\biggl(
	\mathbf{1}_{\Gamma}\int_{t_0}^{\tilde{\sigma}_j } 
	\int_{\Tor} u^{\beta+2}\,+\, u^\beta\, \dd x\, \dd t\biggr)^q
	\biggr]
	\\&\qquad
	+\, \E\biggl[\biggl(	\mathbf{1}_{\Gamma}
	\sum_{k\in \N}
	\int_{t_0}^{\tilde{\sigma}_j}\biggl(\int_{\T}
	h_{\beta}''(u)u_x g(u) \psi_k \, \dd x\biggr)^2
	\, \dd t
	\biggr)^{q/2}
	\biggr]
	.
\end{split}
\end{align}
To estimate the latter term, we integrate by parts and obtain that
\begin{equation}
\biggl(\int_{\T}
h_{\beta}''(u)u_x g(u) \psi_k \, \dd x\biggr)^2 \,=\, \biggl(
\int_{\T} \int_1^u
h_{\beta}''(r) g(r)\, \dd r \,   \psi_k' \, \dd x
\biggr)^2\,\le \, 
\int_{\T} \biggl(\int_1^u
r^\beta m^{-1/2}(r)\, \dd r\biggr)^2 \,   (\psi_k')^2 \, \dd x.
\end{equation}
Using \eqref{Eq40} and \eqref{Eq41}, we estimate separately
\begin{equation}
	\biggl(\int_1^u
	r^\beta m^{-1/2}(r)\, \dd r\biggr)^2\, \lesssim\, 
	\biggl(\int_1^u
	r^\beta \, \dd r\biggr)^2\,\lesssim_\beta\, u^{2\beta+2} \quad \text{ on }\{u>1\}
\end{equation}
and 
\begin{equation}
	\biggl(\int_1^u
	r^\beta m^{-1/2}(r)\, \dd r\biggr)^2\, \lesssim\, 
	\biggl(\int_1^u
	r^{\beta-n/2-1} \, \dd r\biggr)^2\,\lesssim_\beta\, \begin{cases}
		u^{2\beta-n},& 2\beta-n<0,\\
		\log^2(u), & 2\beta-n\ge 0 
	\end{cases} \quad \text{ on }\{u\le 1\}.
\end{equation}
Thus, we arrive at 
\begin{align}
\sum_{k\in \N}
\int_{t_0}^{\tilde{\sigma}_j}\biggl(\int_{\T}
h_{\beta}''(u)u_x g(u) \psi_k \, \dd x\biggr)^2
\, \dd t\,\lesssim_\beta\, 
\int_{t_0}^{\tilde{\sigma}_j}\int_{\{ u>1 \}}u^{2\beta+2} \, \dd x\, \dd t\,+\, \int_{t_0}^{\tilde{\sigma}_j}
\int_{\{ u\le 1 \}}u^{2\beta-n} +\log^2(u) \, \dd x
\, \dd t
\end{align}
on $[t_0, \tilde{\sigma}_j]\times \Gamma$.
To estimate the power $u^{2\beta-n+2}$, we use that 
\begin{equation}
	h_\beta(r)\,\gtrsim\, \int_1^r\int_1^{r'} (r'')^{\beta-n}\, \dd r''\, \dd r'\,=\, \begin{cases}
		\tfrac{1}{\beta-n+1}\Bigl(\tfrac{r^{\beta-n+2}-1}{\beta-n+2}-r+1\Bigr),& \beta-n \notin \{-1,-2\},\\
		r-1-\log(r),& \beta-n = -2, \\
		r\log(r) - r+1, &\beta-n=-1
	\end{cases}
\end{equation}
by \eqref{Eq40}
for $r\le 1$ such that
\begin{equation}
	\int_{\{u\le 1\}} u^{\beta-n+2}\, \dd x\, \lesssim_\beta\, \int_{\{u\le 1\}} h_\beta(u) + (u+1)\, \dd x \, \le \, \H_\beta(u)\,+\, \int_{\T} u_0\, \dd x\,+\, 1
\end{equation}
due to the positivity of $u$ and conservation of mass. Therefore, we deduce that
\begin{align}&\int_{t_0}^{\tilde{\sigma}_j}
	\int_{\{u\le 1\}} u^{2\beta-n} \, \dd x
	\, \dd t\, \le \, \biggl(\sup_{t_0\le t \le \tilde{\sigma}_j} \int_{\{u\le 1\}}u^{\beta-n+2}\, \dd x\biggr) \biggl(\int_{t_0}^{\tilde{\sigma}_j}
	\sup_{x\in\T}
	u^{\beta-2}	
	\, \dd t\biggr)\\&\quad 
	\le \epsilon\biggl(
	\sup_{t_0\le t \le \tilde{\sigma}_j} 
	\H_\beta(u)\,+\, \int_{\T} u_0\, \dd x\,+\, 1
	\biggr)^2 +\, C_{\beta,\epsilon} \biggl(\int_{t_0}^{\tilde{\sigma}_j}
	\sup_{x\in\T}
	u^{\beta-2}	
	\, \dd t\biggr)^2
\end{align}
for any $\epsilon>0$.
Inserting all this in \eqref{Eq51} yields
\begin{align}
	\begin{split}\label{Eq52}&
		\E\biggl[
		\mathbf{1}_{\Gamma} \sup_{t_0\le t \le \tilde{\sigma}_j} \H_\beta^q(u(t))
		\biggr]\,+\, \E \biggl[\biggl( \int_{t_0}^{\tilde{\sigma}_j } \int_{\Tor} \mathbf{1}_{\Gamma}
		u^{\beta-\vt-2} u_x^2
		\, \dd x\, \dd t\biggr)^q
		\biggr]
		\\&+\,\E\biggl[
		\biggl( \int_{t_0}^{\tilde{\sigma}_j}\int_{\Tor}\mathbf{1}_{\Gamma} u^{\beta-2\gamma+2}(u^\gamma)_{xx}^2\, \dd x\, \dd t\biggr)^q
		\,
		+\, 
		\biggl( \int_{t_0}^{\tilde{\sigma}_j }\int_{\Tor}\mathbf{1}_{\Gamma}
		u^{\beta-2}u_x^4\, \dd x\, \dd t\biggr)^q
		\biggr]
		\\&\quad 
		\lesssim_{\beta,\gamma ,q}\,
		\E\bigl[ \mathbf{1}_{\Gamma}
		{\H}_\beta^q(u(t_0))
		\bigr]
		\,+\,  \E\biggl[\biggl(
		\mathbf{1}_{\Gamma}
		\int_{\Tor} u_0\, \dd x \biggr)^q \,+\, \mathbf{1}_{\Gamma}
		\biggr]
		\\&\qquad
		+\, \E\biggl[\mathbf{1}_{\Gamma}\biggl(	
		\int_{t_0}^{\tilde{\sigma}_j}
		\sup_{x\in \T} \Bigl(u^{\beta+2} \,+\, u^{\beta-2}	 \,+\, u^{2\beta+2}\,+\, \log^2(u)\Bigr)
		\, \dd t
		\biggr)^{q}
		\biggr]
		.
	\end{split}
\end{align}
To estimate also the latter term, we observe that all the powers lie between $\beta-\vt$ and $\beta+3$, so that an application of Young's inequality leads to 
\begin{align}
		\sup_{x\in \T} \Bigl(u^{\beta+2} \,+\, u^{\beta-2}	 \,+\, u^{2\beta+2}\,+\, \log^2(u)\Bigr)\,\le\, \epsilon \sup_{x\in \T} u^{\beta-\vt}\,+\, C_{\beta,\epsilon}\sup_{x\in \T} u^{\beta+3}
\end{align}
for each $\epsilon>0$. By applying Lemma \ref{Lemma_max_min_est}, once more, Young's inequality, and conservation of mass, we can estimate this further by
\begin{align}&
	 \tfrac{\epsilon(\beta-\vt)^2}{2}\int_{\Tor} u^{\beta-\vt -2} u_x^2\, \dd x \,+\,2\epsilon\biggl(
	\int_{\Tor} u_0\, \dd x
	\biggr)^{\beta-\vt}\\&+\, C_{\beta,\epsilon}
  \biggl(\int_{\Tor}
	u^{\beta-2}u_x^4\, \dd x
	\biggr)^{(\beta+3)/(\beta+5)}\biggl(\int_{\Tor} u_0\, \dd x
	\biggr)^{3(\beta+3)/(\beta+5)}\!\!+\, C_{\beta,\epsilon}\biggl( \int_{\Tor}u_0\, \dd x\biggr)^{\beta+3}
	\\&\quad \le \,  \tfrac{\epsilon(\beta-\vt)^2}{2}\int_{\Tor} u^{\beta-\vt -2} u_x^2\, \dd x \,+\,2\epsilon\biggl(
	\int_{\Tor} u_0\, \dd x
	\biggr)^{\beta-\vt}\\&\qquad 
	+\, \epsilon \biggl(\int_{\Tor}
	u^{\beta-2}u_x^4\, \dd x
	\biggr) \,+\, C_{\beta,\epsilon}\biggl(\int_{\Tor} u_0\, \dd x
	\biggr)^{3(\beta+3)/2}\!\!+\, C_{\beta,\epsilon}\biggl( \int_{\Tor} u_0\, \dd x\biggr)^{\beta+3}\mkern-26mu.
\end{align}
Choosing $\epsilon$ sufficiently small to absorb the resulting terms in the left-hand side of \eqref{Eq52} and dropping the intermediate powers of the mass yields 
\begin{align}
	\begin{split}&
		\E\biggl[
		\mathbf{1}_{\Gamma} \sup_{t_0\le t \le \tilde{\sigma}_j} \H_\beta^q(u(t))
		\biggr]\,+\, \E \biggl[\biggl( \int_{t_0}^{\tilde{\sigma}_j } \int_{\Tor} \mathbf{1}_{\Gamma}
		u^{\beta-\vt-2} u_x^2
		\, \dd x\, \dd t\biggr)^q
		\biggr]
		\\&+\,\E\biggl[
		\biggl( \int_{t_0}^{\tilde{\sigma}_j}\int_{\Tor}\mathbf{1}_{\Gamma} u^{\beta-2\gamma+2}(u^\gamma)_{xx}^2\, \dd x\, \dd t\biggr)^q
		\,
		+\, 
		\biggl( \int_{t_0}^{\tilde{\sigma}_j }\int_{\Tor}\mathbf{1}_{\Gamma}
		u^{\beta-2}u_x^4\, \dd x\, \dd t\biggr)^q
		\biggr]
		\\&\quad 
		\lesssim_{\beta,\gamma ,q,T}\,
		\E\bigl[ \mathbf{1}_{\Gamma}
		{\H}_\beta^q(u(t_0))
		\bigr]
		\,+\,  \E\biggl[
		\mathbf{1}_{\Gamma}\biggl(
		\int_{\Tor} u_0\, \dd x \biggr)^{(\beta-\vt)q} \mkern-16mu+\,\mathbf{1}_{\Gamma}\biggl(
		\int_{\Tor} u_0\, \dd x \biggr)^{3(\beta+3)q/2}
		\biggr]
		.
	\end{split}
\end{align}
We obtain the claimed estimate \eqref{Eq16} by letting $j\to \infty$ and using Fatou's lemma.
\end{proof}

\subsection{Proof of the energy estimate -- Lemma \ref{Lemma_Energy_Est}}\label{Sec_Energy_Est}
To close the a-priori estimate from Lemma \ref{Lemma_Energy_Est} on the energy of the solution $u$, we use additionally an estimate on the minimum and maximum of a function in terms of the energy functional. Moreover, we recall once more the additional regularity properties of $u$ from Proposition \ref{Prop_regularization} as well as that there exists $\vt >2$ with $r^{-\vt} \lesssim \phi(r)$ by Assumption \ref{Assumptions_phi}.
\begin{lemma}\label{Lemma_max_min_Energy} It holds
	\begin{align}
		&		\label{Eq54}
		\sup_{x\in \Tor} f^{(2-\vt)/2}(x)\,\lesssim \, \EE(f)\,+\, \biggl(\int_{\Tor} f \, \dd x\biggr)^{(2-\vt)/2}\mkern-26mu,
		\\& \label{Eq60}
		\sup_{x\in \Tor} f^3(x)\,\lesssim\, \EE(f) \biggl(  \int_{\Tor} f\, \dd x\biggr)\,+\,   \biggl(\int_{\Tor} f\, \dd x\biggr)^3
	\end{align}
	for every positive function $f\in C^1(\T)$ bounded away from $0$.
\end{lemma}
\begin{proof}
	The proof is completely analogous to the proof of Lemma \ref{Lemma_max_min_est} using additionally \eqref{Eq44}, for a discrete version of \eqref{Eq54} see also \cite[Lemma 4.1]{fischer_gruen_2018}.
\end{proof}

\begin{proof}[Proof of Lemma \ref{Lemma_Energy_Est}]We again make use of the stopping time $\tilde{\sigma}_j$ introduced  in \eqref{Eq55} and the process $u^{(j)}$ defined in \eqref{Eq56}. As in the proof of Lemma \ref{Lemma_alpha_Entr}, we  replace $\phi$ by a two times differentiable function $\tilde{\phi}\colon \R \to \R$ with bounded second derivative, which agrees with $\phi$ on $[{1}/{j}, j]$. Then the assumptions of \cite[Proposition A.1]{DHV_16} are satisfied and an application of It\^o's formula yields 
	\begin{align}&
	\int_{\T} \tilde{\phi}(u^{(j)}(t)) \, \dd x\,=\, \int_{\T} \tilde{\phi}(u^{(j)}(t_0)) \, \dd x \,+\, \int_{t_0}^{t} \int_{\T} \mathbf{1}_{[t_0,\tilde{\sigma}_j]\times \Gamma} \tilde{\phi}''(u^{(j)})u^{(j)}_x m(u^{(j)})(u^{(j)}_{xxx} - \phi''(u^{(j)})u^{(j)}_x)\, \dd x\, \dd r
	\\&\quad +\, \tfrac{1}{2} \sum_{k\in \N}  \int_{t_0}^t
	\int_{\T}\mathbf{1}_{[t_0,\tilde{\sigma}_j]\times \Gamma}
	\tilde{\phi}''(u^{(j)}) (g(u^{(j)})\psi_k)_x^2
	\, \dd x
	\, \dd r\,+\, \sum_{k\in \N} \int_{t_0}^{t}\int_{\T} \mathbf{1}_{[t_0,\tilde{\sigma}_j]\times \Gamma}
	\tilde{\phi}'(u^{(j)})  (g(u^{(j)})\psi_k)_x\, \dd x\, \dd \beta^{(k)}.
\end{align}
It\^o's formula is also applicable to the functional $\|u^{(j)}_x\|_{L^2(\T)}^2$. Indeed,  as carried out in detail in \cite[Appendix C]{Sauerbrey_2021}, one can for example identify $u^{(j)}$ with its equivalence class $\bar{u}^{(j)}$ of homogeneous distributions such that the functional $\|u_x\|_{L^2(\T)}^2$ coincides with the squared $\dot{H}^1(\T)$-norm of $\bar{u}^{(j)}$ and \cite[Theorem 4.2.5]{LiuRock} becomes applicable leading to
\begin{align}&
	\tfrac{1}{2}\int_{\T}(u^{(j)}_x)^2(t)\, \dd x\,=\, \tfrac{1}{2}\int_{\T}(u^{(j)}_x)^2(t_0)\, \dd x\,-\, 
	\int_{t_0}^{t} \int_{\T} \mathbf{1}_{[t_0,\tilde{\sigma}_j]\times \Gamma} u^{(j)}_{xxx} m(u^{(j)})(u^{(j)}_{xxx} - \phi''(u^{(j)})u^{(j)}_x)
	\, \dd x\, \dd r
	\\&
	\quad +\, \tfrac{1}{2} \sum_{k\in \N}  \int_{t_0}^t
	\int_{\T}\mathbf{1}_{[t_0,\tilde{\sigma}_j]\times \Gamma} (g(u^{(j)})\psi_k)_{xx}^2
	\, \dd x
	\, \dd r\,-\, \sum_{k\in \N} \int_{t_0}^{t}\int_{\T} \mathbf{1}_{[t_0,\tilde{\sigma}_j]\times \Gamma}
	u^{(j)}_{xx} (g(u^{(j)})\psi_k)_{x}\, \dd x\, \dd \beta^{(k)}.
\end{align}
By inserting the definition of $u^{(j)}$, using that $\tilde{\phi}$ and $\phi$ coincide on its range and adding the two It\^o expansions together, we conclude that
\begin{align}\begin{split}\label{Eq58}
		&
		\mathbf{1}_\Gamma\EE(u(t\wedge \tilde{\sigma}_j)) \,=\, \mathbf{1}_\Gamma \EE(u(t_0))\,-\, \int_{t_0}^{t} \int_{\T} \mathbf{1}_{[t_0,\tilde{\sigma}_j]\times \Gamma} m(u)  (u_{xx} - \phi'(u))_x^2\, \dd x\, \dd r
		\\&
		\quad +\, 
		\tfrac{1}{2} \sum_{k\in \N}  \int_{t_0}^t
		\int_{\T}\mathbf{1}_{[t_0,\tilde{\sigma}_j]\times \Gamma}
	\Bigl[	\phi''(u) (g(u)\psi_k)_x^2\,+\, (g(u)\psi_k)_{xx}^2\Bigr]
		\, \dd x
		\, \dd r\\&\quad +\, \sum_{k\in \N} \int_{t_0}^{t}\int_{\T} \mathbf{1}_{[t_0,\tilde{\sigma}_j]\times \Gamma}
		(\phi'(u)-u_{xx})  (g(u)\psi_k)_x\, \dd x\, \dd \beta^{(k)}
	\end{split}
\end{align}
for $t\in [t_0, T]$. To estimate the It\^o correction, we use \eqref{Eq40}, \eqref{Eq41}, \eqref{Eq44} and again \eqref{Eq57} to deduce that
\begin{align}&
	\sum_{k\in \N} \int_{\T}
	\phi''(u)(g(u)\psi_k)_x^2
	\, \dd x\, \lesssim \, \sum_{k\in \N}
	\int_{\T} \big(\, u^{-\vt-2}(g'(u))^2u_x^2 \psi_k^2\,+\, u^{-\vt-2} m(u)(\psi_k')^2\,\big)\, \dd x
	\\&
	\quad \lesssim\, \int_{\T}
	\big(\,u^{-\vt-4}m(u)u_x^2\,+\, u^{-\vt-2}m(u)\,\big)\, \dd x
	\\&\quad \lesssim\, \int_{\{u>1\}} \big(\, u^{\nu-\vt-4}u_x^2\,+\, u^{\nu-\vt-2}\,\big)\, \dd x
	\,+\, \int_{\{u\le 1\}} \big(\,u^{n-\vt-4}u_x^2\,+\, u^{n-\vt-2}\,\big)\, \dd x.
\end{align}
Similarly, using that
\begin{equation}
	(g(u)\psi_k)_{xx}\,=\, (g'(u)u_{xx}+g''(u)u_x^2)\psi_k\,+\, 2 g'(u)u_x\psi_k'\,+\, g(u)\psi_k''
\end{equation}
we obtain
\begin{align}\label{Eq59}
	\sum_{k\in \N}
	\int_{\T} (g(u)\psi_k)_{xx}^2\, \dd x\,\lesssim \, \int_{\T}\big[\,
	(g'(u))^2u_{xx}^2\,+\, (g''(u))^2 u_x^4\, +\, (g'(u))^2u_x^2\, +\, m(u)
	\,\big]\, \dd x.
\end{align}
Because of \eqref{Eq42} it holds
\begin{equation}
(g''(u))^2\,=\, \Bigl(
\tfrac{2m''(u)m(u)\,-\,(m'(u))^2 }{4m^{3/2}(u)}
\Bigr)^2\,\lesssim \, u^{-4}m(u),
\end{equation}
which together with
\eqref{Eq40}, \eqref{Eq41} and \eqref{Eq57} yields that \eqref{Eq59} is bounded by
\begin{align}&
	\int_{\T} \big(\,u^{-2}m(u)u_{xx}^2\,+\, u^{-4}m(u)u_x^4\,+\, u^{-2}m(u)u_x^2\,+\, m(u)\,\big)\,\dd x
	\\&\quad\lesssim
	\, \int_{\T}\big(\, u^{-2}m(u)u_{xx}^2\,+\, u^{-4}m(u)u_x^4\,+\, m(u) \,\big)\, \dd x
	\\&\quad \lesssim\, \int_{\{u>1\}}\big(\,u^{\nu-2}u_{xx}^2\,+\, u^{\nu-4}u_x^4\,+\, u^{\nu} \,\big)\, \dd x\,+\, 
	\int_{\{u\le 1\}} \big(\,u^{n-2} u_{xx}^2\,+\, u^{n-4} u_x^4\,+\, u^{n} \, \big)\,\dd x.
\end{align} 
In summary, we have verified the estimate
\begin{align}&
	\sum_{k\in \N}  
	\int_{\T}\big(\,	\phi''(u) (g(u)\psi_k)_x^2\,+\, (g(u)\psi_k)_{xx}^2\,\big)
	\, \dd x
	\\&\quad 
	\lesssim\, 
	 \int_{\{u>1\}} \big(\, u^{\nu-\vt-4}u_x^2\,+\, u^{\nu-\vt-2}\, +\, u^{\nu-2}u_{xx}^2\,+\, u^{\nu-4}u_x^4\,+\, u^{\nu}\,\big)\, \dd x
	\\&\qquad +\, \int_{\{u\le 1\}} \big(\, u^{n-\vt-4}u_x^2\,+\, u^{n-\vt-2} \,+\,  u^{n-2} u_{xx}^2\,+\, u^{n-4} u_x^4\,+\, u^{n} \,\big)\, \dd x
\end{align}
on $[t_0, \tilde{\sigma}_j]\times \Gamma$. Inserting this in \eqref{Eq58} and taking the supremum in time, we arrive at
\begin{align}\begin{split}\label{Eq64}&
	\mathbf{1}_\Gamma \sup_{t_0\le t\le \tilde{\sigma}_j} \EE(u(t))\,+\, \int_{t_0}^{\tilde{\sigma}_j}\int_{\T} \mathbf{1}_{ \Gamma} m(u)  (u_{xx} - \phi'(u))_x^2\, \dd x\, \dd t
	\\&\quad \lesssim\, \mathbf{1}_\Gamma \EE(u(t_0)) \,+\,  \mathbf{1}_\Gamma\int_{t_0}^{\tilde{\sigma}_j}
	\int_{\{u>1\}}\big(\, u^{\nu-\vt-4}u_x^2\,+\, u^{\nu-\vt-2}\, +\, u^{\nu-2}u_{xx}^2\,+\, u^{\nu-4}u_x^4\,+\, u^{\nu}\,\big)\, \dd x
	\, \dd t
	\\&\qquad +\,\mathbf{1}_\Gamma\int_{t_0}^{\tilde{\sigma}_j}
	\int_{\{u\le 1\}} \big(\,u^{n-\vt-4}u_x^2\,+\, u^{n-\vt-2} \,+\,  u^{n-2} u_{xx}^2\,+\, u^{n-4} u_x^4\,+\, u^{n} \,\big)\, \dd x
	\,\dd t
	\\&\qquad +\, \sup_{t_0\le t\le T} \biggl|
	\sum_{k\in \N} \int_{t_0}^{t}\int_{\T} \mathbf{1}_{[t_0,\tilde{\sigma}_j]\times \Gamma}
	(\phi'(u)-u_{xx})  (g(u)\psi_k)_x\, \dd x\, \dd \beta^{(k)}
	\biggr|.\end{split}
\end{align} 
For the sake of clarity, we divide the rest of the proof into three steps.

\textit{Estimate on the integral over $\{u>1\}$.} We define $\tilde{\nu}= \max\{3,\nu\}$ and estimate using \eqref{Eq48}, \eqref{Eq13}, \eqref{Eq60} and Young's inequality 
\begin{align}\begin{split}\label{Eq78}
	&
	\int_{\{u>1\}} \big(\,u^{\nu-\vt-4}u_x^2\,+\, u^{\nu-\vt-2}\, +\, u^{\nu-2}u_{xx}^2\,+\, u^{\nu-4}u_x^4\,+\, u^{\nu}\,\big)\, \dd x
	\\&\quad\le\, 
	\int_{\T}\big(\, u^{\tilde{\nu}-\vt-4}u_x^2\,+\, u^{\tilde{\nu}-\vt-2}\, +\, u^{\tilde{\nu}-2}u_{xx}^2\,+\, u^{\tilde{\nu}-4}u_x^4\,+\, u^{\tilde{\nu}}\,\big)\, \dd x
	\\&\quad \le \,
\Bigl(	\sup_{x\in \T} u^{\tilde{\nu}- \beta-2} \Bigr)
	\int_{\T}\big(\, u^{\beta-\vt-2}u_x^2\,+\, u^{\beta-\vt}\, +\, u^{\beta}u_{xx}^2\,+\, u^{\beta-2}u_x^4\,+\, u^{\beta+2}\,\big)\, \dd x
	\\&\quad \lesssim_\beta \,
	\biggl(
	\EE^{(\tilde{\nu}-\beta-2)/3}(u)\biggl(
	 \avg
	\biggr)^{(\tilde{\nu}-\beta-2)/3}\,+\, 
	\biggl(
	\avg
	\biggr)^{\tilde{\nu}-\beta-2}
	\biggr)
	\\&\qquad \times
	{\biggl(
	\int_{\T} \big(\, u^{\beta-\vt-2}u_x^2\, +\, u^{\beta}u_{xx}^2\,+\, u^{\beta-2}u_x^4\, \big)\,\dd x
	\,+\, \biggl(\avg\biggr)^{\beta-\vt}\,+\, \biggl(
	\avg
	\biggr)^{\beta+2}
	\biggr)}.
\end{split}
\end{align}
We denote the terms in the last row as $\I(u)$ and obtain, due to Young's inequality, the bound
\begin{align}\begin{split}\label{Eq62}&
	\int_{t_0}^{\tilde{\sigma}_j}
	\int_{\{u>1\}}\big(\, u^{\nu-\vt-4}u_x^2\,+\, u^{\nu-\vt-2}\, +\, u^{\nu-2}u_{xx}^2\,+\, u^{\nu-4}u_x^4\,+\, u^{\nu}\,\big)\, \dd x
	\, \dd t
	\\&\quad \lesssim_\beta 
	\biggl(
	\sup_{t_0\le t\le \tilde{\sigma}_j}\EE^{(\tilde{\nu}-\beta-2)/3}(u) \,+\, \biggl(
	\avg
	\biggr)^{2(\tilde{\nu}-\beta-2)/3}
	\biggr)\times \biggl(
	\avg
	\biggr)^{(\tilde{\nu}-\beta-2)/3}\mkern-16mu \times  \int_{t_0}^{\tilde{\sigma}_j} \I(u)\, \dd t
	\\&\quad 
	\le 
	\epsilon \biggl( \sup_{t_0\le t\le \tilde{\sigma}_j}\EE(u)\,+\, \biggl(\avg\biggr)^2\biggr)\,+\,
	C_{\beta,\epsilon} \biggl(
	\biggl(\avg\biggr)^{(\tilde{\nu}-\beta -2)/3} \times \int_{t_0}^{\tilde{\sigma}_j} \I(u)\, \dd t
	\biggr)^{3/(5+\beta-\tilde{\nu})} 
		\\&
	\quad 
	\le\, \epsilon \biggl( \sup_{t_0\le t\le \tilde{\sigma}_j}\EE(u)\,+\, \biggl(\avg\biggr)^2\biggr) \\&\qquad +\, C_{\beta,\epsilon} \biggl( \biggl(\avg\biggr)^{2(\tilde{\nu}-\beta-2)/(5+\beta-\tilde{\nu})}\mkern-16mu \,+\,
\biggl(	\int_{t_0}^{\tilde{\sigma}_j}\, 
	\I(u)\, \dd t\biggr)^{6/(5+\beta-\tilde{\nu})}
	\biggr)
	\end{split}
\end{align} 
on 
$\Gamma$
for any $\epsilon>0$. To estimate the $\I(u)$-term later on, we remark that
\begin{align}\begin{split}\label{Eq65}&
	\E\biggl[\mathbf{1}_\Gamma\biggl(
	\int_{t_0}^{\tilde{\sigma}_j}\I(u)\, \dd t\biggr)^{6q/(5+\beta-\tilde{\nu})}
	\biggr]
	\,\lesssim_{\beta,q,T}\, \E\bigl[ \mathbf{1}_{\Gamma}
	{\H}_\beta^{6q/(5+\beta-\tilde{\nu})}(u(t_0))
	\bigr]
	\\&\quad +\,  \E\biggl[	\mathbf{1}_{\Gamma}\biggl(
	\int_{\Tor} u_0\, \dd x \biggr)^{6(\beta-\vt)q/(5+\beta-\tilde{\nu})} \mkern-16mu+\,\mathbf{1}_{\Gamma}\biggl( 
	\int_{\Tor} u_0\, \dd x \biggr)^{9(\beta+3)q/(5+\beta-\tilde{\nu})}
	\biggr]\end{split}
\end{align}
by an application of Lemma \ref{Lemma_alpha_Entr} with $\gamma=1$ and additionally absorbing the intermediate power of the mass.

\textit{Estimate on the integral over $\{u\le 1\}$.} We set $\tilde{n} =\min \{2,n\}$, such that
\begin{align}\begin{split}&
	\int_{\{u\le 1\}} \big(\,u^{n-\vt-4}u_x^2\,+\, u^{n-\vt-2} \,+\,  u^{n-2} u_{xx}^2\,+\, u^{n-4} u_x^4\,+\, u^{n}\,\big) \, \dd x
	\\&\quad \le 
	\int_{\T} \big(\,u^{\tilde{n}-\vt-4}u_x^2\,+\, u^{\tilde{n}-\vt-2} \,+\,  u^{\tilde{n}-2} u_{xx}^2\,+\, u^{\tilde{n}-4} u_x^4\,+\, u^{\tilde{n}} \,\big)\, \dd x
		\\&\quad \le \,
	\Bigl(	\sup_{x\in \T} u^{\tilde{n}- 2} \Bigr)
	\int_{\T} \big(\,u^{-\vt-2}u_x^2\,+\, u^{-\vt}\, +\, u_{xx}^2\,+\, u^{-2}u_x^4\,+\, u^{2}\,\big)\, \dd x
	\\&\quad \lesssim \,
	\biggl(
	\EE^{2(2-\tilde{n})/(\vt-2)}(u) \,+\, \biggl(
	\avg
	\biggr)^{\tilde{n}-2}
	\biggr)
	\\&\qquad \times
	{\biggl(
		\int_{\T}\big(\, u^{-\vt-2}u_x^2\, +\, u_{xx}^2\,+\, u^{-2}u_x^4\,\big)\, \dd x
		\,+\, \biggl(\avg\biggr)^{-\vt}\,+\, \biggl(
		\avg
		\biggr)^{2}
		\biggr)}. \label{Eq61}\end{split}
\end{align}
because of \eqref{Eq48}, \eqref{Eq13} and \eqref{Eq54}. We define $\J(u)$ as the last row of \eqref{Eq61} and observe that for $\tilde{n} =2$ its prefactor equals $1$, such that the following estimate trivializes. Otherwise we have $0 < \frac{2(2-\tilde{n})}{\vt-2} <1$ by Assumption \ref{Assumptions_phi} and an application of Young's inequality yields the  bound
\begin{align}\begin{split}\label{Eq63}&
	\int_{t_0}^{\tilde{\sigma}_j}
	\int_{\{u\le 1\}} \big(\, u^{n-\vt-4}u_x^2\,+\, u^{n-\vt-2} \,+\,  u^{n-2} u_{xx}^2\,+\, u^{n-4} u_x^4\,+\, u^{n} \,\big)\, \dd x\, \dd t
	\\&\quad 
	\lesssim \, \biggl(\sup_{t_0\le t\le \tilde{\sigma}_j}\EE(u)^{2(2-\tilde{n})/(\vt-2)}(u)\,+\, 
	\biggl(
	\avg
	\biggr)^{\tilde{n}-2}
	\biggr)\times \int_{t_0}^{\tilde{\sigma}_j} \J(u)\, \dd t
	\\&\quad \le \, 
	\epsilon \biggl(\sup_{t_0\le t\le \tilde{\sigma}_j}
	\EE(u)\,+\, \biggl(
	\avg
	\biggr)^{(2-\vt)/2}
	\biggr) \,+\, C_\epsilon \biggl(\int_{t_0}^{\tilde{\sigma}_j}\J(u)\, \dd t\biggr) ^{(\vt-2)/(\vt+2\tilde{n}-6)}\end{split}
\end{align}
on $\Gamma$.
As in the previous step we remark that Lemma \ref{Lemma_alpha_Entr} with $\beta'=0$, $\gamma'=1$ yields 
\begin{align}\begin{split}\label{Eq66}&
	\E\biggl[\mathbf{1}_\Gamma\biggl(
	\int_0^{\tilde{\sigma}_j}\J(u)\, \dd t\biggr)^{(\vt-2)q/(\vt+2\tilde{n}-6)}
	\biggr]
	\,\lesssim_{\beta,q}\, \E\bigl[ \mathbf{1}_{\Gamma}
	{\H}_0^{(\vt-2)q/(\vt+2\tilde{n}-6)}(u(t_0))
	\bigr]
	\\&\quad +\,  \E\biggl[	\mathbf{1}_{\Gamma}\biggl(
	\int_{\Tor} u_0\, \dd x \biggr)^{-\vt(\vt-2)q/(\vt+2\tilde{n}-6)} \mkern-16mu+\,\mathbf{1}_{\Gamma}\biggl( 
	\int_{\Tor} u_0\, \dd x \biggr)^{9(\vt-2)q/(2\vt+4\tilde{n}-12)}
	\biggr]\end{split}
\end{align}
for later purposes.

\textit{Closing the energy estimate.} Inserting \eqref{Eq62} and \eqref{Eq63} in \eqref{Eq64} and choosing $\epsilon$ sufficiently small  yields 
\begin{align}\begin{split}&
		\mathbf{1}_\Gamma \sup_{t_0\le t\le \tilde{\sigma}_j} \EE(u(t))\,+\, \int_{t_0}^{\tilde{\sigma}_j}\int_{\T} \mathbf{1}_{ \Gamma} m(u)  (u_{xx} - \phi'(u))_x^2\, \dd x\, \dd t
		\\&\quad \lesssim_\beta\, \mathbf{1}_\Gamma \EE(u(t_0)) \,+\,  \mathbf{1}_\Gamma
		\biggl(
		\avg
		\biggr)^{(2-\vt)/2}
		\,+\,\mathbf{1}_\Gamma \biggl(
		\avg\biggr)^{\max \bigl\{  \frac{2(\tilde{\nu}-\beta-2)}{5+\beta-\tilde{\nu}},2 \bigr\}}
		\\&\qquad +\,\mathbf{1}_\Gamma\biggl(	\int_{t_0}^{\tilde{\sigma}_j}\, 
		\I(u)\, \dd t\biggr)^{6/(5+\beta-\tilde{\nu})}\mkern-16mu +\,\mathbf{1}_\Gamma
		\biggl(\int_{t_0}^{\tilde{\sigma}_j}\J(u)\, \dd t\biggr) ^{(\vt-2)/(\vt+2\tilde{n}-6)}
		\\&\qquad +\, \sup_{t_0\le t\le T} \biggl|
		\sum_{k\in \N} \int_{t_0}^{t}\int_{\T} \mathbf{1}_{[t_0,\tilde{\sigma}_j]\times \Gamma}
		(\phi'(u)-u_{xx})  (g(u)\psi_k)_x\, \dd x\, \dd \beta^{(k)}
		\biggr|.\end{split}
\end{align} 
We take the $q$-th power on both sides, use the estimates \eqref{Eq65} and \eqref{Eq66} and the BDG-inequality to estimate
further
\begin{align}\begin{split}&
		\E\biggl[\mathbf{1}_\Gamma \sup_{t_0\le t\le \tilde{\sigma}_j} \EE^q(u(t))\biggr]\,+\,\E\biggl[\biggl( \int_{t_0}^{\tilde{\sigma}_j}\int_{\T} \mathbf{1}_{ \Gamma} m(u)  (u_{xx} - \phi'(u))_x^2\, \dd x\, \dd t\biggr)^q\biggr]
		\\&\quad \lesssim_{\beta,q,T}\,\E \bigl[\mathbf{1}_\Gamma \EE^q(u(t_0))\,+\, \mathbf{1}_{\Gamma}
		{\H}_\beta^{6q/(5+\beta-\tilde{\nu})}(u(t_0))\,+\,  \mathbf{1}_{\Gamma}
		{\H}_0^{(\vt-2)q/(\vt+2\tilde{n}-6)}(u(t_0))
		\bigr] 
		\\&\qquad \,+\,  \E\biggl[	\mathbf{1}_{\Gamma}\biggl(
		\int_{\Tor} u_0\, \dd x \biggr)^{\min \bigl\{\frac{6(\beta-\vt)q}{5+\beta-\tilde{\nu}}
			,\frac{-\vt(\vt-2)q}{\vt+2\tilde{n}-6}
			\bigr\}} \mkern-16mu+\,\mathbf{1}_{\Gamma}\biggl( 
		\int_{\Tor} u_0\, \dd x \biggr)^{\max\bigl\{\frac{9(\beta+3)q}{5+\beta-\tilde{\nu}}, \frac{9(\vt-2)q}{2\vt+4\tilde{n}-12}\bigr\}}
		\biggr]
		\\&\qquad +\, \E\biggl[\biggl(	\mathbf{1}_{\Gamma}
		\sum_{k\in \N}
		\int_{t_0}^{\tilde{\sigma}_j}\biggl(\int_{\T}(\phi'(u)-u_{xx})  (g(u)\psi_k)_x \, \dd x\biggr)^2
		\, \dd t
		\biggr)^{q/2}
		\biggr],\end{split}
\end{align} 
where we absorb the intermediate power of the mass again.
Since 
\begin{align}&
	\biggl(	
	\sum_{k\in \N}
	\int_{t_0}^{\tilde{\sigma}_j}\biggl(\int_{\T}(\phi'(u)-u_{xx})  (g(u)\psi_k)_x \, \dd x\biggr)^2
	\, \dd t
	\biggr)^{q/2}\mkern-6mu \lesssim\, 
	\biggl(	
	\int_{t_0}^{\tilde{\sigma}_j}\int_{\T}(\phi'(u)-u_{xx})_x^2  g^2(u)\, \dd x
	\, \dd t
	\biggr)^{q/2}
	\\&\quad \le \,\epsilon 	\biggl(	
	\int_{t_0}^{\tilde{\sigma}_j}\int_{\T}m(u)(\phi'(u)-u_{xx})_x^2  \, \dd x
	\, \dd t
	\biggr)^{q} \,+\, C_\epsilon
\end{align}
by another application of Young's inequality, we deduce \eqref{Eq67} by taking the limit $j\to \infty$ and using Fatou's lemma.
\end{proof}

\appendix
\renewcommand{\thetheorem}{\Alph{section}.\arabic{theorem}}
\renewcommand{\theequation}{\Alph{section}.\arabic{equation}}

\section{Some functional analytic tools}
In this appendix, we discuss some tools that played an important role throughout the manuscript. Firstly, regarding the precise definition of a UMD and  type $2$ Banach spaces, we refer the reader to \cite[Definition 4.2.1]{Analysis1} and \cite[Definition 7.1.1]{Analysis2}, where for our purposes it suffices to note that the spaces $L^q(\T^d)$ and $H^{s,q}(\T^d)$
are UMD for $q\in (1,\infty)$, cf. \cite[Prop. 4.2.15, Prop. 4.2.17 (1)]{Analysis1}, and type $2$ whenever $q\in [2,\infty)$, see \cite[Prop. 7.1.4]{Analysis2}. The remaining topics are discussed in more detail in their respective subsection below.

\subsection{Banach-valued Sobolev spaces and trace embeddings}
\label{app:Sobolev_spaces}
We begin by defining time-weighted, Banach-valued Sobolev spaces. As in Subsection \ref{ss:notation}, for a complex Banach space $\mathscr{X}$ (see, e.g., \cite[Subsection B.4]{Analysis1} for complexification of real Banach spaces), an open interval $I\subseteq \R$ and a weight $w_{\a}(t)=|t|^\a$ where $\a\in (-1,p-1)$, we let $L^p(I,w_\kappa;\mathscr{X})$ be the set of all strongly measurable mapping $f:I \to X$ such that (see, e.g., \cite[Definition 1.2.1]{Analysis1})
\begin{equation*}%
\textstyle \|f\|_{L^p(I,w_\a;\mathscr{X})} \,:= \,\big(\int_{I} \|f(t)\|_\mathscr{X}^p \,w_\a(t)\, \dd t\big)^{1/p}\,<\,\infty.
\end{equation*}
We also recall that for separable $\mathscr{X}$ strong measurability is equivalent to $t\ni I \mapsto \langle f(t),x^*\rangle $ being measurable for all $x^*\in \mathscr{X}^*$, see  \cite[Theorem 1.1.20]{Analysis1}. 
Our assumption  $\a\in (-1,p-1)$ ensures that $L^{p}(\R,w_\a;\mathscr{X})\embed L^1_{\loc}(\R;\mathscr{X})$, so that using also the rational decay of $w_\a(t)$ any $f\in L^p(\R,w_\a;X)$ can be identified with an $\mathscr{X}$-valued  Schwartz distribution $ f\in \S'(\R;\mathscr{X}):= \calL(\S(\R),\mathscr{X})$, i.e., a continuous linear operator from the topological space $\S(\R)$ to $\mathscr{X}$. As in the scalar-valued case, the Fourier transform $\Fou$ naturally acts on $\S,(\R; \mathscr{X})$ via duality, i.e., $\langle \Fou{f},\varphi\rangle:=\langle f,\overline{\Fou{\varphi}}\rangle$ for $f\in \S'(\R;\mathscr{X}) $ and $\varphi\in \S(\R)$.

We are now ready to define $\mathscr{X}$-valued fractional Sobolev spaces.  
For $\sigma\in \R$, $p\in (1,\infty)$ and $\a\in (-1,p-1)$, we let $H^{\sigma,p}(\R,w_{\a};\mathscr{X})$ be the set of all Schwartz distributions $f\in \S'(\R;\mathscr{X})$ such that 
\begin{equation}
\label{eq:definition_banach_valued_Sob_spaces_R}
(1-\Delta)^{\sigma/2}f\in L^p(\R,w_{\a};\mathscr{X}),\quad 
\text{and }\quad \|f\|_{H^{\sigma,q}(\R,w_\a;\mathscr{X})}:=\|(1-\Delta)^{\sigma/2}f\|_{L^p(I,w_\a;\mathscr{X})},
\end{equation}
where, as usual, $(1-\Delta)^{\sigma/2}$ is the Fourier multiplier with symbol $(1+|\xi|^2)^{\sigma/2}$.

For an interval $I\subseteq \R$, let $\D'(I;\mathscr{X}):= \calL(\D(I);\mathscr{X})$ the set of all $\mathscr{X}$-valued distributions. We set
\begin{equation}
\label{eq:definition_banach_valued_Sob_spaces}
H^{\sigma,p}(I,w_{\a};\mathscr{X}):=\{g\in \D'(I;\mathscr{X})\,:\,\exists f \in H^{\sigma,p}(I,w_{\a};\mathscr{X}) \text{ such that }  g=f|_{I}\}
\end{equation}
endowed with the natural quotient norm $\|g\|_{H^{\sigma,p}(I,w_{\a};\mathscr{X})}:=\inf_{f|_I =g} \|f\|_{H^{\sigma,p}(I,w_{\a};\mathscr{X})}$. Finally, we often employ the shorthand notation $H^{\sigma,p}(0,T,w_{\a};\mathscr{X}):=H^{\sigma,p}(I,w_{\a};\mathscr{X})$ if $I=(0,T)$ for $T< \infty$.

Before going further, we comment on the above construction. Firstly, the definition of the $\mathscr{X}$-valued Sobolev spaces can be extended to $\R^d$ and $\T^d$, where in the periodic case  $\S(\R)$ and the Fourier transform are replaced by the Fr\'echet space $C^\infty(\T^d)$ and the Fourier series.
Secondly, the well-known characterization of real-valued Sobolev spaces through difference quotients holds if and only if $\mathscr{X}$ is a Hilbert space. In a non-Hilbertian setting, difference quotients define the related (but different) scale of Triebel--Lizorkin spaces. For details on the latter, the reader is referred to \cite[Section 3.2]{ALV19} and \cite[Subsection 14.3]{Analysis3}. Therefore, the approach via the Fourier transform and restrictions as in \eqref{eq:definition_banach_valued_Sob_spaces_R}--\eqref{eq:definition_banach_valued_Sob_spaces} is the most direct definition in our case. Alternatively, one can proceed via complex interpolation (cf., \cite[Subsection 2.3]{AVSurvey}), which yields the same space if $\mathscr{X}$ is UMD (see \cite[Theorem 5.6.9]{Analysis1}).

We now turn our attention to the so-called \emph{trace embedding} for anisotropic vector-valued spaces, which was mentioned in Subsection \ref{ss:local_all_dimensions_intro}, and plays an important role in the approach to SPDEs developed in \cite{AV19_QSEE_1,AV19_QSEE_2}. We emphasize that this embedding motivates the presence of Besov spaces as spaces for the initial data in Theorem \ref{Thm_local}.
To set the stage, let us recall that in the maximal $L^p$-regularity approach to SPDEs with $p\in (2,\infty)$ and time weight $\a\in [0,\frac{p}{2}-1)$, the solution typically lies in the space 
\begin{equation}
\label{eq:smr_space}
\textstyle \bigcap_{\theta\in [0,1/2)} H^{\theta,p}(0,T,w_{\a};X_{1-\theta}),
\end{equation}
cf. Definition \ref{def:SMRgeneralized},
where $X_{1-\theta}=[X_0,X_1]_{1-\theta}$, $X_0$ is the ground space, $X_1$ the domain of the leading order  operator and $[\cdot,\cdot]_{1-\theta}$ denotes the complex interpolation functor.
Note that weighted Sobolev embeddings ensure that $H^{\theta,p}(0,T,w_{\a};X_{1-\theta})\subseteq C^{\theta-\frac{1+\a}{p}}(0,T;X_{1-\theta})$ if $\theta>\frac{1+\a}{p}$ (cf., \cite[Proposition 7.4]{MV12}). In particular, if $u$ belong to the space in \eqref{eq:smr_space}, then it has well-defined trace $u(0)\in X_{\sigma}$ for all $\sigma <1- \frac{1+\a}{p}$. As in scalar-valued spaces, in the threshold case, the embedding $H^{\frac{1+\a}{p},p}(0,T,w_{\a};X_{1-\frac{1+\a}{p}})$ into $ C([0,T];X_{1-\frac{1+\a}{p}})$ fails. To some extend, the trace embedding can be seen as a fix of the latter situation allowing one to reach the optimal smoothness $\theta=\frac{1+\a}{p}$ at the expense of working with real interpolation spaces instead of the complex ones, which naturally appear already in the case of integer smoothness, see \cite[Subsection 3.12]{BeLo}.
The key in the trace embedding is to exploit the presence in \eqref{eq:smr_space} of a sufficiently large smoothness parameter $\theta\in (\frac{1+\a}{p},\frac{1}{2})$ to define traces in $X_{1-\theta}$ together with the highest possible regularity in space $L^p(0,T,w_{\a};X_1)$, achieved by taking $\theta=0$. The following is a consequence of \cite[Theorem 1.2]{ALV19} as well as \cite[(2.8) and Proposition 3.2]{ALV19} to pass from $(0,T)$ to the half-line, see also \cite[Section 3.4]{pruss2016moving}.

\begin{proposition}[Trace embedding with fractional smoothness]
\label{prop:tracespace}
Fix $p\in (1,\infty)$, $\a\in [0,p-1)$, $\theta\in (0,1)$ and $0<T\leq \infty$. Let $X_0,X_1$ be Banach spaces such that $X_1\hookrightarrow X_0$ and let $\Xap=(X_0,X_1)_{1-\frac{1+\a}{p},p}$ and $\Xp=\Xzp$ be real interpolation spaces. Then the following embeddings hold:
\begin{enumerate}[{\rm(1)}]
\item\label{it:trace_with_weights_Xap} If $\theta>\frac{1+\a}{p}$, then
$
H^{\theta,p}(0,T,w_{\a};X_{1-\theta})\cap L^p(0,T,w_{\a};X_1)\hookrightarrow
C([0,T];\Xap).
$
\item\label{it:trace_without_weights_Xp} If $\theta>\frac{1}{p}$, then
$
H^{\theta,p}(0,T,w_{\a};X_{1-\theta})\cap L^p(0,T,w_{\a};X_1)\embed
C((0,T];\Xp).
$
\end{enumerate}
\end{proposition}

 The above is known to be optimal if $X_1$ is the domain of a sectorial operator, see \cite[Proposition 4.1]{MV14}, which is in particular the case in our application, 
 where $X_1=H^{s+2,q}(\T^d)$ is the domain of $\Delta^2$ on $X_0=H^{s-2,q}(\T^d)$, cf. Section \ref{ss:notation}.

We remark that  \eqref{it:trace_without_weights_Xp} encodes the effect of \emph{instantaneous regularization} in time-weighted spaces. Indeed, in the case $\a>0$ and $X_1\neq X_0$, comparing \eqref{it:trace_with_weights_Xap} with \eqref{it:trace_without_weights_Xp}, one sees that a function in $H^{\theta,p}(0,T,w_{\a};X_{1-\theta})\cap L^p(0,T,w_{\a};X_1)$ takes at $t=0$ values in $\Xap$, but for all $t>0$ it lies in the strictly smaller spaces $\Xp$.
This is an important ingredient in both the blow-up criteria and the instantaneous regularization for the regularized problem \eqref{Eq102} derived in Propositions \ref{Wellposedness_truncated_problem}--\ref{Prop_regularization_trunc}.
The intuition behind \eqref{it:trace_without_weights_Xp} is that the weight $w_\a(t)=t^\a$ is active only at $t=0$, and for positive times, one can apply \eqref{it:trace_with_weights_Xap} with $\a=0$; for details see \cite[Theorem 4.7]{ALV19} or \cite[Corollary 7.6]{AV19}.

\subsection{$\g$-radonifying operators}
\label{app:gamma_radonifying}
Let $(\wt{\O},\wt{\A},\wt{\P})$ be a probability space and let $(\wt{\g}_n)_{n\geq 1}$ be a sequence of independent standard Gaussian random variables on $\wt{\O}$. Let $H$ be a separable Hilbert space with orthonormal basis  $(h_n)_{n\geq 1}$. 
A bounded linear operator $T$ from $H $ to a Banach space $\mathscr{X}$ is said to be $\g$-radonifying if the series $\sum_{n\geq 1} \wt{\g}_n T h_n$ converges in $L^2(\wt{\O};\mathscr{X})$, in which case we set
\begin{equation}
\label{eq:gamma_radonifying}
\textstyle
\|T\|_{\g(H,\mathscr{X})}
:=\big(
\wt{\E}\|\sum_{n\geq 1} \wt{\g}_n T h_n\|_{\mathscr{X}}^2\big)^{1/2},
\end{equation}
where $\wt{\E}$ denotes the expectation w.r.t.\ $\wt{ \P}$. One can prove that the space $\g(H,\mathscr{X})$ as well as the norm \eqref{eq:gamma_radonifying} do not depend on the particular choice of the orthonormal basis, see \cite[Theorem 9.1.17]{Analysis2}.
A detailed exposition of $\g$-radonifying spaces can be found in \cite[Chapter 9]{Analysis2}.

Starting from the seminal works \cite{NVW1,NWalpha}, the key role of $\g$-radonifying operators in the theory of stochastic integration in a non-Hilbertian setting has become clear. Indeed, let us point out that if $\mathscr{X}$ is a Hilbert space, then $\g(H,\mathscr{X})$ coincides with the space of Hilbert-Schmidt operators from $H$ to $\mathscr{X}$, so that the results in \cite{NVW1} are an extension of the classical It\^o integral in Hilbert spaces \cite{DPZ}. 

An important tool in the study of $\g$-radonifying operators is the Kahane--Khintchine inequality, stating that for each finite sequence $(x_n)_{n=1}^N\subseteq \mathscr{X}$ and all $0< q,p<\infty$,
\begin{equation}
\label{eq:KK_inequality}
\textstyle
\big(
\wt{\E}\|\sum_{n=1}^N \wt{\g}_n x_n\|_{\mathscr{X}}^q\big)^{1/q}
\eqsim_{q,p}
\big(
\wt{\E}\|\sum_{n=1}^N \wt{\g}_n x_n\|_{\mathscr{X}}^p\big)^{1/p}.
\end{equation}
For a proof of the latter, we refer to \cite[Theorem 6.2.6]{Analysis2}.
In particular, the above shows that the integrability $2$ in the definition of the $\g$-radonifying operators can be replaced by any $p\in [1,\infty)$. We additionally point out that if $\mathscr{X}=\R$,  \eqref{eq:KK_inequality} is nothing but the equivalence of moments for Gaussian random variables. 

In our setting, the two-sided estimate \eqref{eq:KK_inequality} allows to prove a  characterization of $\g(H,\mathscr{X})$ in the case that $\mathscr{X}=H^{\sigma,q}(\T^d)$, which is usually referred to as $\g$-Fubini isomorphism, see \cite[Subsection 9.4.b]{Analysis2}. 

\begin{proposition}
\label{prop:identification_gamma_spaces}
For all $\sigma\in \R$ and $q\in (1,\infty)$, it holds that  
$
H^{\sigma,q}(\T^d;H)
= \g(H,H^{\sigma,q}(\T^d)).
$
More precisely, the following mapping is an isomorphism between the corresponding spaces  
\begin{align*}
\Iso:
H^{\sigma,q}(\T^d;H)
\to
\g(H,H^{\sigma,q}(\T^d)), 
\quad \text{where }\quad (\Iso f)h := (f,h)_H  ;
\end{align*}
and $(\cdot,\cdot)_H$ denotes the inner product in $H$.
\end{proposition}

Recall that the $H$-valued Sobolev space $H^{\sigma,q}(\T^d;H)$ is defined via Fourier methods, see below \eqref{eq:definition_banach_valued_Sob_spaces}.

\begin{proof}[Sketch of the proof]
From a density argument, it follows that $(1-\Delta)^{\sigma/2}$ commutes with $\Iso$, and induces an isomorphism from $H^{\sigma,q}(\T^d;H)$ and $\g(H,H^{\sigma,q}(\T^d))$ to $L^{q}(\T^d;H)$ and $\g(H,L^{q}(\T^d))$, respectively. Thus, it is enough to consider $\sigma=0$, which is a special case of \cite[Theorem 9.4.8]{Analysis2}. For the reader's convenience, we recall the main step to highlight the role of \eqref{eq:KK_inequality} and (the classical version of) Fubini's theorem. 

By density, it suffices to prove the norm equivalence for a simple function $f\in L^q(\T^d;H)$, i.e., $f(x)=\sum_{n=1}^N h_n  f_n(x)$ for a finite sequence $(f_n)_{n=1}^N\subseteq L^q(\T^d)$ and an orthonormal system $(h_n)_{n\geq 1}\subseteq H$, which we can extended to an orthonormal basis of $H$. Then, by Fubini's theorem,
\begin{align*}
\textstyle
\|\Iso f\|_{\g(H,L^q(\T^d))}^q \stackrel{(i)}{\eqsim}_q
\wt{\E}\|\sum_{n=1}^N \wt{\g}_n f_n\|_{L^q(\T^d)}^q
&\textstyle= \int_{\T^d} \wt{\E} \big| \sum_{n=1}^N \wt{\g} f_n(x)\big|^q \,\dd x \\
&\textstyle\stackrel{(ii)}{\eqsim} \int_{\T^d} \big(\wt{\E} \big| \sum_{n=1}^N \wt{\g} f_n(x)\big|^2\big)^{q/2} \,\dd x\\
&\textstyle\stackrel{(iii)}{\eqsim} \int_{\T^d} \big(\sum_{n=1}^N |f_n(x)|^2 \big)^{q/2} \,\dd x = \|f\|_{L^q(\T^d;H)}^q,
\end{align*}
where in $(i)$-$(ii)$ we applied \eqref{eq:KK_inequality} and in $(iii)$ we used the independence of $(\wt{\g}_n)_{n\geq 1}$ and $f_n=(f,h_n)_H$.
\end{proof} 
\subsection{Product and composition estimates}\label{app:para_est} Since we work throughout this manuscript in Bessel potential and Besov spaces, we require, among others, efficient estimates on products and compositions of functions in said scales. While in the case of integer smoothness like $H^{1,q}(\T^d)$ a  product estimates can be derived via
\begin{equation}\label{eq:prod_est_example}
\| \nabla (uv)\|_{L^q(\T^d  ; \R^d)} \,\le\, \| u\nabla v\|_{L^q(\T^d  ; \R^d)} \,+\, \|v\nabla u\|_{L^q(\T^d  ; \R^d)} \,\lesssim\, \| u\|_{H^{1,l}(\T^d)}\|v\|_{L^r(\T^d)}  \,+\, \| v\|_{H^{1,L}(\T^d)}\|u\|_{L^R(\T^d)}  ,
\end{equation}
for $l,L\in (1,\infty)$ and $r, R \in [1,\infty]$ such that $1/l + 1/r = 1/R +1/L =1/q$, and a composition estimate from 
\begin{align}\begin{split}\label{Eq76}
		&
\|\nabla (f(u)  -  f(v)) \|_{L^q(\T^d  ; \R^d)} \,\le \, \| f'(u) (\nabla u -\nabla v)\|_{L^q(\T^d  ; \R^d)} 
\,+\, \| (f'(u) -f'(v))\nabla v\|_{L^q(\T^d  ; \R^d)} 
\\&
\quad \lesssim \, \| f'(u) \|_{L^\infty(\T^d)} \|  u - v\|_{H^{1,q}(\T^d)} 
\,+\, \|  f'(u) -f'(v)\|_{L^\infty(\T^d)} \| v\|_{H^{1,q}(\T^d)}, 
\end{split}
\end{align}
the fractional smoothness case is more delicate.
A frequency decomposition of the product into its \emph{resonant part} and \emph{Bony's paraproduct}, for example, is one way to obtain estimates similar to \eqref{eq:prod_est_example}, see \cite[Chapter 2]{ToolsPDEsTaylor} or \cite[Chapters 4 and 5]{Runst_Sickel} for a more complete review of such techniques. For the purpose of this appendix, we just recall two special cases of  \cite[Proposition 4.1 (1), (3)]{AV21_max_reg_torus} which are repeatedly used in this manuscript.
\begin{proposition}\label{prop:para}
	Let $s>0$, $q \in (1,\infty)$ and $\tau \in (s,\infty)$, then  the following estimates hold whenever the right-hand side is finite:
	\begin{enumerate}[(i)]
		\item \label{item:prop_para_1}
		$\|uv\|_{H^{s,q}(\T^d)} \lesssim  \|u\|_{H^{s,l}(\T^d)}\|v\|_{L^r(\T^d)} +  \|v\|_{H^{s,L}(\T^d)}\|u\|_{L^R(\T^d)}  $ provided $l,L \in (1,\infty)$ and $r,R \in (1,\infty]$ are such that $1/l + 1/r = 1/R +1/L =1/q$.
		\item \label{item:prop_para_3} $\|uv\|_{H^{-s,q}(\T^d)} \lesssim \|u\|_{H^{-s,q}(\T^d)} \|v\|_{L^\infty(\T^d)} + \|v\|_{H^{\tau,\zeta } (\T^d) }  \|u\|_{H^{ - s-\epsilon,q}(\T^d)} $, where $\tau > d/\zeta$, $\zeta \in [q' ,\infty)$ and $\epsilon>0$ depends only on $(d,s,q,\tau,\zeta)$.
	\end{enumerate}
\end{proposition}
Secondly, we also recall a special case of the composition estimate in Triebel--Lizorkin and Besov spaces from \cite[Theorem 1, p.373]{Runst_Sickel}, reminiscent of \eqref{Eq76} above, where we can drop the assumption that $f(0)=0$, which only plays a role when working on unbounded domains.
Moreover, in the following statement, we leave out the immediate part of the local Lipschitz estimate, namely that 
$
\|f(u) - f(v)\|_{L^\infty(\T^d)} \le C_n \| u-v\|_{L^\infty(\T^d)}
$ 
for $f$ and $u,v$ as below.

\begin{proposition}\label{prop:comp}Let $s>0$, $q,p\in (1,\infty)$ and $f\colon \R\to\R$ be smooth. Then the mapping $u \mapsto f(u)$ satisfies the following local Lipschitz estimates on  $L^\infty(\T^d)\cap H^{s,q}(\T^d)$ and   $L^\infty(\T^d)\cap B^{s}_{q,p}(\T^d)$: 
	\begin{enumerate}[(i)]
		\item \label{item:prop_comp_H} $\|f(u) - f(v)\|_{H^{s,q}(\T^d)} \le C_n \bigl(\| u-v\|_{H^{s,q}(\T^d)} + \|u -v\|_{L^\infty(\T^d)} ( \|u \|_{H^{s,q}(\T^d)} +  \|v \|_{H^{s,q}(\T^d)}  ) \bigr)$ for a constant $C_n$ depending on $(d,q,s,n,f)$ but independent of $u,v$ satisfying $\|u\|_{L^\infty(\T^d)}, \|v\|_{L^\infty(\T^d)} \le n$.
		\item \label{item:prop_comp_B} $\|f(u) - f(v)\|_{B^{s}_{q,p}(\T^d)} \le C_n \bigl(\| u-v\|_{B^{s}_{q,p}(\T^d)} + \|u -v\|_{L^\infty(\T^d)} ( \|u \|_{B^{s}_{q,p}(\T^d)} +  \|v \|_{B^{s}_{q,p}(\T^d)}  ) \bigr)$ for a constant $C_n$ depending on $(d,p,q,s,n,f)$ but independent of $u,v$ satisfying $\|u\|_{L^\infty(\T^d)}, \|v\|_{L^\infty(\T^d)} \le n$.
	\end{enumerate}
\end{proposition}
\section{Maximal $L^p$-regularity with time measurable coefficients}
\label{app:MR_Lp_time_measurable}
Here, we provide a relatively self-contained proof of $L^p$-regularity estimates for the following PDE: 
\begin{equation}
\label{eq:parabolic_appendix}
\partial_t u +a(t)\Delta^2 u=f\text{ on }(0,T)\times \T^d , \qquad  u(0)=0\text{ on } \T^d,
\end{equation}
where $T>0$ and $\T^d$ denotes the $d$-dimensional torus. Such estimates were used in Section \ref{Sec_SMR} as a starting point to show stochastic maximal $L^p$-regularity for thin-film type operators, see Lemma \ref{Lemma_SMR_bilplace} there.

In the above, $f$ is a given inhomogeneity which will be specified later, and $a$ satisfies the following.

\begin{assumption}
\label{ass:bounded_appendix}
The mapping $a:\R \to \R$ is measurable and satisfies $\lambda^{-1}\leq a\leq \lambda$ on $\R_+$ for some $\lambda>0$. 
\end{assumption}

In particular, $a$ does not have any smoothness in time, and standard perturbation arguments such as  \cite[Proposition 3.5.6]{pruss2016moving} do not apply.
Even in the presence of weights, $L^p$-estimates for parabolic problems with measurable in time coefficients are by now well-understood, see, e.g., \cite{GV17} and \cite{DK11_solvability,DK_weights}. The aim of this appendix is to provide a self-contained and (to some extent) independent proof of such estimates in the range of interests for stochastic PDEs. In particular, we only consider the case $p> 2$ and time weights of the form $t^\kappa \,\dd t $ for $\kappa\in [0,\frac{p}{2}-1)$. The reader is referred to, e.g., Section \ref{Sec_SMR}, \cite{LoVer}, \cite[Section 3]{AV19_QSEE_1} and \cite[Section 3]{AVSurvey} for details and comments on this limitation. Let us emphasize that the case $p=2$ and $\a=0$ is instead standard, and is a consequence of the energy estimate (see, e.g., \cite[Chapter 4]{LiuRock}). Finally, the case $p<2$ and more general weights can be obtained via duality and extrapolation, but we do not pursue this here. 

Our approach to $L^p$-estimates for \eqref{eq:parabolic_appendix} is based on two tools from real analysis: The Rubio de Francia weighted extrapolation (see, e.g., \cite{Rubio_book}) and an interpolation result due to \cite{Shen_weighted}. The proof of the latter shares some similarity with the well-known Marcinkiewicz interpolation. 
The combination of the above ingredients to prove $L^p$-estimates for \eqref{eq:parabolic_appendix} seems new, and gives a relatively straightforward proof. However, let us mention that the Rubio de Francia extrapolation was already employed in \cite{DK_weights,GV17}. Moreover, to apply the extrapolation in \cite{Shen_weighted}, we exploit local smoothing of the solution to \eqref{eq:parabolic_appendix} with $f=0$ via Caccioppoli inequalities (see Lemma \ref{lem:Caccioppoli} below), and closely related arguments can be found in \cite[Section 4]{DK11_solvability}.

To proceed further, let us recall the Muckenhoupt classes $A_p$ with $p\in (1,\infty)$ that are central in Rubio de Francia's extrapolation. We say that a locally integrable function $w:\R\to (0,\infty)$ belongs to $A_p(\R)$ if  
\begin{equation}
\label{eq:A_p_weights_def}
[w]_{A_p}:=\sup_{I} \Big(\fint_{I} w\,\dd t\Big)\Big(\fint_{I} w^{1-p'}\,\dd t\Big)^{p-1}<\infty,
\end{equation}
where the supremum is taken over all finite intervals  in $\R$, $\fint_I=|I|^{-1}\int_I$ denotes the average and $1/p +1/p' =1$. 
It is well-known that $[t\mapsto |t|^{\kappa}]\in A_p(\R)$ if and only if $\kappa\in (-1,p-1)$ (cf. \cite[Example 7.1.7]{Gafrakos_classical}).

\begin{proposition}[Weighted maximal $L^p$-regularity]
\label{prop:deterministic_SMR}
Let Assumption \ref{ass:bounded_appendix} be satisfied. 
Fix $T<\infty$ and $q,p\in (2,\infty)$. 
Then, for all $w\in A_{p/2}(\R)$ and $f\in L^p(0,T,w;L^q(\T^d))$, there exists a unique solution $u\in W^{1,{p}}(0,T,w;L^q(\T^d))\cap L^{{ p}}(0,T,w;W^{4,q}(\T^d))$ to \eqref{eq:parabolic_appendix} and 
\begin{equation}
\label{eq:maximal_Lp_regularity_measurable_appendix}
\int_{0}^T \|\partial_t u\|_{L^{q}(\T^d)}^p \, w \,\dd t +
\int_{0}^T \| u\|_{W^{4,q}(\T^d)}^p \, w  \,\dd t \leq C \int_0^T \|f\|_{ L^{q}(\T^d)}^p\,w \,\dd t,
\end{equation}
where $C$  depends only on $d,\lambda,q,p, T$ and $[w]_{A_{p/2}}$.
\end{proposition}

As noted above, Proposition \ref{prop:deterministic_SMR} is not new, and more general results can be found in \cite{DK_weights,GV17}. Our goal here is to provide a relatively self-contained proof of \eqref{eq:maximal_Lp_regularity_measurable_appendix} within the range of parameters $p$ and $w$ that are relevant for applications to SPDEs, see Section \ref{Sec_SMR}.
By the method of continuity, as employed twice in Section \ref{Sec_SMR}, it suffices to show the a priori estimate \eqref{eq:maximal_Lp_regularity_measurable_appendix}.
Before proceeding with the proof of the latter, we start with a reduction to the case $p=q$, exploiting the Rubio de Francia extrapolation argument \cite{Rubio_book}: Firstly, we note that due to the initial value in \eqref{eq:parabolic_appendix} being $0$, it suffices for \eqref{eq:maximal_Lp_regularity_measurable_appendix} to show 
\begin{equation}
	\label{eq:maximal_Lp_regularity_measurable_appendix_2}
	\int_{0}^T \|\Delta^2 u\|_{L^{q}(\T^d)}^p \, w \,\dd t  \leq C \int_0^T \|f\|_{ L^{q}(\T^d)}^p\,w \,\dd t,
\end{equation}
with $C$  depending as above only on $d,\lambda,q,p, T$ and $[w]_{A_{p/2}}$.
Suppose furthermore that the above estimate \eqref{eq:maximal_Lp_regularity_measurable_appendix_2} holds with $p=q$ so that there exists a constant $C$ which depends only on $d,\lambda,q,T$ and $[w]_{A_{q/2}}$ such that
\begin{equation}
	\label{eq:estrapolation_1}
	\int_{\R} (\one_{(0,T)}\|\Delta^2 u\|_{L^{q}(\T^d)}^2)^{q/2}\, w\,\dd t 
	\leq C 
	\int_{\R} (\one_{(0,T)}\|f\|_{L^{q}(\T^d)}^2)^{q/2} \,w\,\dd t.
\end{equation}
Rubio de Francia weighted extrapolation yields for all $r\in (1,\infty)$ and $w\in A_r$ (see e.g., \cite[Theorem J.2.1]{Analysis2} or \cite[Section 2]{Rubio_book})  
\begin{equation}
	\label{eq:estrapolation_2}
	\int_{\R} (\one_{(0,T)}\|\Delta^2 u\|_{L^{q}(\T^d)}^2)^{r} \,w\,\dd t 
	\leq C_r
	\int_{\R} (\one_{(0,T)}\|f\|_{L^{q}(\T^d)}^2)^{r} \,w\,\dd t.
\end{equation}
where $C_r$ which depends only on $d,\lambda,q,r$ and $[w]_{A_{r}}$.
Now, the conclusion follows by letting $p=2r\in (2,\infty)$ (the argument \eqref{eq:estrapolation_1}-\eqref{eq:estrapolation_2} is sometimes referred to as rescaling, see \cite[Subsection 3.3]{Rubio_book}).

The remaining part of this section is dedicated to the proof of \eqref{eq:maximal_Lp_regularity_measurable_appendix_2},
for which we also recall the extrapolation result due to \cite{Shen_weighted}. For $r>0$ and $(t,x)\in \R\times \R^d$ we denote by $Q_r(t,x)$ the parabolic cubes
$$
Q_r(t,x)=I_r (t)\times B_{r^{1/4}}(x) ,\ \text{ where }\  I_r (t)=(t-r,t+r)  \ \text{ and }\ B_r(x)=\{y\,:\, |x-y|\leq r\},
$$ 
 adapted to 4th-order operators. 
For simplicity, we often write $Q$ instead of $Q_r(t,x)$ when the dependence on $r,t$, and $x$ is clear or unimportant. Moreover, if $Q=Q_{r}(t,x)$, we let $N Q:=Q_{N r}(t,x)$ for a constant $N$. The same notation is employed for intervals $I$, and for balls we set $NB_r =B_{N^{1/4}r}$.

\begin{theorem}
\label{t:shen_interpolation}
Assume that $\Qstar=\Istar\times \Bstar$ is a parabolic cube in $\R\times \R^d$, and let $F\in L^2(4\Qstar)$. 
Fix $\ppstar\in (2,\infty)$ and $p\in (2,\ppstar)$, a weight $w\in A_{p/2}(\R)$ and  $f\in L^{p}(4 \Qstar)$. Suppose that there exists a constant $\Cstar>0$ such that for each parabolic cube $Q=I\times B\subseteq {2 }\Qstar$ satisfying $|Q|\leq \frac{1}{2}|\Qstar|$, there exist functions $F_Q$ and $R_Q$ such that $|F|\leq |F_Q|+|R_Q|$ a.e.\ on $Q$ and the following estimates hold:
\begin{align*}
\Big(\fint_{2Q} |F_Q|^2\,\dd x \,\dd t \Big)^{1/2}
&\leq \Cstar 
\Big(\fint_{4Q} |f|^2\,\dd x \,\dd t \Big)^{1/2},\\
\Big(\fint_{2Q} |R_Q|^{\ppstar} w \,\dd x \,\dd t\Big)^{1/\ppstar}
&\leq \Cstar \Big[
\Big(\fint_{4Q} |f|^{2} \,\dd x \,\dd t\Big)^{1/2}
+ 
\Big(\fint_{4Q} |F|^{2}\,\dd x \,\dd t \Big)^{1/2}\Big]
\Big(\fint_{I} w  \,\dd t\Big)^{1/\ppstar}.
\end{align*}
Then there exists $C$ depending only on $p,d,\Cstar$ and $[w]_{A_{p/2}}$ such that 
$$
\Big(\fint_{\Qstar} |F|^p\,w\,\dd x \,\dd t \Big)^{1/p}
\leq C 
\Big(\fint_{4 \Qstar} |f|^{p} \,w\,\dd x \,\dd t \Big)^{1/p}
+ C
\Big(\fint_{4\Qstar} |F|^{2} \,\dd x \,\dd t\Big)^{1/2}
\Big(\fint_{\Istar} w\,\dd t \Big)^{1/p}.
$$
\end{theorem} 

The result above is a simplified and adapted version of the parabolic setting of \cite[Theorem 2.1]{Shen_weighted}. While spatial weights can also be included, we omit them here as they are not needed for our purposes.

We now turn to the proof of Proposition \ref{prop:deterministic_SMR}. For the remainder of this section, we fix $T\in (0,\infty)$ and note that $L^p(0,T,w)\embed L^2(0,T)$ for all $w\in A_{p/2}$ due to H\"older's inequality, and therefore standard variational methods ensures the existence of a solution $v\in H^1(0,T+1;L^2(\T^d))\cap L^2(0,T+1;H^4(\T^d))$ to \eqref{eq:parabolic_appendix} with $f$ replaced by $g:=\one_{(0,T)} f$. By extending $v=0$ on $(-1,0)$, it follows that $v$ also satisfies 
\begin{equation}
\label{eq:reduction_whole_line_appendix}
\partial_t v +b(t)\Delta^2 v=g \text{ on }(-1,T+1)\times \T^d,\quad v(-1)=0 \text{ on }\T^d,
\end{equation}
where $b=\lambda^{-1}\one_{(-1,0)}+a\one_{\R_+}$. 
In particular, to prove Proposition \ref{prop:deterministic_SMR}, it suffices to prove the estimate \eqref{eq:maximal_Lp_regularity_measurable_appendix} with $u$ replaced by $v$. Next, for any parabolic cube $Q=Q_{r}(t,x)$ with $r\in (0,\frac{1}{64}]$, we decompose $v=v_Q+ r_Q$ where 
$r_Q:=v-v_Q$, and $v_Q\in H^1(t-{2}r,t+{2}r;L^2(2B))\cap L^2(t-{2}r,t+{2}r;H^4(2B))$ is the solution to 
\begin{equation}
\label{eq:vq_problem_appendix}
\left\{
\begin{aligned}
\partial_t v_Q +b (t)\Delta^2 v_Q&=\one_{2Q} g &\text{ in }&2Q, \\
v_Q&=0 &\text{ on }&\{t-{2}r\}\times 2B,\\
v_Q=\Delta v_Q&=0 &\text{ on }&(t-2r,t+{2}r)\times\partial (2B).
 \end{aligned}
 \right.
\end{equation}
As in Subsection \ref{Sec_SMR}, we identify balls in $\T^d$  with their corresponding counterparts in the cube $[-1,2)^d\subseteq \R^d$, provided the radius of the ball is less than $1/2$.
By integrating by parts we find that there exists $C_{\lambda}>0$ independent of $Q$ such that 
\begin{equation}
\label{eq:estimate_L2_vQ}
\int_{2Q} |{\Delta^2} v_Q|^2\,\dd x \,\dd t 
\leq C_{\lambda}
\int_{2Q} |g|^2\,\dd x \,\dd t .
\end{equation}
To handle the remainder $r_Q$, we exploit regularization effects via Caccioppoli inequalities.

\begin{lemma}[Caccioppoli inequality]
\label{lem:Caccioppoli}
Let Assumption \ref{ass:bounded_appendix} be satisfied, and let ${Q = I\times B}\subseteq (-1,T)\times \R^d$ be a parabolic cube. Set $b=\lambda^{-1}\one_{(-1,0)}+a\one_{\R_+}$.
Let $u\in H^1(2I;L^2(2B))\cap L^2(2I;H^{4}(2B))$ be a solution to  
\begin{equation}
\label{eq:basic_problem_caccioppoli}
\partial_t u +b(t)\Delta^2 u=0 \text{ on }2Q.
\end{equation}
Then, $u\in C^{0}(2I;C^\infty(2B))\cap W^{1,\infty}(2I;C(2B))$ and for each integer $m\geq 0$, then there exists a constant {$C_{d,\lambda,m}$} independent of $Q$ such that
\begin{equation}
\label{eq:estimate_local_solution_smoothing}
\sup_{(t,x)\in Q} |\Delta^m u(t,x)|\leq C_{d,\lambda,m}
\Big(\fint_{2Q} | \Delta^m u|^2\,\dd x \,\dd t\Big)^{1/2}.
\end{equation}
\end{lemma}

\begin{proof}
The method of difference quotients readily implies $u\in C^{0}(2I;C^\infty(2B))$. 
Hence, $u\in W^{1,\infty}(2I;C(2B))$ from the boundedness of $a$ and \eqref{eq:basic_problem_caccioppoli}. It remains to prove the claimed estimate \eqref{eq:estimate_local_solution_smoothing}. From the previous fact, $\Delta^m u$ is still a local solution to \eqref{eq:basic_problem_caccioppoli}, with the same regularity as $u$; it is enough to prove \eqref{eq:estimate_local_solution_smoothing} with $m=0$.
We now divide the proof into two steps. Below, for convenience, we write $Q$ instead of $Q_r(t,x)$.

\emph{Local smoothing.} 
For each $N>1$, there exists a constant $K_N$ independent of $Q$ such that 
\begin{equation}
\label{eq:local_smoothing_caccioppoli}
\int_{Q} |\Delta u|^2\,\dd x \,\dd t
\leq  \frac{K_N}{r}
\int_{NQ} |u|^2\,\dd x \,\dd t.
\end{equation} 
For convenience, we consider $N=2$; the general case is similar. 
Let $\psi$ be a cutoff function such that $\psi=1$ on $Q$ and $\psi=0$ on $(\frac{3}{2}Q)^{{\rm c}}$ such that $|\partial_t \psi|\lesssim r^{-1}$ and $|\partial_{x}^\alpha \psi|\lesssim r^{-|\alpha|/4}$
 for all multi-indices $\alpha=(\alpha_1,\dots,\alpha_d)\in \N_0^d$ with $|\alpha|=\sum_{i=1}^d \alpha_i$. Integrating by parts and using the Cauchy-Schwarz inequality,
\begin{align*}
0=&\frac{1}{4}
\int_{2Q} \partial_t ( u^2\psi^4 )\,\dd x \,\dd t 
=
 \int_{2Q} u^2 \psi^3 \partial_t \psi \,\dd x \,\dd t
-\frac{1}{2}\int_{2Q}  b\psi^4  u \Delta^2 u \,\dd x \,\dd t \\
& 
\leq 
 \int_{2Q} u^2 \psi^3 \partial_t \psi \,\dd x \,\dd t
- \frac{1}{2}\int_{2Q}  b\Delta(\psi^4  u) \Delta^2 u \,\dd x \,\dd t \\
 & 
 \leq -\frac{\lambda}{2} \int_{2Q} \psi^4 |\Delta u|^2\,\dd x \,\dd t
 +C_{\lambda} \int_{2Q}\big[|\psi^2||\nabla \psi|^2 |\nabla u|^2 + (\psi^3|\partial_t \psi|+\psi^2|\Delta \psi|^2 + |\nabla \psi|^4) u^2\big]\,\dd x\,\dd t .
\end{align*}
To conclude, it remains to estimate the terms involving $\nabla u$. Since $|\nabla u|^2=\nabla \cdot (\nabla u \, u)-\Delta u\, u$,
\begin{align*}
&\int_{2Q}|\psi^2||\nabla \psi|^2 |\nabla u|^2\,\dd x \,\dd t 
=-\int_{2Q}\nabla (\psi^2|\nabla \psi|^2) \cdot \nabla u \, u\,\dd x \,\dd t - 
\int_{2Q}|\psi^2||\nabla \psi|^2  \Delta u \,u \,\dd x \,\dd t \\
&\qquad\qquad \leq \int_{2Q}\big( \psi |\nabla\psi ||\nabla \psi|^2 | \nabla u | | u|
+ \psi^2 |\nabla^2 \psi | |\nabla \psi|  | \nabla u | | u|+\frac{\lambda}{4}\psi^4|\Delta u|^2 +C_\lambda  |\nabla \psi|^4   \,u^2\big) \,\dd x \,\dd t \\
&\qquad\qquad\leq \frac{1}{2} 
\int_{2Q}|\psi^2||\nabla \psi|^2 |\nabla u|^2\,\dd x \,\dd t+
C \int_{2Q}( |\nabla \psi|^4 + \psi^2 |\nabla^2 \psi|^2)u^2\,\dd x \,\dd t.
\end{align*}
Hence, the first term on the RHS of the above can be absorbed in the LHS. The claim \eqref{eq:local_smoothing_caccioppoli} follows by collecting the above estimates and the assumed behaviour of the cutoff function $\psi$.

\emph{Conclusion.} 
Let $(s,y)\in \R\times \R^d$ and $r>0$ be such that $Q=Q_r(s,y)$. Note that $v=\Delta^{\ell}u$ is also a solution of \eqref{eq:basic_problem_caccioppoli} and with the same regularity of $u$, for all $\ell\geq 0$. From Step 1 and $\partial_t u=- b(t)\Delta^2 u$, we have 
\begin{equation}
\label{eq:local_smoothing_caccioppoli_iterated}
\int_{Q} |\Delta^{\ell} u|^2\,\dd x \,\dd t
\lesssim_\ell r^{-\ell}
\int_{2Q} | u|^2\,\dd x \,\dd t,\qquad
\int_{Q} |\partial_t \Delta^\ell u|^2\,\dd x \,\dd t
\lesssim_\lambda r^{-2-\ell}
\int_{2Q} | u|^2\,\dd x \,\dd t.
\end{equation}
Next, Sobolev embeddings imply
$$
H^{1}(0,1;H^{2d}(B_1))\embed L^\infty(0,1;H^{2d}(B))\embed L^\infty(0,1;L^\infty(B)),
$$
where $B_1$ is the ball centered at $0$ with radius $1$. 
By applying the above embedding to $(t,x)\mapsto  u(s+rt, y+r^{1/4} x)$, and using $f\mapsto\|\Delta^{d} f\|_{L^2(B_1)}+\| f\|_{L^2(B_1)}$ as a norm on $H^{2d}(B_1)$, we obtain
\begin{align*}
\sup_{(t,x)\in Q} | u(t,x)|^2
&\lesssim_d
r^{2+d} \fint_{Q} |\partial_t\Delta^{d} u|^2\,\dd x\,\dd t 
+r^{2} \fint_{Q} |\partial_t u|^2\,\dd x\,\dd t 
+r^{d} \fint_{Q} |\Delta^d u|^2\,\dd x\,\dd t 
+\fint_{Q} | u|^2\,\dd x\,\dd t \\
&\lesssim_{d,\lambda} 
\fint_{2Q} | u|^2\,\dd x\,\dd t ,
\end{align*}
where the last estimate follows from \eqref{eq:local_smoothing_caccioppoli_iterated}.
\end{proof}

\begin{proof}[Proof of Proposition \ref{prop:deterministic_SMR}]
As discussed below Proposition \ref{prop:deterministic_SMR}, it remains to verify the bound \eqref{eq:maximal_Lp_regularity_measurable_appendix_2} with $p=q$, which we accomplish in the following two steps.

\textit{Localized $L^q$-estimate: There exists a constant $C$ depending only on $d,\lambda, q$ and $[w]_{A_{q/2}}$  such that for each parabolic cube $\Qstar=Q_{r_*}(t_*,x_*)\subseteq (-\frac{1}{4},T+1)\times \T^d$ with $r_*\in (0,\frac{1}{64}]$ we have}
\begin{equation}
\label{eq:Q0_estimate_almost_there}
\Big(\fint_{Q_*} |\Delta^2 v|^q\,w \,\dd x \,\dd t \Big)^{1/q}
\leq C\Big[
\Big(\fint_{4Q_*} |g|^{q} w\,\dd x \,\dd t \Big)^{1/q}
+
\Big(\fint_{4Q_*} |\Delta^2 v|^{2} \,\dd x \,\dd t\Big)^{1/2}
\Big(\fint_{I_*} w\,\dd t \Big)^{1/q}\Big]
\end{equation}
\textit{where $v$ solves \eqref{eq:reduction_whole_line_appendix}.}
Here, we apply Theorem \ref{t:shen_interpolation} with $F=\Delta^2 v$ and using the decomposition $|F|\leq |\Delta^2 v_Q|+|\Delta^2 r_Q|$ for each parabolic cube $Q$, where $v_Q$ solves \eqref{eq:vq_problem_appendix} {and $r_Q = v-v_Q$}. From \eqref{eq:estimate_L2_vQ}, it remains to check the assumption on the remainder. 
Now, for all $\pstar\in (2,\infty)$, we have
\begin{align*}
\Big(\fint_{2Q} |\Delta^2 r_Q|^{\pstar}\, w \,\dd x \,\dd t\Big)^{1/\pstar}
&\stackrel{(i)}{\lesssim}_{(d,q,[w]_{A_{q/2}})}\Big( \sup_{2Q} |\Delta^2 r_Q| \Big)
\Big(\fint_{I} w \,\dd t\Big)^{1/\pstar}\\	
&\stackrel{(ii)}{\lesssim}_{d,\lambda} \Big(\fint_{4Q} |\Delta^2 r_Q|^2\,\dd x \,\dd t\Big)^{1/2} 
\Big(\fint_{I} w \,\dd t \Big)^{1/\pstar}\\	
&\leq \Big[\Big(\fint_{4Q} |\Delta^2v_Q|^2\,\dd x \,\dd t\Big)^{1/2} 
+\Big(\fint_{4Q} |\Delta^2 v|^2\,\dd x \,\dd t\Big)^{1/2}\Big]
\Big(\fint_{I} w \,\dd t \Big)^{1/\pstar}\\	
&\stackrel{\eqref{eq:estimate_L2_vQ}}{\lesssim_{\lambda}} \Big[\Big(\fint_{4Q} |g|^2\,\dd x \,\dd t\Big)^{1/2} 
+\Big(\fint_{4Q} |\Delta^2 v|^2\,\dd x \,\dd t\Big)^{1/2}\Big]
\Big(\fint_{I} w \,\dd t \Big)^{1/\pstar},
\end{align*}
where in $(i)$ we used the doubling property of Muckenhout weights, i.e., $\int_{2I} w \,\dd t\lesssim_{d,q,[w]_{A_{q/2}}} \int_{I} w \,\dd t$ (see e.g., \cite[Proposition 7.1.5(9)]{Gafrakos_classical}) and in $(ii)$ Lemma \ref{lem:Caccioppoli}.
The estimate \eqref{eq:Q0_estimate_almost_there} now follows from  Theorem \ref{t:shen_interpolation}.

\textit{Conclusion.} 
Select a finite collection of points $((t_j,x_j))_{j=1}^N$ such that $(0,T)\times \T^d= \cup_{j=1}^N Q_{r}(t_j,x_j)$ where $r=\frac{1}{64}$. Note that one can select these points such that $N$ depends only on $T$. Then, applying  \eqref{eq:Q0_estimate_almost_there},
\begin{align*}
\Big(\int_{0}^T \int_{\T^d} |\Delta^2 v|^q\,w\,\dd x \,\dd t \Big)^{1/q}
&\leq 
\sum_{j=1}^N 
\Big(\int_{Q_{r}(t_j,x_j)} |\Delta^2 v|^q\,w\,\dd x \,\dd t \Big)^{1/q}\\
&\lesssim N \Big[
\Big(\int_{-1}^{T+1} \int_{\T^d} |g|^{q} \,w\,\dd x \,\dd t \Big)^{1/q}
+
\Big(\int_{-1}^{T+1}\int_{\T^d} |\Delta^2 v|^{2} \,\dd x \,\dd t\Big)^{1/2}
\Big(\int_{0}^T w\,\dd t \Big)^{1/q}\Big].
\end{align*}
Next, as in \eqref{eq:estimate_L2_vQ}, by energy methods and $g=\one_{(0,T)} f$, it follows that 
\begin{align*}
\Big(\int_{0}^{T+1} \int_{\T^d} |\Delta^2 v|^2\,\dd x \,\dd t \Big)^{1/2}
&\lesssim_{\lambda} 
\Big(\int_{0}^{T+1} \int_{\T^d} |g|^2\,\dd x \,\dd t \Big)^{1/2}\\
&\leq \Big(\int_{0}^T w^{-2/(q-2)}\,\dd t \Big)^{(q-2)/{q}} \Big(\int_{0}^{T}\int_{\T^d} |f|^q\,w\,\dd x\,\dd t \Big)^{1/q}\\
&\leq [w]_{A_{q/2}}^{1/q}\Big(\int_{0}^T w\,\dd t\Big)^{-1/q} \Big(\int_0^{T}\int_{ \T^d} |f|^q\,w\,\dd x\,\dd t \Big)^{1/q}.
\end{align*}
Thus, \eqref{eq:maximal_Lp_regularity_measurable_appendix_2} with $p=q$ follows from the above and $g=\one_{(0,T)} f$, $u=v$ on $(0,T)$, $v=0$ on $(-1,0)$.
\end{proof}

\subsubsection*{Acknowledgements}
The authors thank Mark Veraar for the discussions during the initial stage of the project. They are grateful to Manuel V.\ Gnann and Mark  Veraar for carefully reading the manuscript. Finally, the authors thank the anonymous referees for their valuable comments, which significantly improved the clarity and presentation of this work, as well as pointing out the possibility of relaxing the condition \eqref{Eq40} to its current form.

\bigskip

\noindent
\textbf{Data availability.} This manuscript has no associated data.

\medskip

\noindent
\textbf{Declaration -- Conflict of interest.} The authors have no conflict of interest.

\def\polhk#1{\setbox0=\hbox{#1}{\ooalign{\hidewidth
			\lower1.5ex\hbox{`}\hidewidth\crcr\unhbox0}}} \def\cprime{$'$}


\begin{thebibliography}{10}
	
	\bibitem{ALV19}
	A.~Agresti, N.~Lindemulder, and M.C. Veraar.
	\newblock On the trace embedding and its applications to evolution equations.
	\newblock {\em Mathematische Nachrichten}, 296(4):1319--1350, 2023.
	
	\bibitem{AV24_variational}
	A.~Agresti and M.~Veraar.
	\newblock The critical variational setting for stochastic evolution equations.
	\newblock {\em Probab. Theory Related Fields}, 188(3-4):957--1015, 2024.
	
	\bibitem{AV19}
	A.~Agresti and M.C. Veraar.
	\newblock Stability properties of stochastic maximal {$L^p$}-regularity.
	\newblock {\em J. Math. Anal. Appl.}, 482(2):123553, 35, 2020.
	
	\bibitem{AV19_QSEE_1}
	A.~Agresti and M.C. Veraar.
	\newblock Nonlinear parabolic stochastic evolution equations in critical spaces
	{P}art {I}. {S}tochastic maximal regularity and local existence.
	\newblock {\em Nonlinearity}, 35(8):4100, 2022.
	
	\bibitem{AV19_QSEE_2}
	A.~Agresti and M.C. Veraar.
	\newblock Nonlinear parabolic stochastic evolution equations in critical spaces
	part {II}.
	\newblock {\em Journal of Evolution Equations}, 22(2):1--96, 2022.
	
	\bibitem{AV22_localRD}
	A.~Agresti and M.C. Veraar.
	\newblock Reaction-diffusion equations with transport noise and critical
	superlinear diffusion: {L}ocal well-posedness and positivity.
	\newblock {\em J. Differential Equations}, 368:247--300, 2023.
	
	\bibitem{AV21_max_reg_torus}
	A.~Agresti and M.C. Veraar.
	\newblock Stochastic maximal $ {L}^{p}({L}^{q})$-regularity for second order
	systems with periodic boundary conditions.
	\newblock {\em {A}nnales de l'institut {H}enri {P}oincar\'{e} {(B)}
		{P}robability and {S}tatistics}, 60(1):413--430, 2024.
	
	\bibitem{AVSurvey}
	A.~Agresti and M.C. Veraar.
	\newblock Nonlinear {SPDE}s and maximal regularity: an extended survey.
	\newblock {\em NoDEA Nonlinear Differential Equations Appl.}, 32(6):Paper No.
	123, 150, 2025.
	
	\bibitem{Beretta_Bertsch_DalPasso_95}
	E.~Beretta, M.~Bertsch, and R.~Dal~Passo.
	\newblock Nonnegative solutions of a fourth-order nonlinear degenerate
	parabolic equation.
	\newblock {\em Arch. Rational Mech. Anal.}, 129:175--200, 1995.
	
	\bibitem{BeLo}
	J.~Bergh and J.~L{\"o}fstr{\"o}m.
	\newblock {\em Interpolation spaces. {A}n introduction}.
	\newblock Springer-Verlag, Berlin, 1976.
	\newblock Grundlehren der Mathematischen Wissenschaften, No. 223.
	
	\bibitem{BERNIS1990}
	F.~Bernis and A.~Friedman.
	\newblock Higher order nonlinear degenerate parabolic equations.
	\newblock {\em Journal of Differential Equations}, 83:179--206, 1990.
	
	\bibitem{Breit_Feireisl_Hofmanova}
	D.~Breit, E.~Feireisl, and M.~Hofmanov\'{a}.
	\newblock {\em Stochastically forced compressible fluid flows}, volume~3 of
	{\em De Gruyter Series in Applied and Numerical Mathematics}.
	\newblock De Gruyter, Berlin, 2018.
	
	\bibitem{Brz2}
	Z.~Brze{\'z}niak.
	\newblock On stochastic convolution in {B}anach spaces and applications.
	\newblock {\em Stochastics Stochastics Rep.}, 61(3-4):245--295, 1997.
	
	\bibitem{Rubio_book}
	D.V. Cruz-Uribe, J.M. Martell, and C.~P\'{e}rez.
	\newblock {\em Weights, extrapolation and the theory of {R}ubio de {F}rancia},
	volume 215 of {\em Operator Theory: Advances and Applications}.
	\newblock Birkh\"{a}user/Springer Basel AG, Basel, 2011.
	
	\bibitem{DPZ}
	G.~Da~Prato and J.~Zabczyk.
	\newblock {\em Stochastic equations in infinite dimensions}, volume~44 of {\em
		Encyclopedia of Mathematics and its Applications}.
	\newblock Cambridge University Press, Cambridge, 1992.
	
	\bibitem{dareiotis2021nonnegative}
	K.~Dareiotis, B.~Gess, M.V. Gnann, and G.~Gr{\"u}n.
	\newblock Non-negative martingale solutions to the stochastic thin-film
	equation with nonlinear gradient noise.
	\newblock {\em Archive for Rational Mechanics and Analysis}, 242:179--234,
	2021.
	
	\bibitem{dareiotis2023solutions}
	K.~Dareiotis, Benjamin Gess, M.V. Gnann, and M.~Sauerbrey.
	\newblock Solutions to the stochastic thin-film equation for initial values
	with non-full support.
	\newblock {\em arXiv preprint arXiv:2305.06017}, 2023.
	\newblock To appear in {T}ransactions of the {AMS}.
	
	\bibitem{DMES2005}
	B.~Davidovitch, E.~Moro, and H.A. Stone.
	\newblock Spreading of viscous fluid drops on a solid substrate assisted by
	thermal fluctuations.
	\newblock {\em Phys. Rev. Lett.}, 95:244505, Dec 2005.
	
	\bibitem{DHV_16}
	A.~Debussche, M.~Hofmanov\'{a}, and J.~Vovelle.
	\newblock Degenerate parabolic stochastic partial differential equations:
	quasilinear case.
	\newblock {\em Ann. Probab.}, 44(3):1916--1955, 2016.
	
	\bibitem{DK11_solvability}
	H.~Dong and D.~Kim.
	\newblock On the {$L_p$}-solvability of higher order parabolic and elliptic
	systems with {BMO} coefficients.
	\newblock {\em Arch. Ration. Mech. Anal.}, 199(3):889--941, 2011.
	
	\bibitem{DK_weights}
	H.~Dong and D.~Kim.
	\newblock On {$L_p$}-estimates for elliptic and parabolic equations with
	{$A_p$} weights.
	\newblock {\em Trans. Amer. Math. Soc.}, 370(7):5081--5130, 2018.
	
	\bibitem{Gvalani_19}
	M.A. Dur\'{a}n-Olivencia, R.S. Gvalani, S.~Kalliadasis, and G.A. Pavliotis.
	\newblock Instability, rupture and fluctuations in thin liquid films: Theory
	and computations.
	\newblock {\em Journal of Statistical Physics}, 174:579--604, 2019.
	
	\bibitem{fischer_gruen_2018}
	J.~Fischer and G.~Grün.
	\newblock {Existence} of positive solutions to stochastic thin-film equations.
	\newblock {\em SIAM Journal on Mathematical Analysis}, 50:411--455, 2018.
	
	\bibitem{GV17}
	C.~Gallarati and M.C. Veraar.
	\newblock Maximal regularity for non-autonomous equations with measurable
	dependence on time.
	\newblock {\em Potential Analysis}, 46:527--567, 2017.
	
	\bibitem{GessGann2020}
	B.~Gess and M.V. Gnann.
	\newblock The stochastic thin-film equation: Existence of nonnegative
	martingale solutions.
	\newblock {\em Stochastic Processes and their Applications},
	130(12):7260--7302, 2020.
	
	\bibitem{GGKO21}
	B.~Gess, R.~Gvalani, F.~Kunick, and F.~Otto.
	\newblock Thermodynamically consistent and positivity-preserving discretization
	of the thin-film equation with thermal noise.
	\newblock {\em Mathematics of Computation}, 92:1931--1976, 2023.
	
	\bibitem{gnann_wisse}
	M.V. Gnann and A.C. Wisse.
	\newblock Classical solutions to the thin-film equation with general mobility
	in the perfect-wetting regime.
	\newblock {\em J. Funct. Anal.}, 289(8):Paper No. 110941, 61, 2025.
	
	\bibitem{Gafrakos_classical}
	L.~Grafakos.
	\newblock {\em Classical {F}ourier analysis}, volume 249 of {\em Graduate Texts
		in Mathematics}.
	\newblock Springer, New York, third edition, 2014.
	
	\bibitem{KleinGruen22}
	G.~Gr{\"u}n and L.~Klein.
	\newblock Zero-contact angle solutions to stochastic thin-film equations.
	\newblock {\em Journal of Evolution Equations}, 22:1--37, 2022.
	
	\bibitem{grun2000nonnegativity}
	G.~Gr{\"u}n and M.~Rumpf.
	\newblock Nonnegativity preserving convergent schemes for the thin film
	equation.
	\newblock {\em Numerische Mathematik}, 87:113--152, 2000.
	
	\bibitem{GruenMeckeRauscher2006}
	G.~Grün, K.~Mecke, and M.~Rauscher.
	\newblock Thin-film flow influenced by thermal noise.
	\newblock {\em Journal of Statistical Physics}, 122:1261--1291, 01 2006.
	
	\bibitem{gvalani2023stochastic}
	R.S. Gvalani and M.~Tempelmayr.
	\newblock Stochastic estimates for the thin-film equation with thermal noise.
	\newblock {\em arXiv preprint arXiv:2309.15829}, 2023.
	
	\bibitem{Analysis1}
	T.P. Hyt\"onen, J.M.A.M.~van Neerven, M.C. Veraar, and L.~Weis.
	\newblock {\em Analysis in {B}anach spaces. {V}ol. {I}. {M}artingales and
		{L}ittlewood-{P}aley theory}, volume~63 of {\em Ergebnisse der Mathematik und
		ihrer Grenzgebiete. 3. Folge.}
	\newblock Springer, 2016.
	
	\bibitem{Analysis2}
	T.P. Hyt\"onen, J.M.A.M.~van Neerven, M.C. Veraar, and L.~Weis.
	\newblock {\em Analysis in {B}anach spaces. {V}ol. {II}. {P}robabilistic
		{M}ethods and {O}perator {T}heory.}, volume~67 of {\em Ergebnisse der
		Mathematik und ihrer Grenzgebiete. 3. Folge.}
	\newblock Springer, 2017.
	
	\bibitem{Analysis3}
	T.P. Hyt\"{o}nen, J.M.A.M. van Neerven, M.C. Veraar, and L.~Weis.
	\newblock {\em Analysis in {B}anach spaces. {V}ol. {III}. {H}armonic analysis
		and spectral theory}, volume~76 of {\em Ergebnisse der Mathematik und ihrer
		Grenzgebiete. 3. Folge.}
	\newblock Springer, 2023.
	
	\bibitem{kapustyan2023film}
	O.~Kapustyan, O.~Martynyuk, O.~Misiats, and O.~Stanzhytskyi.
	\newblock Thin film equations with nonlinear deterministic and stochastic
	perturbations.
	\newblock {\em Nonlinear Analysis}, 250:113646, 2025.
	
	\bibitem{KS98}
	I.~Karatzas and S.E. Shreve.
	\newblock {\em Methods of mathematical finance}, volume~39 of {\em Applications
		of Mathematics}.
	\newblock Springer-Verlag, New York, 1998.
	
	\bibitem{LiuRock}
	W.~Liu and M.~R{\"o}ckner.
	\newblock {\em Stochastic partial differential equations: an introduction}.
	\newblock Universitext. Springer, Cham, 2015.
	
	\bibitem{LoVer}
	E.~Lorist and M.C. Veraar.
	\newblock Singular stochastic integral operators.
	\newblock {\em Analysis \& PDE}, 14(5):1443--1507, 2021.
	
	\bibitem{Luninterp}
	A.~Lunardi.
	\newblock {\em Interpolation {T}heory}.
	\newblock Appunti. Scuola Normale Superiore Pisa, 1999.
	
	\bibitem{metzger2022existence}
	Stefan Metzger and Günther Grün.
	\newblock Existence of nonnegative solutions to stochastic thin-film equations
	in two space dimensions.
	\newblock {\em Interfaces and Free Boundaries}, 24:307--387, 2022.
	
	\bibitem{MV12}
	M.~Meyries and M.C. Veraar.
	\newblock Sharp embedding results for spaces of smooth functions with power
	weights.
	\newblock {\em Studia Math.}, 208(3):257--293, 2012.
	
	\bibitem{MV14}
	M.~Meyries and M.C. Veraar.
	\newblock Traces and embeddings of anisotropic function spaces.
	\newblock {\em Math. Ann.}, 360(3-4):571--606, 2014.
	
	\bibitem{NVW1}
	J.M.A.M.~van Neerven, M.C. Veraar, and L.W. Weis.
	\newblock Stochastic integration in {UMD} {B}anach spaces.
	\newblock {\em Ann. Probab.}, 35(4):1438--1478, 2007.
	
	\bibitem{MaximalLpregularity}
	J.M.A.M.~van Neerven, M.C. Veraar, and L.W. Weis.
	\newblock Stochastic maximal {$L^p$}-regularity.
	\newblock {\em Ann. Probab.}, 40(2):788--812, 2012.
	
	\bibitem{NVW13}
	J.M.A.M.~van Neerven, M.C. Veraar, and L.W. Weis.
	\newblock Stochastic integration in {B}anach spaces---a survey.
	\newblock In {\em Stochastic analysis: a series of lectures}, volume~68 of {\em
		Progr. Probab.}, pages 297--332. Birkh\"{a}user/Springer, Basel, 2015.
	
	\bibitem{NWalpha}
	J.M.A.M.~van Neerven and L.W Weis.
	\newblock Stochastic integration of operator-valued functions with respect to
	{B}anach space-valued {B}rownian motion.
	\newblock {\em Potential Anal.}, 29(1):65--88, 2008.
	
	\bibitem{ODB97}
	A.~Oron, S.H. Davis, and S.G. Bankoff.
	\newblock Long-scale evolution of thin liquid films.
	\newblock {\em Reviews of {M}odern {P}hysics}, 69(3):931, 1997.
	
	\bibitem{VP18}
	P.~Portal and M.C. Veraar.
	\newblock Stochastic maximal regularity for rough time-dependent problems.
	\newblock {\em Stoch. Partial Differ. Equ. Anal. Comput.}, 7(4):541--597, 2019.
	
	\bibitem{pruss2016moving}
	J.~Pr\"{u}ss and G.~Simonett.
	\newblock {\em Moving interfaces and quasilinear parabolic evolution
		equations}, volume 105 of {\em Monographs in Mathematics}.
	\newblock Birkh\"{a}user/Springer, 2016.
	
	\bibitem{Runst_Sickel}
	T.~Runst and W.~Sickel.
	\newblock {\em Sobolev spaces of fractional order, {N}emytskij operators, and
		nonlinear partial differential equations}, volume~3 of {\em De Gruyter Series
		in Nonlinear Analysis and Applications}.
	\newblock Walter de Gruyter \& Co., Berlin, 1996.
	
	\bibitem{Sauerbrey_2021}
	M.~Sauerbrey.
	\newblock Martingale solutions to the stochastic thin-film equation in two
	dimensions.
	\newblock {\em {A}nnales de l'institut {H}enri {P}oincar\'{e} {(B)}
		{P}robability and {S}tatistics}, 60(1):373--412, 2024.
	
	\bibitem{sauerbrey2023solutions}
	M.~Sauerbrey.
	\newblock Solutions to the stochastic thin-film equation for the range of
	mobility exponents {$n\in (2,3)$}.
	\newblock {\em Stoch. Partial Differ. Equ. Anal. Comput.}, 13(3):1502--1557,
	2025.
	
	\bibitem{schmeisser1987topics}
	H.J. Schmeisser and H.~Triebel.
	\newblock {\em Topics in Fourier Analysis and Function Spaces}.
	\newblock Wiley, 1987.
	
	\bibitem{Shen_weighted}
	Z.~Shen.
	\newblock Weighted {$L^2$} estimates for elliptic homogenization in {L}ipschitz
	domains.
	\newblock {\em J. Geom. Anal.}, 33(1):Paper No. 3, 33, 2023.
	
	\bibitem{Stein70}
	E.M. Stein.
	\newblock {\em Singular integrals and differentiability properties of
		functions}.
	\newblock Princeton Mathematical Series, No. 30. Princeton University Press,
	Princeton, N.J., 1970.
	
	\bibitem{ToolsPDEsTaylor}
	M.E. Taylor.
	\newblock {\em Tools for {PDE}}, volume~81 of {\em Mathematical Surveys and
		Monographs}.
	\newblock American Mathematical Society, Providence, RI, 2000.
	\newblock Pseudodifferential operators, paradifferential operators, and layer
	potentials.
	
	\bibitem{Bertozzi_numerics}
	L.~Zhornitskaya and A.~L. Bertozzi.
	\newblock Positivity-preserving numerical schemes for lubrication-type
	equations.
	\newblock {\em SIAM J. Numer. Anal.}, 37(2):523--555, 2000.
	
\end{thebibliography}
\end{document}